\documentclass[11pt]{article}
\usepackage{amsfonts}
\usepackage{amsfonts,latexsym}
\usepackage{amsmath}
\usepackage{amsthm}
\usepackage{amssymb}
\usepackage{comment}
\usepackage{mathrsfs}
\usepackage{booktabs}
\usepackage{color}
\usepackage[]{cite}
\usepackage{dsfont}
\usepackage{graphicx}
\usepackage[toc]{appendix}
\usepackage{mathtools}
\allowdisplaybreaks[4]
\usepackage{indentfirst}
\usepackage{geometry}
\newtheorem{theorem}{Theorem}[section]
\newtheorem{proposition}[theorem]{Proposition}
\newtheorem{definition}[theorem]{Definition}

\newtheorem{lemma}[theorem]{Lemma}
\newtheorem{remark}[theorem]{Remark}
\newtheorem{corollary}[theorem]{Corollary}

\newcommand{\Hh}{\mathcal H}
\newcommand{\norm}[1]{\left\lVert #1\right\rVert}

\newcommand{\ip}[2]{\left\langle #1,#2\right\rangle}
\newcommand{\pvar}{p\text{-var}}
\newcommand{\qvar}{q\text{-var}}
\newcommand{\dist}{\operatorname{dist}}

\newcommand{\N}{\mathbb N}
\newcommand{\Z}{\mathbb Z}
\newcommand{\tnorm}[1]{%
  \left\lvert\!\left\lvert\!\left\lvert
  #1
  \right\rvert\!\right\rvert\!\right\rvert
}

\renewcommand{\pvar}{p\text{-var}}
\renewcommand{\qvar}{q\text{-var}}

\begin{document}

\title{Random Pullback Attractors and Numerical Approximation
for Nonlocal Lattice Delay Equations Driven by Rough Noise\footnote{This work was supported by the National Natural Science Foundation of China under grant 12371198}}
	\author{Lijuan Zhang$^\text{a}$, Jiangwei Zhang$^\text{b}$, Jianhua Huang$^\text{c}$\thanks{Corresponding author:jhhuang32@nudt.edu.cn}\\
	{\small  \it $^\text{a}$ School of Mathematics, Nanjing Audit University, Nanjing, 211815, P.R.China}\\
    {\small  \it $^\text{b}$ Institute of Applied Physics and Computational Mathematics, Beijing, 100088, P. R. China}\\
	{\small  \it $^\text{c}$ College of Science, National University of Defense Technology, Changsha, 410073, P.R.China}
}
\date{}


\renewcommand{\theequation}{\arabic{section}.\arabic{equation}}
\numberwithin{equation}{section}

\maketitle

\begin{abstract}
We investigate the pathwise asymptotic dynamics of a nonlocal lattice equation
with fixed delay on $\ell^2(\mathbb Z)$ driven by an infinite-dimensional
stationary weakly geometric rough path.
Using a rough-flow
Doss-Sussmann transformation and the method of steps, we establish
global well-posedness in the controlled-path framework and construct a
continuous random dynamical system. We further construct a tempered
pullback absorbing family and use a spatially localized
Lyapunov-Krasovskii functional, together with a sewing estimate for
the spatial tails of the driving rough path, to prove that the lattice
tails vanish uniformly.
The existence of a unique random pullback attractor is then
established by combining this lattice-tail estimate with an
Arzel\`a--Ascoli compactness argument.
Finally, for the delay-compatible compensated Euler schemes
and finite-dimensional lattice truncations,
we establish uniform
absorbing and lattice-tail estimates, which yield the existence of the approximate
random pullback attractors and their joint upper
semicontinuity, without a coupling condition between
the time step and the spatial truncation level.

\end{abstract}



\section{Introduction}

We consider the following nonlocal lattice equation with fixed delay
$\tau>0$:
\begin{equation}\label{eq:main}
\begin{aligned}
du_i(t)
={}&
\left[
\nu\sum_{j\in\mathbb Z}w_{ij}\bigl(u_j(t)-u_i(t)\bigr)
-\lambda u_i(t)
+f_i\bigl(u_i(t),u_i(t-\tau)\bigr)
\right]dt\\
&+h_i\bigl(u_i(t)\bigr)\,dx^i(t),
\qquad i\in\mathbb Z,\quad t\geq0.
\end{aligned}
\end{equation}
The initial history is prescribed by
$u_i(s)=\xi_i(s)$ for $s\in[-\tau,0]$ and
$i\in\mathbb Z$, where \(\xi=(\xi_i)_{i\in\mathbb Z}\in
\Hh:=C([-\tau,0];\ell^2(\mathbb Z))\).
Here $\nu$ and $\lambda$ denote the coupling and damping
parameters, respectively, and the symmetric weights $w_{ij}$
describe the nonlocal interactions between lattice sites.
The driving signal $\mathbf x=(x,\mathbb X)$ is an
$\ell^2(\mathbb Z)$-valued weakly geometric $p$-rough path
with stationary increments, where $2<p<3$.
The rough term is interpreted as a controlled rough integral.

Delay lattice equations provide a standard mathematical model for processes involving time delays in spatially discrete systems, such as signal transmission between coupled oscillators \cite{Zanette2000}, population maturation \cite{MaZou2005}, recovery of excitable cells \cite{HanKloeden2016}, and gene expression within cells and protein exchange between cells \cite{Ribeiro2007}.
There exist numerous studies on delay lattice equations, particularly on their well-posedness and dynamics.
Zhao and Zhou gave necessary and
sufficient conditions for the existence of compact uniform
attractors for dissipative delayed lattice systems and
applied their results to retarded lattice Zakharov equations
\cite{ZhaoZhou2008}. Caraballo, Morillas and Valero proved
the existence of solutions to an abstract delay equation
in a Banach space and established a compact global
attractor for a retarded lattice system
\cite{CaraballoMorillasValero2014}. For a nonautonomous
lattice system with switching effects and delayed recovery,
Han and Kloeden established the existence of weak and
strong solutions and a pullback attractor
\cite{HanKloeden2016}. Wang, Kloeden and Yang proved global
existence and uniqueness of solutions and obtained a global
attractor for a delayed neural-field lattice model
\cite{WangKloedenYang2020}. Li and Wang proved global well-posedness and obtained a
pullback attractor and invariant measures for first-order
nonlocal lattice systems with variable and distributed
delays \cite{LiWang2024}.
In the stochastic setting, Yan,
Li and Ji established a random attractor for first-order
retarded lattice systems driven by additive Brownian noise \cite{YanLiJi2010}.

The well-posedness and long-time behavior of delay equations
under irregular forcing have attracted increasing attention.
Qu and Gao established existence and stability results for
delay rough evolution equations and studied the limit as
the delay vanishes \cite{QuGao2026}. For stochastic delay
evolution equations driven by fractional Brownian motion
with Hurst parameter greater than \(1/2\), Caraballo et al. proved existence, uniqueness
and exponential asymptotic behavior of mild solutions
\cite{CaraballoGarridoAtienzaTaniguchi2011}. In related
spatially discrete and rough evolution settings, Zeng et al. constructed pathwise solutions for rough
stochastic lattice systems and studied random attractors
for rough stochastic partial differential equations
\cite{ZengLinHuangYang2020,YangLinZeng2023}.
Blessing Neam\c{t}u et al. established the existence and
regularity of random pullback attractors for stochastic
evolution equations driven by rough noise
\cite{BlessingNeamtuSeitz2025}.
For related results on rough and Young equations
and attractor approximation, see \cite{NeuenkirchNourdinTindel2008,
CaoGao2025,
Ananova2023,Cong2025,DucKloeden2023,
HesseNeamtu2020,Cong2022, Zhou2004}.
However, to the best of our knowledge, the long-time
dynamics of delay lattice equations driven by low-regularity
rough noises remain unexplored. In particular, even for
fractional Brownian motion with low Hurst regularity, the
existence and approximation of random attractors for delay
lattice systems have not been investigated.

In this paper, we study the well-posedness and pullback
dynamics of the nonlocal lattice equation \eqref{eq:main}
with fixed delay driven by an infinite-dimensional stationary
weakly geometric rough path with low regularity, together with the approximation
of its random pullback attractor.
The presence of delay leads us to study the dynamics
on the history space $C([-\tau,0];\ell^2(\mathbb Z))$
and introduces additional delayed terms into the energy
estimates. Under the dissipativity condition
$\lambda>L_0+L_1$, these terms are controlled by a
Lyapunov--Krasovskii functional.
Moreover, the rough noise introduces a rough integral
term into the lattice-tail energy estimate. Estimating
this term directly using the rough-path norm does not
show that it vanishes as the spatial cutoff tends to
infinity. To obtain the uniform lattice-tail decay
required for pullback asymptotic compactness,
we establish a
projected sewing estimate and tempered pullback bounds to
show that the rough integral contribution to the lattice-tail
energy vanishes as the spatial cutoff increases.
For delay-compatible compensated Euler schemes and
finite-dimensional lattice truncations, the main difficulty
is to establish absorbing bounds and lattice-tail estimates
uniform in both approximation parameters.
We overcome this difficulty by using discrete
Lyapunov--Krasovskii estimates, a rough-flow
Doss--Sussmann transformation, and spatially localized
projected sewing estimates.
This work consists of three parts:
\begin{enumerate}
\item[\textbf{(i)}] The first part of this work focuses on the global well-posedness for the rough delay lattice equation \eqref{eq:main} in the controlled-path framework. Using a local rough-flow Doss--Sussmann transformation (Lemma~3.2), we transform the rough equation into an absolutely continuous system. Combining this system with the method of steps for the delay term, we obtain global controlled solutions and derive uniform controlled-path estimates (Lemma~3.4 and Corollary~3.5). In Theorem~3.6, we prove global existence, uniqueness, and continuous dependence on the initial history and the driving rough path. Finally, Theorem~3.7 shows that the solution map generates a continuous cocycle on the history space \(C ([-\tau,0];\ell^2(\mathbb Z))\).
\item[\textbf{(ii)}] The second part is devoted to the existence of a unique random pullback attractor for the rough delay lattice equation \eqref{eq:main}.
We first use a Lyapunov--Krasovskii functional to establish
a decay estimate for the deterministic delay equation. By
comparing the solutions of the rough and deterministic
equations through a local rough-flow Doss--Sussmann
transformation, we establish a dissipative estimate for the
cocycle on the history space (Lemmas~4.1--4.2 and
Proposition~4.3). This allows us to construct a tempered
pullback absorbing family (Theorem~4.5 and Lemma~4.6).
To prove pullback asymptotic compactness, we use a
spatially localized Lyapunov--Krasovskii functional and
projected sewing estimates to establish uniform decay of
the lattice tails along pullback trajectories
(Theorem~4.11). Combining this tail estimate with
equicontinuity of the history segments, we prove that
the rough delay lattice equation \eqref{eq:main} admits a unique random
pullback attractor (Theorems~4.13--4.14).

\item[\textbf{(iii)}]
The third part concerns approximations of the random pullback
attractor under time discretisation and spatial truncation.
For the delay-compatible compensated Euler scheme, we prove
that its solutions converge to those of the original equation
\eqref{eq:main} on finite time intervals, uniformly over
compact sets of initial histories (Proposition~5.7).
Using a discrete Lyapunov--Krasovskii estimate (Lemma~5.8)
and a local rough-flow Doss--Sussmann transformation, we
compare the compensated Euler solution with the deterministic
delay Euler solution (Lemma~5.9) and construct a discrete
pullback absorbing family (Proposition~5.10). Using a projected sewing estimate (Lemma~5.11), we then establish the uniform lattice-tail decay for pullback
trajectories of the compensated Euler scheme
(Theorem~5.13). This gives a discrete pullback
attractor (Theorem~5.14) and its upper semicontinuity as
the time step tends to zero (Theorem~5.16). For the spatially
truncated equation, we prove convergence of its solutions
to those of \eqref{eq:main} on finite time intervals
(Proposition~5.17). Finally, upper semicontinuity of the
attractors under spatial truncation and joint time and
spatial approximation follows (Theorem~5.19 and
Corollary~5.20).

 \end{enumerate}

The paper is organized as follows. Section~2 introduces
the rough-path framework, the nonlocal operator, and the
standing assumptions. Section~3 establishes global
well-posedness and the continuous cocycle property.
Section~4 constructs a tempered pullback absorbing family
and proves the existence of a unique random pullback attractor.
In Section~5, we study compensated Euler schemes and
finite-dimensional lattice truncations. We prove convergence
of their solutions on finite time intervals and establish
the existence and upper semicontinuity of their random
pullback attractors under time discretisation, spatial
truncation, and joint approximation.

\section{Preliminaries}

Let \(U:=\ell^2(\mathbb Z)\) and let \(V\) be a Banach
space. Denote by \(U\widehat\otimes_\pi U\) the completed
projective tensor product. Every bounded bilinear map
\(B:U\times U\to V\) admits a unique bounded linear
extension
$\widetilde B:U\widehat\otimes_\pi U\to V$,
such that
$\widetilde B(u\otimes v)=B(u,v)$ and \(\|\widetilde B\|=\|B\|\).
By a slight abuse of notation, we continue to denote this
extension by $B$.

For $r\geq 1$, denote by
$C^{r\text{-var}}([a,b];V)\subset C([a,b];V)$ the space of all
continuous paths $y:[a,b]\to V$ of finite $r$-variation
$\tnorm{y}_{r\text{-var},[a,b]}
:=
\left(
\sup_{\Pi}
\sum_{i=1}^{n}
\|y_{t_i,t_{i+1}}\|^r
\right)^{1/r}
<\infty,$
where the supremum is taken over the whole class of finite
partition of $[a,b]$. Then $C^{r\text{-var}}([a,b];V)$ with the equipped $r$-var norm
$\|y\|_{r\text{-var},[a,b]}:=\|y_{a}\|+
\tnorm{y}_{r\text{-var},[a,b]}$ is a Banach space.

Fix $p\in(2,3)$ and set \(q:=p/2\). A pair $\mathbf x=(x,\mathbb X)$,
with $x\in C^{p\text{-var}}([a,b];U)$ and $\mathbb X\in
C_2^{q\text{-var}}
\bigl([a,b];U\widehat{\otimes}_{\pi}U\bigr)$, is called a $p$-rough path if Chen's relation
\begin{align}\label{eq2.1}
\mathbb X_{s,t}
-
\mathbb X_{s,r}
-
\mathbb X_{r,t}
=
x_{s,r}\otimes x_{r,t},
\qquad
a\leq s\leq r\leq t\leq b,
\end{align}
holds. Here,
$C_2^{q\text{-var}}
([a,b];U\widehat{\otimes}_{\pi}U)$
denotes the space of continuous two-index maps
\[
\mathbb X:
\Delta_{[a,b]}
\longrightarrow
U\widehat{\otimes}_{\pi}U,
\qquad
\Delta_{[a,b]}
:=
\{(s,t):a\leq s\leq t\leq b\},
\]
such that $\mathbb X_{t,t}=0$ for every $t\in[a,b]$, and
$\tnorm{\mathbb X}_{q\text{-var},[a,b]}<\infty$.
The bracket of $\mathbf x$ is defined by
\begin{align}\label{eq2.2}
[\mathbf x]_{s,t}
:=
x_{s,t}\otimes x_{s,t}
-
2\operatorname{Sym}\mathbb X_{s,t}
\in U\widehat{\otimes}_{\pi}U.
\end{align}
The rough path $\mathbf x$ is called weakly geometric if $\operatorname{Sym}\mathbb X_{s,t}=\frac12x_{s,t}^{\otimes2}$, or, equivalently, $[\mathbf x]_{s,t}=0$.
On an interval $I\subset[a,b]$, we use the inhomogeneous
rough-path seminorm
\begin{align}\label{eq2.3}
\tnorm{\mathbf x}_{p\text{-var},I}
:=
\tnorm{x}_{p\text{-var},I}
+
\tnorm{\mathbb X}_{q\text{-var},I}^{1/2},
\qquad q=\frac{p}{2}.
\end{align}
For \(a\leq s\leq t\leq b\), define the continuous
control \(\omega_{\mathbf x}\) by
\begin{equation}\label{eq2.4}
 \omega_{\mathbf x}(s,t)
 :=
 (t-s)
 +
 \tnorm{x}_{p\text{-var},[s,t]}^p
 +
 \tnorm{\mathbb X}_{q\text{-var},[s,t]}^q.
\end{equation}
All greedy partitions used below are constructed with
respect to \(\omega_{\mathbf x}\).

We write \(\mathcal L(U,V)\) for the space of bounded
linear operators from \(U\) to \(V\), equipped with the
operator norm. Fix \(x\in C^{p\text{-var}}([a,b];U)\).
A pair
$(y,y')
 \in
 C([a,b];V)
 \times
 C^{p\text{-var}}
 \bigl([a,b];\mathcal L(U,V)\bigr)$
is said to be controlled by \(x\) if
\begin{equation}\label{eq2.5}
 R^{y}_{s,t}
 :=
 y_{s,t}-y'(s)x_{s,t},
 \qquad a\leq s\leq t\leq b,
\end{equation}
belongs to \(C_2^{q\text{-var}}([a,b];V)\).
The path $y'$ is called a Gubinelli derivative of $y$
with respect to $x$, and
$R^y$ is called the controlled remainder. Let \(\mathscr D_x^p([a,b];V)\) denote the space of all pairs $(y,y')$ controlled by $x$. Equipped with the norm
\begin{equation}\label{eq2.7}
 \norm{(y,y')}_{\mathscr D_x^p([a,b];V)}
 :=
 \norm{y}_{\infty,[a,b]}
 +
 \norm{y'}_{\infty,[a,b]}
 +
 \tnorm{y'}_{p\text{-var},[a,b]}
 +
 \tnorm{R^{y;x}}_{q\text{-var},[a,b]},
\end{equation}
the space $\mathscr D_x^p([a,b];V)$ is a Banach space. For $R>0$, define the closed ball
\begin{align}\label{eq2.8}
B_R^{\mathscr D_x^p([a,b];V)}
:=
\left\{
(y,y')\in\mathscr D_x^p([a,b];V):
\|(y,y')\|_{\mathscr D_x^p([a,b];V)}
\leq R
\right\}.
\end{align}

Let \(Y:[a,b]\to\mathcal L(U,V)\) be controlled by \(x\),
with Gubinelli derivative \(Y'\). Using the canonical
isometric identification
$\mathcal L\bigl(U,\mathcal L(U,V)\bigr)
 \cong
 \mathcal L(U\widehat\otimes_\pi U,V)$,
we regard
$Y'
 \in
 C^{p\text{-var}}
 \bigl(
 [a,b];
 \mathcal L(U\widehat\otimes_\pi U,V)
 \bigr)$.
Accordingly,
$R^Y_{s,t}
 :=
 Y_{s,t}
 -
 Y'(s)(x_{s,t}\otimes\cdot)
 \in\mathcal L(U,V)$
belongs to
\(C_2^{q\text{-var}}([a,b];\mathcal L(U,V))\).
By the sewing lemma, the rough integral of \(Y\) against
\(\mathbf x=(x,\mathbb X)\) is well defined by
\begin{equation}\label{eq2.11}
 \int_a^bY(r)\,dx(r)
 :=
 \lim_{|\Pi|\to0}
 \sum_{[s,t]\in\Pi}
 \left(
 Y(s)x_{s,t}
 +
 Y'(s)\mathbb X_{s,t}
 \right).
\end{equation}
Moreover, there exists \(c_p>0\), depending only on
\(p\), such that
\begin{align}\label{eq2.12}
 &\left\|
 \int_s^tY(r)\,dx(r)
 -
 Y(s)x_{s,t}
 -
 Y'(s)\mathbb X_{s,t}
 \right\|\notag\\
 &\quad\leq
 c_p\left(
 \tnorm{Y'}_{p\text{-var},[s,t]}
 \tnorm{\mathbb X}_{q\text{-var},[s,t]}
 +
 \tnorm{R^Y}_{q\text{-var},[s,t]}
 \tnorm{x}_{p\text{-var},[s,t]}
 \right).
\end{align}

\begin{definition}\label{def2.1}
Let \(E\) be a Banach space. A nonnegative random
variable \(R\) is called tempered if, for every \(c>0\),
\begin{equation}\label{eq:tempered-random-variable}
 \lim_{t\to\infty}
 e^{-ct}R(\theta_{-t}\omega)
 =
 0
 \qquad\text{for almost every }\omega.
\end{equation}
A random bounded set \(D(\omega)\subset E\) is called
tempered if its radius
$r_D(\omega)
 :=
 \sup_{z\in D(\omega)}
 \|z\|_E$
is tempered. The collection of all tempered random
bounded subsets of \(E\) is denoted by \(\mathcal D(E)\).
\end{definition}

Throughout the paper, we impose the following assumptions
on the coefficients and the driving rough path in the lattice system
\eqref{eq:main}.

\begin{enumerate}

\item[\textbf{(A1)}]
The weights \(w_{ij}\), \(i,j\in\mathbb Z\), satisfy
$w_{ij}\geq0$, $w_{ij}=w_{ji}$ and
\[
 \sup_{i\in\mathbb Z}\sum_{j\in\mathbb Z}w_{ij}<\infty,
 \qquad
 \sup_{i\in\mathbb Z}\sum_{j\in\mathbb Z}
 w_{ij}|i-j|<\infty.
\]

\item[\textbf{(A2)}]
For every \(i\in\mathbb Z\),
\(f_i\in C^1(\mathbb R^2;\mathbb R)\), and
\[
 \sup_i\norm{\partial_1f_i}_\infty\leq L_0,
 \qquad
 \sup_i\norm{\partial_2f_i}_\infty\leq L_1.
\]
Moreover,
$F_0:=(f_i(0,0))_{i\in\mathbb Z}\in U$ and $\lambda>L_0+L_1$.

\item[\textbf{(A3)}]
The diagonal map \(h:U\to\mathcal L(U)\), defined by
\((h(u)z)_i=h_i(u_i)z_i\), belongs to
\(C_b^4(U;\mathcal L(U))\) and satisfies
\[
 \max_{0\leq m\leq4}\norm{D^mh}_\infty\leq C_h,
\]
where \(D^m h\) denotes the \(m\)-th Fr\'{e}chet derivative,
with \(D^0h:=h\).

\item[\textbf{(A4)}]
Let \((\Omega,\mathcal F,\mathbb P,\theta)\) be a metric
dynamical system carrying a measurable weakly geometric
\(p\)-rough path \(\mathbf x=(x,\mathbb X)\) over \(U\).
Assume that there exist \(1/p<\beta<1\) and a
\(\theta\)-invariant set \(\Omega_0\in\mathcal F\) with
\(\mathbb P(\Omega_0)=1\) such that, for every
\(\omega\in\Omega_0\), \(\mathbf x(\omega)\) is locally
\(\beta\)-H\"older and
\[
 x_{a,b}(\theta_s\omega)=x_{a+s,b+s}(\omega),
 \qquad
 \mathbb X_{a,b}(\theta_s\omega)
 =\mathbb X_{a+s,b+s}(\omega)
\]
for all \(a\leq b\) and \(s\in\mathbb R\). Moreover, for
every bounded interval \([a,b]\), the random variables
$\tnorm{x(\omega)}_{p\text{-var},[a,b]}$
and
$\tnorm{\mathbb X(\omega)}_{q\text{-var},[a,b]}$
are tempered.

\end{enumerate}

Henceforth, we work on \(\Omega_0\), endowed with the
restrictions of \(\mathcal F\), \(\mathbb P\), and
\(\theta\).

Define the nonlocal operator \(\mathcal L:U\to U\) by
\[
 (\mathcal Lu)_i
 :=
 \nu\sum_{j\in\mathbb Z}w_{ij}(u_i-u_j),
 \qquad i\in\mathbb Z.
\]
Assumption \textbf{(A1)} implies that \(\mathcal L\) is
bounded, self-adjoint and nonnegative. Indeed, for all
\(u,v\in U\),
\begin{align}\label{eq2.14}
 \ip{\mathcal Lu}{v}=\frac{\nu}{2}\sum_{i,j\in\Z}w_{ij}(u_i-u_j)(v_i-v_j),
 \quad \ip{\mathcal Lu}{u}\geq0,
 \quad \norm{\mathcal L}\leq2\nu \sup_{i\in\Z}\sum_{j\in\Z}w_{ij}.
\end{align}
Hence, \(-\mathcal L\) is dissipative.
Define \(F:U\times U\to U\) by
$F(u,v):=(f_i(u_i,v_i))_{i\in\mathbb Z}$.
Then it follows from Assumption \textbf{(A2)} that
\begin{align}
 \norm{F(u_1,v_1)-F(u_2,v_2)}
 &\leq
 L_0\norm{u_1-u_2}
 +
 L_1\norm{v_1-v_2},
 \label{eq2.15}\\
 \norm{F(u,v)}
 &\leq
 L_0\norm{u}
 +
 L_1\norm{v}
 +
 \norm{F_0}.
 \label{eq2.16}
\end{align}
With the above notation, the lattice system \eqref{eq:main} can be reformulated as
\begin{equation}\label{eq:vector}
\left\{
\begin{aligned}
&du(t)=\bigl[-\mathcal Lu(t)-\lambda u(t)
+F(u(t),u(t-\tau))\bigr]dt+h(u(t))dx(t),~~t>0,\\
&u(s)=\xi(s),~~s\in[-\tau,0].
\end{aligned}
\right.
\end{equation}
For a solution \(u\), its history segment at time \(t\)
is defined by
$u_t(r):=u(t+r)$ for
$r\in[-\tau,0]$.

\section{Global well-posedness}

In this section, we establish the global well-posedness
of \eqref{eq:vector}. Fix \(v\in C([a,b];U)\) and define
\begin{equation}
\label{eq:local-integral-0}
 G(z,w):=-\mathcal Lz-\lambda z+F(z,w),
 \qquad
 b(t,z;v):=G(z,v(t)).
\end{equation}
Consider
\begin{equation}\label{eq:local-integral}
 du(t)
 =
 b(t,u(t);v)\,dt+h(u(t))\,dx(t),
 \qquad t\in[a,b],
\end{equation}
with initial condition \(u(a)=u_a\in U\).
By \eqref{eq2.14}--\eqref{eq2.16},
we observe that
\begin{align}\label{eq2.40}
 \norm{b(t,z_1;v)-b(t,z_2;v)}\leq L_b\norm{z_1-z_2},
~~~
 \norm{b(t,z;v)}\leq L_b\norm z+L_1\norm v(t)+\norm{F_0},
\end{align}
where $L_b:=\norm{\mathcal L}+\lambda+L_0$.

We first establish several estimates that will be used
repeatedly in the sequel.

\begin{lemma}\label{lem3.1}
Let $(u,u')\in\mathscr D_x^p([a,b];U)$. Then the
operator-valued path $h(u)$ is controlled by $x$, with Gubinelli
derivative $(h(u))'(t)=Dh(u(t))u'(t)$. Moreover, there exists a constant $C_R>0$ such that, for every $(u,u')\in B_R^{\mathscr D_x^p([a,b];U)}$ and every $[s,t]\subset[a,b]$,
\begin{align}
\tnorm{(h(u))'}_{p\text{-var},[s,t]}
&\leq
C_R\left(
\tnorm{u}_{p\text{-var},[s,t]}
+
\tnorm{u'}_{p\text{-var},[s,t]}
\right),
\label{eq3.1}\\
\tnorm{R^{h(u)}}_{q\text{-var},[s,t]}
&\leq
C_R\left(
\tnorm{R^u}_{q\text{-var},[s,t]}
+
\tnorm{u}_{p\text{-var},[s,t]}^2
\right).
\label{eq3.2}
\end{align}
Furthermore, for any $(u,u'),(v,v')
 \in B_R^{\mathscr D_x^p([a,b];U)}$,
\begin{align}
&\|h(u)-h(v)\|_{\infty,[a,b]}
+\|Dh(u)u'-Dh(v)v'\|_{\infty,[a,b]}
\notag\\
&\quad
+\tnorm{Dh(u)u'-Dh(v)v'}_{p\text{-var},[a,b]}
+\tnorm{R^{h(u)}-R^{h(v)}}_{q\text{-var},[a,b]}
\notag\\
&\qquad\leq
C_R
\|(u,u')-(v,v')\|_{\mathscr D_x^p([a,b];U)}.
\label{eq3.3}
\end{align}
\end{lemma}

\begin{proof}
For \(a\leq s\leq t\leq b\), Taylor's formula yields
\begin{align*}
 h(u(t))-h(u(s))
 &=
 Dh(u(s))u_{s,t}
 +
 \int_0^1(1-r)
 D^2h\bigl(u(s)+ru_{s,t}\bigr)
 [u_{s,t},u_{s,t}]\,dr\\
 &=
 Dh(u(s))u'(s)x_{s,t}
 +
 R^{h(u)}_{s,t},
\end{align*}
where
\[
 R^{h(u)}_{s,t}
 :=
 Dh(u(s))R^u_{s,t}
 +
 \int_0^1(1-r)
 D^2h\bigl(u(s)+ru_{s,t}\bigr)
 [u_{s,t},u_{s,t}]\,dr.
\]
By Assumption~\textbf{(A3)} and the definition of
\(B_R^{\mathscr D_x^p([a,b];U)}\), we have
\begin{align*}
 \left\|
 Dh(u(t))u'(t)-Dh(u(s))u'(s)
 \right\|
 &\leq
 C_R\bigl(
 \|u_{s,t}\|+\|u'_{s,t}\|
 \bigr)\\
 \left\|R^{h(u)}_{s,t}\right\|
 &\leq
 C_R\bigl(
 \|R^u_{s,t}\|+\|u_{s,t}\|^2
 \bigr).
\end{align*}
Taking the \(p\)- and \(q\)-variation seminorms,
respectively, and using \(2q=p\), we obtain
\eqref{eq3.1} and \eqref{eq3.2}. Consequently,
\(h(u)\) is controlled by \(x\), with Gubinelli
derivative
$(h(u))'(t)=Dh(u(t))u'(t)$.

Furthermore, for any $(u,u'),(v,v')
\in B_R^{\mathscr D_x^p([a,b];U)}$, subtracting the corresponding Taylor
expansions and using Assumption~\textbf{(A3)}, we conclude that
\begin{align*}
&\|h(u)-h(v)\|_{\infty,[a,b]}
+\|Dh(u)u'-Dh(v)v'\|_{\infty,[a,b]}
\\
&\quad
+\tnorm{Dh(u)u'-Dh(v)v'}_{p\text{-var},[a,b]}
+\tnorm{R^{h(u)}-R^{h(v)}}_{q\text{-var},[a,b]}
\\
&\leq
C_R\Bigl(
\|u-v\|_{\infty,[a,b]}
+\|u'-v'\|_{\infty,[a,b]}
+\tnorm{u-v}_{p\text{-var},[a,b]}
\\
&\hspace{4.2cm}
+\tnorm{u'-v'}_{p\text{-var},[a,b]}
+\tnorm{R^u-R^v}_{q\text{-var},[a,b]}
\Bigr),
\end{align*}
which, together with
\[
 \tnorm{u-v}_{p\text{-var},[a,b]}
 \leq
 \|u'-v'\|_{\infty,[a,b]}
 \tnorm{x}_{p\text{-var},[a,b]}
 +
 \tnorm{R^u-R^v}_{q\text{-var},[a,b]},
\]
implies \eqref{eq3.3}.
\end{proof}

\begin{lemma}\label{lem3.2}
There exist constants \(\gamma_0\in(0,1)\) and \(C>0\), depending
only on \(p\) and \(C_h\), such that the following assertions hold. Let $I=[s,t]$ and suppose that
$\omega_{\mathbf x}(s,t)\leq\gamma_0$. The pure rough differential equation
\begin{align}\label{eq3.4-0}
 dy(r)=h(y(r))\,d\mathbf x(r),~~y(s)=\zeta,~~r\in I,
\end{align}
generates a \(C^2\)-flow $(\psi_{s,r})_{r\in I}$ of diffeomorphisms of \(U\).
Set
$\rho_I:=\omega_{\mathbf x}(s,t)^{1/p}$. Then
\begin{align}
\sup_{r\in I, \zeta\in U}
\left\|
\psi_{s,r}(\zeta)-\zeta
\right\|
&\leq
C\rho_I,
\label{eq3.4}\\
\sup_{r\in I, \zeta\in U}
(\|D\psi_{s,r}(\zeta)-\operatorname{Id}_U
\|&+\|
D(\psi_{s,r}^{-1})(\zeta)-\operatorname{Id}_U
\|)\leq C\rho_I,
\label{eq3.5}\\
\sup_{r\in I, \zeta\in U}
(\|D\psi_{s,r}(\zeta)
\|+\|D(\psi_{s,r}^{-1})(\zeta)
\|&+\|D^2\psi_{s,r}(\zeta)
\|+\|
D^2(\psi_{s,r}^{-1})(\zeta)\|
)\leq C.\label{eq3.6}
\end{align}
Furthermore, let \(u\) be a controlled solution of
\eqref{eq:local-integral} on \(I\), and set
$z(r):=\psi_{s,r}^{-1}(u(r))$. Then $z$ is absolutely continuous with \(z(s)=u(s)\), and for almost every
\(r\in I\),
\begin{align}\label{eq3.7}
 \dot z(r)=D\psi_{s,r}^{-1}(u(r))\,b(r,u(r);v),
~~~u(r)=\psi_{s,r}(z(r)).
\end{align}

\end{lemma}

\begin{proof}
By using \cite[Theorem~6]{Bailleul2014} and an argument
analogous to that of
\cite[Proposition~11.13]{FrizVictoir2010},
we obtain that \eqref{eq3.4-0} generates a \(C^2\)-flow of
diffeomorphisms
\((\psi_{s,r})_{s\leq r}\) on \(U\), with
$\psi_{s,s}=\operatorname{Id}_U$.

Moreover,
the rough Taylor expansion gives
\[
 \psi_{s,r}(\zeta)-\zeta
 =
 h(\zeta)x_{s,r}
 +
 H(\zeta)\mathbb X_{s,r}
 +
 R_{s,r}^{\psi}(\zeta),
\]
where \(H(\zeta)=Dh(\zeta)h(\zeta)\). In view of \(h\in C_b^4(U;\mathcal L(U))\), by the sewing estimate we have
$\sup_{\zeta\in U}
\|
R_{s,r}^{\psi}(\zeta)\|
\leq
C\omega_{\mathbf{x}}(s,r)^{3/p}$.
Consequently,
\[
\begin{aligned}
\sup_{\zeta\in U}
\left\|
\psi_{s,r}(\zeta)-\zeta
\right\|
&\leq
\sup_{\zeta\in U}
\left\|
h(\zeta)x_{s,r}\right\|
+
\sup_{\zeta\in U}
\left\|
Dh(\zeta)h(\zeta)\mathbb X_{s,r}
\right\|
+
\sup_{\zeta\in U}
\|
R_{s,r}^{\psi}(\zeta)\|
\\
&\leq
C\left(
\omega_{\mathbf{x}}(s,r)^{1/p}
+
\omega_{\mathbf{x}}(s,r)^{2/p}
+
\omega_{\mathbf{x}}(s,r)^{3/p}
\right)\leq C\rho_I,
\end{aligned}
\]
where in the last inequality we have used the superadditivity
of \(\omega_{\mathbf x}\) and \(\omega_{\mathbf x}(s,t)\leq\gamma_0<1\).
Taking the supremum over \(r\in I\) proves
\eqref{eq3.4}.

Taking the Fr\'echet derivative of the integral equation
for \(\psi_{s,r}(\zeta)\) with respect to \(\zeta\), we
obtain the first variational equation
\begin{align}\label{eq3.7-1}
 D\psi_{s,r}(\zeta)-\operatorname{Id}_U
 =
 \int_s^r
 Dh\bigl(\psi_{s,\ell}(\zeta)\bigr)
 D\psi_{s,\ell}(\zeta)
 d\mathbf x(\ell),
\end{align}
where the equation is understood in the operator sense.
It follows from \(h\in C_b^4(U;\mathcal L(U))\) that
\begin{align*}
&\sup_{\zeta\in U}
\left\|
D\psi_{s,r}(\zeta)-\operatorname{Id}_U
\right\|
\leq
C\rho_I\sup_{\zeta\in U, \ell\in[s,r]}
 \left\|D\psi_{s,\ell}(\zeta)\right\|
\leq
C\rho_I
\left(
1+
\sup_{\zeta\in U, \ell\in[s,r]}
\left\|
D\psi_{s,\ell}(\zeta)-\operatorname{Id}_U
\right\|
\right).
\end{align*}
By choosing \(\gamma_0>0\) sufficiently small such that $C\rho_I\leq C\gamma_0^{1/p}\leq
\frac12$, we obtain that
\begin{equation}\label{eq3.7-0-1}
\sup_{r\in I, \zeta\in U}
\left\|
D\psi_{s,r}(\zeta)-\operatorname{Id}_U
\right\|
\leq
C\rho_I.
\end{equation}
Reducing \(\gamma_0\) further, if necessary, we may
assume that the right-hand side of
\eqref{eq3.7-0-1} is at most \(\frac12\).
The Neumann series then implies that
\(D\psi_{s,r}(\zeta)\) is invertible for all
\(r\in I\) and \(\zeta\in U\), with
$\sup_{r\in I,\zeta\in U}
\|[D\psi_{s,r}(\zeta)]^{-1}\|
\leq C$ and $\sup_{r\in I,\zeta\in U}
\|D\psi_{s,r}(\zeta)\|\leq C$.

For \(y\in U\), set
\(\zeta=\psi_{s,r}^{-1}(y)\). The Fr\'echet chain rule
gives
$D(\psi_{s,r}^{-1})(y)=
 [D\psi_{s,r}(\zeta)]^{-1}$.
Hence
\begin{align}
&\sup_{r\in I, y\in U}\left\|
D(\psi_{s,r}^{-1})(y)-\operatorname{Id}_U
\right\|
=\sup_{r\in I, \zeta\in U}
\left\|
[D\psi_{s,r}(\zeta)]^{-1}-\operatorname{Id}_U
\right\|\nonumber\\
&\leq\sup_{r\in I, \zeta\in U}
\left\|
[D\psi_{s,r}(\zeta)]^{-1}
\right\|\sup_{r\in I, \zeta\in U}
\left\|
D\psi_{s,r}(\zeta)-\operatorname{Id}_U
\right\|
\leq
C\rho_I,\nonumber
\end{align}
which together with \eqref{eq3.7-0-1} proves \eqref{eq3.5}.

Taking the Fr\'echet derivative of
\eqref{eq3.7-1} once more with respect to \(\zeta\), we
obtain
\begin{align}
 D^2\psi_{s,r}(\zeta)
 ={}&
 \int_s^r
 \Bigl\{
 Dh\bigl(\psi_{s,\ell}(\zeta)\bigr)
 D^2\psi_{s,\ell}(\zeta)
+
 D^2h\bigl(\psi_{s,\ell}(\zeta)\bigr)
 \bigl[
 D\psi_{s,\ell}(\zeta),
 D\psi_{s,\ell}(\zeta)
 \bigr]
 \Bigr\}
 \,d\mathbf x(\ell),
 \label{eq:second-variational}
\end{align}
with \(D^2\psi_{s,s}(\zeta)=0\).
Thanks to \(h\in C_b^4(U;\mathcal L(U))\) and $\sup_{r\in I,\zeta\in U}
\|D\psi_{s,r}(\zeta)\|\leq C$, from
\eqref{eq:second-variational} we deduce that
\[
 \sup_{r\in I,\zeta\in U}
 \left\|
 D^2\psi_{s,r}(\zeta)
 \right\|
 \leq
 C\rho_I
 \left(
 \sup_{r\in I,\zeta\in U}
 \left\|
 D^2\psi_{s,r}(\zeta)
 \right\|
 +1
 \right).
\]
By reducing \(\gamma_0\) once more, if necessary, we may
ensure that
$C\rho_I
 \leq
 C\gamma_0^{1/p}
 \leq
 \frac12$.
It then follows that
$\sup_{r\in I,\zeta\in U}
 \left\|D^2\psi_{s,r}(\zeta)\right\|
 \leq C$.
Applying the second-order Fr\'echet chain rule to
$\psi_{s,r}^{-1}\circ\psi_{s,r}
=\operatorname{Id}_U$
yields, for \(y=\psi_{s,r}(\zeta)\),
\[
 D^2(\psi_{s,r}^{-1})(y)
 =
 -[D\psi_{s,r}(\zeta)]^{-1}
 D^2\psi_{s,r}(\zeta)
 \left[
 [D\psi_{s,r}(\zeta)]^{-1},
 [D\psi_{s,r}(\zeta)]^{-1}
 \right],
\]
where the right-hand side is understood as a bilinear
map.
Hence
$\sup_{r\in I, \zeta\in U}
\|D^2(\psi_{s,r}^{-1})(\zeta)\|
\leq C$, which proves \eqref{eq3.6}.

Finally, the rough Doss--Sussmann transformation associated with
the flow \(\psi_{s,r}\) establishes a one-to-one
correspondence between controlled solutions \(u\) of
\eqref{eq:local-integral} and absolutely continuous
solutions \(z\) of
\[
 \dot z(r)
 =
 [D\psi_{s,r}(z(r))]^{-1}
 b\bigl(r,\psi_{s,r}(z(r));v\bigr),
\]
through $z(r)=\psi_{s,r}^{-1}(u(r))$
and $u(r)=\psi_{s,r}(z(r))$. We complete the proof.
\end{proof}

\begin{lemma}\label{lem3.3}
Let $a=t_0<t_1<\cdots<t_N=b$ be a finite partition of $[a,b]$, and let $I_j:=[t_j,t_{j+1}]$, $j=0,\ldots,N-1$.
Suppose that $x\in C^{p\text{-var}}([a,b];U)$, $y\in C([a,b];V)$, $y'\in C\bigl([a,b];\mathcal L(U,V)\bigr)$, and $(y,y')$ is controlled by $x$ on each interval $I_j$ with
$y'|_{I_j}
\in
C^{p\text{-var}}
\bigl(I_j;\mathcal L(U,V)\bigr)$, $R^{y,j}
\in
C_2^{q\text{-var}}(I_j;V)$, $
q=\frac p2$.
Then $(y,y')$ is controlled by $x$ on $[a,b]$. Moreover,
\begin{align}
\tnorm{y'}_{p\text{-var},[a,b]}
&\leq (N+1)M_{y'},
\label{eq3.30}\\
\tnorm{R^y}_{q\text{-var},[a,b]}
&\leq
C_p(N+1)^2
\left(
M_R
+
M_{y'}\tnorm{x}_{p\text{-var},[a,b]}
\right),
\label{eq3.31}\\
\tnorm{y}_{p\text{-var},[a,b]}
&\leq
\|y'\|_{\infty,[a,b]}
\tnorm{x}_{p\text{-var},[a,b]}
+
\tnorm{R^y}_{q\text{-var},[a,b]},
\label{eq3.32}
\end{align}
where $C_p>0$ depends only on $p$ and is independent of the
partition,
$M_R:=(\sum_{j=0}^{N-1}
\tnorm{R^{y,j}}_{q\text{-var},I_j}^{q})^{1/q}$ and
$M_{y'}:=(\sum_{j=0}^{N-1}
\tnorm{y'}_{p\text{-var},I_j}^{p})^{1/p}$.
\end{lemma}

\begin{proof}
For any $a\leq s\leq t\leq b$, define
$R^y_{s,t}
:=
y_{s,t}-y'(s)x_{s,t}.$
If $s,t\in I_j$, then the local controlled decomposition implies
$R^y_{s,t}=R^{y,j}_{s,t}$. If $s\in I_i$ and $t\in I_j$ with $i\neq j$,
then insert into $(s,t)$ all the points of the fixed
partition $\{t_j\}_{j=0}^N$ lying in this interval, and relabel
the resulting subdivision as $s=r_0<r_1<\cdots<r_m=t$.
Each subinterval $[r_\ell,r_{\ell+1}]$ is contained in one of
the original partition intervals $I_j$. Hence,
\[
y_{s,t}
=\sum_{\ell=0}^{m-1}y_{r_\ell,r_{\ell+1}}
=
\sum_{\ell=0}^{m-1}\bigg(
y'(r_\ell)x_{r_\ell,r_{\ell+1}}
+
R^y_{r_\ell,r_{\ell+1}}\bigg).
\]
It follows that
\begin{align}\label{eq3.33}
R^y_{s,t}
=
y_{s,t}-y'(s)x_{s,t}
=
\sum_{\ell=0}^{m-1}
R^y_{r_\ell,r_{\ell+1}}
+
\sum_{\ell=1}^{m-1}
\bigl(y'(r_\ell)-y'(s)\bigr)
x_{r_\ell,r_{\ell+1}}.
\end{align}
Let $\Pi=\{a=s_0<s_1<\cdots<s_M=b\}$ be an arbitrary finite partition of $[a,b]$. For each $[s_i,s_{i+1}]\in\Pi$, insert the fixed partition points $\{t_{j}\}_{j=0}^{N}$ lying in $[s_i,s_{i+1}]$ and relabel the resulting subdivision
as
$s_i=r_0<r_1<\cdots<r_m=s_{i+1}$.
Then by $(\sum_{\ell=0}^{m-1}a_\ell)^p\leq
m^{p-1}\sum_{\ell=0}^{m-1}a_\ell^p$, $a_\ell\geq0$, and $m\leq N+1$, we obtain that
\begin{align}
&\sum_{[s_i,s_{i+1}]\in\Pi}
\left\|y'_{s_i,s_{i+1}}\right\|^p
\leq
(N+1)^{p-1}
\sum_{[s_i,s_{i+1}]\in\Pi}
\sum_{\ell=0}^{m-1}
\left\|y'_{r_\ell,r_{\ell+1}}\right\|^p
\leq
(N+1)^{p-1}
\sum_{j=0}^{N-1}
\tnorm{y'}_{p\text{-var},I_j}^{p}.\nonumber
\end{align}
Consequently, taking the supremum over all finite partitions
$\Pi$ of $[a,b]$ in the above inequality yields
\eqref{eq3.30}.

In view of \eqref{eq3.33}, using $(a+b)^q
\leq2^{q-1}(a^q+b^q)$, $a,b\geq0$, H\"{o}lder's inequality and $2q=p$, we have
\begin{align}\label{eq3.34}
&\|R^y_{s_i,s_{i+1}}\|^q
\leq
2^{q-1}\bigg\|\sum_{\ell=0}^{m-1}
R^y_{r_\ell,r_{\ell+1}}\bigg\|^q
+
2^{q-1}\bigg\|\sum_{\ell=1}^{m-1}
\bigl(y'(r_\ell)-y'(s_i)\bigr)
x_{r_\ell,r_{\ell+1}}\bigg\|^q\\
&\leq
2^{q-1}
(N+1)^{q-1}\bigg[
\sum_{\ell=0}^{m-1}
\left\|R^y_{r_\ell,r_{\ell+1}}\right\|^q
+\left(
\sum_{\ell=1}^{m-1}
\left\|y'(r_\ell)-y'(s_i)\right\|^p
\right)^{q/p}
\left(
\sum_{\ell=1}^{m-1}
\left\|x_{r_\ell,r_{\ell+1}}\right\|^p
\right)^{q/p}\bigg].\nonumber
\end{align}
Thanks to
\begin{align}
\sum_{\ell=1}^{m-1}\left\|y'(r_\ell)-y'(s_i)\right\|^p
\leq
(N+1)^{p-1}
\sum_{\ell=1}^{m-1}
\sum_{k=0}^{\ell-1}
\left\|y'_{r_k,r_{k+1}}\right\|^p
\leq
(N+1)^p
\sum_{k=0}^{m-1}
\left\|y'_{r_k,r_{k+1}}\right\|^p,\nonumber
\end{align}
we deduce that
\begin{align}\label{eq3.35}
&\sum_{[s_i,s_{i+1}]\in\Pi}
\left(
\sum_{\ell=1}^{m-1}
\left\|y'(r_\ell)-y'(s_i)\right\|^p
\right)^{q/p}
\left(
\sum_{\ell=1}^{m-1}
\left\|x_{r_\ell,r_{\ell+1}}\right\|^p
\right)^{q/p}\nonumber\\
&\leq
(N+1)^q
\left(\sum_{[s_i,s_{i+1}]\in\Pi}
\sum_{k=0}^{m-1}
\left\|y'_{r_k,r_{k+1}}\right\|^p\right)^{q/p}
\left(\sum_{[s_i,s_{i+1}]\in\Pi}
\sum_{\ell=1}^{m-1}
\left\|x_{r_\ell,r_{\ell+1}}\right\|^p
\right)^{q/p}\nonumber\\
&\leq (N+1)^q
M_{y'}^q
\tnorm{x}_{p\text{-var},[a,b]}^q,
\end{align}
where we have again used H\"older's inequality and $2q=p$. Substituting \eqref{eq3.35}, together with $\sum_{[s_i,s_{i+1}]\in\Pi}
\sum_{\ell=0}^{m-1}
\left\|R^y_{r_\ell,r_{\ell+1}}\right\|^q
\leq
\sum_{j=0}^{N-1}
\tnorm{R^{y,j}}_{q\text{-var},I_j}^{q}$,
into \eqref{eq3.34} yields
\[
\sum_{[s_i,s_{i+1}]\in\Pi}
\|R^y_{s_i,s_{i+1}}\|^q
\leq
C_p(N+1)^{2q}
\left(
M_R^q
+
M_{y'}^q
\tnorm{x}_{p\text{-var},[a,b]}^q
\right).
\]
Therefore \eqref{eq3.31} follows immediately by taking the supremum over all finite partitions $\Pi$ of
$[a,b]$ in the above inequality.
Finally, by \eqref{eq3.30}-\eqref{eq3.31} and the decomposition
$y_{s,t}=y'(s)x_{s,t}+R^y_{s,t}$ for $a\leq s\leq t\leq b$, we conclude that
$(y,y')$ is controlled by $x$ on $[a,b]$. Moreover,
Minkowski's inequality and the fact that $q<p$ yield
\eqref{eq3.32}. This completes the proof.
\end{proof}

\begin{lemma}\label{lem3.4}
For every $T>0$, there exists a constant $C_T>0$ such that
every solution $u$ of \eqref{eq:local-integral} on an interval
$[a,b]$ with $b-a\leq T$ satisfies
\begin{align}\label{eq3.36}
&\|u\|_{\infty,[a,b]}
+\tnorm{u}_{p\text{-var},[a,b]}
+\tnorm{R^u}_{q\text{-var},[a,b]}\nonumber\\
&\qquad\leq
C_T
\bigl(
1+\|u(a)\|+\|v\|_{\infty,[a,b]}
\bigr)
\left(
1+\tnorm{\mathbf{x}}_{p\text{-var},[a,b]}^{5p}
\right),
\end{align}
where the constant $C_T$ depends only on $p$, $T$, $L_b$, $L_1$ and $C_h$.
\end{lemma}

\begin{proof}
Let $\gamma_0$ be as in Lemma~\ref{lem3.2}, and fix
$\gamma\in(0,\gamma_0]$. Starting from $\sigma_0=a$, define recursively $\sigma_{k+1}
:=\inf\{r>\sigma_k:
\omega_{\mathbf x}(\sigma_k,r)\geq\gamma\}
\wedge b$.
This yields the greedy partition $a=\sigma_0<\sigma_1<\cdots<\sigma_{N_{\mathrm g}}=b$. By the continuity and superadditivity of $\omega_{\mathbf x}$, we have  $(N_{\mathrm g}-1)\gamma=\sum_{k=0}^{N_{\mathrm g}-2}
\omega_{\mathbf x}(\sigma_k,\sigma_{k+1})
\leq
\omega_{\mathbf x}(a,b)$.
We next refine the greedy partition by inserting all unit-grid
points lying in $(a,b)$ and relabel the resulting partition as
$a=\tau_0<\tau_1<\cdots<\tau_N=b$. It follows from $b-a\leq T$ that
\begin{align}\label{eq3.37}
N
\leq
N_{\mathrm g}+\lceil T\rceil
\leq
1+\frac{\omega_{\mathbf x}(a,b)}{\gamma}
+\lceil T\rceil
\leq
C_{T,\gamma}
\bigl(1+\omega_{\mathbf x}(a,b)\bigr).
\end{align}

On each interval $I_j=[\tau_j,\tau_{j+1}]$, define $z(r):=\psi_{\tau_j,r}^{-1}(u(r))$ for $r\in I_j$,
where $\psi_{\tau_j,r}$ is the $C^2$ rough flow from Lemma~\ref{lem3.2}. By Lemma \ref{lem3.2}, \eqref{eq2.40} and $\rho_j=
\omega_{\mathbf x}(\tau_j,\tau_{j+1})^{1/p}\leq\gamma^{1/p}\leq1$,
we obtain that for
$r\in I_j$,
\begin{align}\label{eq3.38}
\|\dot z(r)\|
&\leq
\left\|D\psi_{\tau_j,r}^{-1}(u(r))
\right\|
\left\|b(r,u(r);v)\right\|\nonumber\\
&\leq
C\left(
1+\|u(r)\|+\|v\|_{\infty,[a,b]}
\right)
\leq C\left(
1+\|z(r)\|+\|v\|_{\infty,[a,b]}
\right).
\end{align}
Then using Gronwall's inequality yields
\begin{align}\label{eq3.39}
1+\|z(r)\|+\|v\|_{\infty,[a,b]}
\leq
e^{C(r-\tau_j)}
\left(
1+\|u(\tau_j)\|+\|v\|_{\infty,[a,b]}
\right),
\qquad r\in I_j.
\end{align}
Taking $r=\tau_{j+1}$ in \eqref{eq3.39} and using $u(\tau_{j+1})
=\psi_{\tau_j,\tau_{j+1}}\bigl(z(\tau_{j+1})\bigr)$ together with
$\|\psi_{\tau_j,\tau_{j+1}}
\bigl(z(\tau_{j+1})\bigr)
-z(\tau_{j+1})\|\leq C\rho_j$,
we obtain
\begin{align}\label{eq3.40}
1+\|u(\tau_{j+1})\|+\|v\|_{\infty,[a,b]}
\leq
e^{C|I_j|}
\left(
1+\|u(\tau_j)\|
+\|v\|_{\infty,[a,b]}
\right)
+
C\rho_j.
\end{align}
Iterating \eqref{eq3.40} and using $\sum_{j=0}^{N-1}|I_j|
=b-a\leq T$ and $\rho_j\leq1$, we obtain that
\begin{align}\label{eq3.41}
1+\|u(\tau_k)\|
+\|v\|_{\infty,[a,b]}
&\leq
e^{C\sum_{j=0}^{k-1}|I_j|}
\left(
1+\|u(a)\|
+\|v\|_{\infty,[a,b]}+C\sum_{j=0}^{k-1}\rho_j
\right)\nonumber\\
&\leq
C_T
\left(
1+\|u(a)\|
+\|v\|_{\infty,[a,b]}
\right)(1+N),\qquad
k=1,\ldots,N.
\end{align}
For \(r\in I_j\), by using
\(u(r)=\psi_{\tau_j,r}(z(r))\),
\eqref{eq3.4}, \eqref{eq3.39} and \eqref{eq3.41}, we have
\[
 \|u(r)\|
 \leq
 \|z(r)\|+C\rho_j
 \leq
 C_T
 \bigl(
 1+\|u(a)\|+\|v\|_{\infty,[a,b]}
 \bigr)(1+N).
\]
Taking the supremum over \(r\in[a,b]\) yields
\begin{align}\label{eq3.42}
 \|u\|_{\infty,[a,b]}
 \leq
 C_T
 \bigl(
 1+\|u(a)\|
 +\|v\|_{\infty,[a,b]}
 \bigr)(1+N).
\end{align}
Fix $j\in\{0,\ldots,N-1\}$ and let $s,t\in I_j$ with
$s\leq t$. Then
\begin{align}\label{eq3.43}
u_{s,t}
=
\Bigl[
\psi_{\tau_j,t}(z(t))
-
\psi_{\tau_j,t}(z(s))
\Bigr]+
\Bigl[
\psi_{\tau_j,t}(z(s))
-
\psi_{\tau_j,s}(z(s))
\Bigr].
\end{align}
From \eqref{eq3.6}, it holds that $\|
\psi_{\tau_j,t}(z(t))
-\psi_{\tau_j,t}(z(s))\|\leq C\|z_{s,t}\|$. By the flow property $\psi_{\tau_j,t}
=\psi_{s,t}\circ\psi_{\tau_j,s}$ and \eqref{eq3.4},
we obtain that
\begin{align}\label{eq3.45}
\left\|\psi_{\tau_j,t}(z(s))
-
\psi_{\tau_j,s}(z(s))\right\|
=
\left\|\psi_{s,t}(\psi_{\tau_j,s}(z(s)))-\psi_{\tau_j,s}(z(s))\right\|
\leq
C\omega_{\mathbf x}(s,t)^{1/p}.
\end{align}
Let $\Pi=\{\tau_j=s_0<s_1<\cdots<s_M=\tau_{j+1}\}$ be an arbitrary finite partition of $I_j$. Combining \eqref{eq3.43} and \eqref{eq3.45}, and using
Minkowski's inequality,
$\left(\sum_i a_i^p\right)^{1/p}
 \leq
 \sum_i a_i$, $a_i\geq0$,
together with the superadditivity of
\(\omega_{\mathbf x}\), we have
\begin{align}\label{eq3.46}
\left(
\sum_{i=0}^{M-1}
\|u_{s_i,s_{i+1}}\|^p
\right)^{1/p}
&\leq
C\left(
\sum_{i=0}^{M-1}
\|z_{s_i,s_{i+1}}\|^p
\right)^{1/p}+
C\left(
\sum_{i=0}^{M-1}
\omega_{\mathbf x}(s_i,s_{i+1})
\right)^{1/p}\nonumber\\
&\leq
C\tnorm{z}_{1\text{-var},I_j}
+C\left(\omega_{\mathbf x}(\tau_j,\tau_{j+1})\right)^{1/p}
\leq
C\tnorm{z}_{1\text{-var},I_j}+C\rho_j.
\end{align}
From \eqref{eq3.38}, it follows that
\begin{align}\label{eq3.44}
\tnorm{z}_{1\text{-var},I_j}
\leq
\int_{\tau_j}^{\tau_{j+1}}
\|\dot z(r)\|\,dr
\leq
C|I_j|
\left(
1+\|u\|_{\infty,I_j}
+\|v\|_{\infty,[a,b]}
\right).
\end{align}
Substituting
\eqref{eq3.44} into
\eqref{eq3.46}, and then taking the supremum
over all finite partitions $\Pi$ of $I_j$, we deduce that
\begin{align}\label{eq3.47}
\tnorm{u}_{p\text{-var},I_j}
\leq
C\rho_j
+
C|I_j|
\left(
1+\|u\|_{\infty,I_j}
+\|v\|_{\infty,[a,b]}
\right).
\end{align}
Consequently,
\begin{align}\label{eq3.47-1}
\left(
\sum_j
\tnorm{u}_{p\text{-var},I_j}^{p}
\right)^{1/p}
\leq
C\omega_{\mathbf x}(a,b)^{1/p}
+
C_T(1+\|u\|_{\infty,[a,b]}+\|v\|_{\infty,[a,b]}).
\end{align}

Observe that for $s,t\in I_j$,
\begin{equation}\label{eq3.48}
R^u_{s,t}=
u_{s,t}-h(u(s))x_{s,t}
=\bigg[\psi_{\tau_j,t}(z(t))
-
\psi_{\tau_j,t}(z(s))\bigg]
+\bigg[\psi_{\tau_j,t}(z(s))
-
\psi_{\tau_j,s}(z(s))-h(u(s))x_{s,t}\bigg].
\end{equation}
By the flow property $\psi_{\tau_j,t}=
\psi_{s,t}\circ\psi_{\tau_j,s}$ and $u(s)=\psi_{\tau_j,s}(z(s))$,
we obtain that
\begin{align}
\left\|\psi_{\tau_j,t}(z(s))
-
\psi_{\tau_j,s}(z(s))-h(u(s))x_{s,t}\right\|
=\left\|H(u(s))\mathbb X_{s,t}
+
\mathcal R^\psi_{s,t}(u(s))\right\|\leq
C\omega_{\mathbf x}(s,t)^{2/p},\nonumber
\end{align}
where in the last inequality we have used \eqref{eq2.4}, Assumption \textbf{(A3)} and sewing lemma. Therefore, the similar arguments as in \eqref{eq3.47}-\eqref{eq3.47-1} can be used to give
\begin{align}\label{eq3.49}
\left(
\sum_j\tnorm{R^{u}}_{q\text{-var},I_j}^{q}
\right)^{1/q}
\leq
C\omega_{\mathbf x}(a,b)^{2/p}
+
C_T(1+\|u\|_{\infty,[a,b]}+\|v\|_{\infty,[a,b]}).
\end{align}

Since $u'=h(u)$ and $Dh$ is bounded, we have
$\tnorm{u'}_{\pvar,I_j}\leq C\tnorm u_{\pvar,I_j}$. In view of \eqref{eq3.47-1},
applying Lemma~\ref{lem3.3} yields
\begin{equation}\label{eq3.50}
\tnorm{u}_{p\text{-var},[a,b]}
+
\tnorm{R^u}_{q\text{-var},[a,b]}
\leq
C_p(N+1)^2(1+\|u\|_{\infty,[a,b]}+\|v\|_{\infty,[a,b]})
\bigl(1+\omega_{\mathbf x}(a,b)\bigr).
\end{equation}
Therefore the assertion follows immediately by inserting \eqref{eq3.42},
$1+N\leq C_T(1+\norm{\mathbf x}_{\pvar,[a,b]}^p)$ and
$1+\omega_{\mathbf x}(a,b)\leq C_T(1+\norm{\mathbf x}_{\pvar,[a,b]}^p)$ into \eqref{eq3.50}.
\end{proof}

\begin{corollary}\label{cor3.5}
For every $T>0$, there exist constants $C_T>0$ and $m_T>0$, such that every controlled solution $u$ of \eqref{eq:vector} with
initial history $\xi\in\Hh$ satisfies
\begin{align}\label{eq3.51-10}
\sup_{-\tau\leq t\leq T}\norm{u(t)}
\leq
C_T\bigl(1+\norm{\xi}_{\Hh}\bigr)
\left(
1+\tnorm{\mathbf x}_{\pvar,[0,T]}^{m_T}
\right),
\end{align}
where $m_T$ is independent of $\xi$, $\mathbf x$
and the particular solution $u$.
\end{corollary}

\begin{proof}
Let
$M:=\lceil T/\tau\rceil$,
$J_j:=[(j-1)\tau,j\tau]\cap[0,T]$,
$j=1,\ldots,M$.
On \(J_1=[0,\tau\wedge T]\), we have
\(t-\tau\in[-\tau,0]\) and hence
\(u(t-\tau)=\xi(t-\tau)\). Set
\(v_1(t):=\xi(t-\tau)\). By
Lemma~\ref{lem3.4},
\begin{align}\label{eq3.51}
 \|u\|_{\infty,J_1}
 +\tnorm{u}_{p\text{-var},J_1}
 +\tnorm{R^u}_{q\text{-var},J_1}
 \leq
 C_T
 \left(1+\|\xi\|_{\Hh}\right)
 \left(
 1+\tnorm{\mathbf x}_{p\text{-var},[0,T]}^{5p}
 \right).
\end{align}
We proceed by induction. Suppose that an estimate
analogous to \eqref{eq3.51} has been established on \(J_k\),
\(1\leq k\leq j\), for some \(1\leq j\leq M-1\). Define
\[
 \mathcal U_j
 :=
 \max_{1\leq k\leq j}
 \left\{
 \|u\|_{\infty,J_k}
 +\tnorm{u}_{p\text{-var},J_k}
 +\tnorm{R^u}_{q\text{-var},J_k}
 \right\}.
\]
For \(t\in J_{j+1}\), set
\(v_{j+1}(t):=u(t-\tau)\). Since
\(t-\tau\in J_j\),
$\|v_{j+1}\|_{\infty,J_{j+1}}
 \leq
 \|u\|_{\infty,J_j}
 \leq\mathcal U_j$.
Applying Lemma~\ref{lem3.4} on \(J_{j+1}\) yields
\begin{align*}
 &\|u\|_{\infty,J_{j+1}}
 +\tnorm{u}_{p\text{-var},J_{j+1}}
 +\tnorm{R^u}_{q\text{-var},J_{j+1}}
 \leq
 C_T(1+\mathcal U_j)
 \left(1+\tnorm{\mathbf x}_{p\text{-var},[0,T]}^{5p}
 \right).
\end{align*}
Furthermore,
\begin{equation}\label{eq3.52}
 \mathcal U_{j+1}
 \leq
 C_T(1+\mathcal U_j)
 \left(
 1+\tnorm{\mathbf x}_{p\text{-var},[0,T]}^{5p}
 \right).
\end{equation}
Iterating \eqref{eq3.52} with the initial bound
\eqref{eq3.51}, we obtain that
\[
 \mathcal U_M
 \leq
 C_T(1+\|\xi\|_{\Hh})
 \left(
 1+\tnorm{\mathbf x}_{p\text{-var},[0,T]}^{5p}
 \right)^M.
\]
Consequently,
\[
 \sup_{-\tau\leq t\leq T}\|u(t)\|
 \leq\mathcal U_M
 \leq
 C_T(1+\|\xi\|_{\Hh})
 \left(
 1+\tnorm{\mathbf x}_{p\text{-var},[0,T]}^{5pM}
 \right).
\]
Therefore, the assertion follows immediately by choosing
$m_T=5pM=5p\lceil T/\tau\rceil$.
\end{proof}

\begin{theorem}\label{thm3.8}
Under Assumptions~\textbf{(A1)}-\textbf{(A4)}, for every $\xi\in\Hh$ and every $\omega\in\Omega_0$, the delay equation \eqref{eq:vector}, driven by $\mathbf{x}(\omega)$, admits a unique global solution
$u\in C([ -\tau,\infty);U)$ with $u|_{[-\tau,0]}=\xi$.
Moreover,
for every $T>0$, $(u,h(u))\in\mathscr D_x^p([0,T];U)$, and the solution map
$(\xi,\mathbf{x})
\mapsto
u|_{[-\tau,T]}$
is continuous with respect to the norm topology of
\(\Hh\) and the local \(p\)-rough-path
topology, with values in \(C([-\tau,T];U)\).
\end{theorem}

\begin{proof}
\textbf{Step~1.}
We first consider \eqref{eq:local-integral} with a
prescribed path \(v\in C([a,b];U)\), which will later
represent the delayed term \(u(\cdot-\tau)\). Let
$a=t_0<t_1<\cdots<t_N=b$
be a finite greedy partition such that
$\omega_{\mathbf x}(t_j,t_{j+1})\leq\gamma_0$,
$j=0,\ldots,N-1$,
where \(\gamma_0\) is given by
Lemma~\ref{lem3.2}.
Fix \(j\in\{0,\ldots,N-1\}\), and let
\((\psi_{t_j,r})_{r\in[t_j,t_{j+1}]}\) be the
corresponding rough flow. On \([t_j,t_{j+1}]\), the
rough Doss--Sussmann technique gives a one-to-one
correspondence between controlled solutions \(u\) of
\eqref{eq:local-integral} and absolutely continuous
solutions \(z_j\) of
\begin{equation}\label{eq3.55}
 \dot z_j(r)
 =
 [D\psi_{t_j,r}(z_j(r))]^{-1}
 b\bigl(r,\psi_{t_j,r}(z_j(r));v\bigr),
 \qquad
 z_j(t_j)=u(t_j),
\end{equation}
through
$u(r)=\psi_{t_j,r}(z_j(r))$.
By Lemma~\ref{lem3.2} and \eqref{eq2.40}, the vector
field in \eqref{eq3.55} is locally Lipschitz in \(z_j\)
and of at most linear growth, uniformly with respect to
\(r\in[t_j,t_{j+1}]\). Hence \eqref{eq3.55} admits a unique
solution \(z_j\), and consequently
\eqref{eq:local-integral} admits a unique controlled
solution on \([t_j,t_{j+1}]\).

Starting from \(u(a)=u_a\), these local solutions can be
concatenated successively to yield a unique continuous
solution \(u\) of \eqref{eq:local-integral} on \([a,b]\). Since \((u,h(u))\) is controlled by \(x\) on
each \([t_j,t_{j+1}]\), Lemma~\ref{lem3.3} yields
$(u,h(u))\in\mathscr D_x^p([a,b];U)$.
Moreover, Lemma~\ref{lem3.4} provides the estimate
\eqref{eq3.36} for the solution \(u\) of
\eqref{eq:local-integral} on $[a,b]$.

\textbf{Step~2.}
On \(J_0:=[0,\tau]\), set
$v_0(t):=\xi(t-\tau)$.
Since \(\xi\in\Hh\), the path \(v_0\) is continuous.
Applying Step~1 with \(v=v_0\) and \(u(0)=\xi(0)\)
yields a unique controlled solution of
\eqref{eq:vector} on \(J_0\).
Suppose that the solution has been constructed uniquely
on \([-\tau,m\tau]\). On
$J_m:=[m\tau,(m+1)\tau]$,
the delayed path \(v_m(t):=u(t-\tau)\) is already known
and continuous.
Applying
Step~1 to \eqref{eq:local-integral} with \(v=v_m\) and
initial value \(u(m\tau)\) uniquely extends the solution
of \eqref{eq:vector} to \(J_m\). Iterating this construction and using
Lemma~\ref{lem3.3}, we obtain, for every
\(T>0\), a unique solution \(u\) of \eqref{eq:vector}
such that
$u\in C([-\tau,T];U)$,
$u|_{[-\tau,0]}=\xi$ and
$(u,h(u))\in\mathscr D_x^p([0,T];U)$.
Since \(T>0\) is arbitrary, \eqref{eq:vector} admits a
unique global solution on \([-\tau,\infty)\).

\textbf{Step~3.}
Fix \(T>0\), and let
\(\xi_n\to\xi\) in \(\Hh\) and
\(\mathbf x_n\to\mathbf x\) in the local
\(p\)-rough-path topology. Denote the corresponding
solutions of \eqref{eq:vector} by \(u_n\) and \(u\),
respectively.

On \(J_0:=[0,\tau\wedge T]\), set
$v_n(t):=\xi_n(t-\tau)$ and
$v(t):=\xi(t-\tau)$.
Then
$\|v_n-v\|_{\infty,J_0}
 \leq\|\xi_n-\xi\|_{\Hh}\rightarrow0$
and
$u_n(0)\longrightarrow u(0)$.
We construct a common finite
partition  $0=t_0<t_1<\cdots<t_N=\tau\wedge T$
such that, for every $j=0,\ldots,N-1$ and all sufficiently large $n$,
\[
\omega_{\mathbf{x}}(t_j,t_{j+1})
\leq
\gamma_0
\quad\text{and}\quad
\omega_{\mathbf{x}_n}(t_j,t_{j+1})
\leq
\gamma_0,
\]
where \(\gamma_0\) is given by
Lemma~\ref{lem3.2}. For \(I_j=[t_j,t_{j+1}]\), let
\(\psi^n_{t_j,r}\) and \(\psi_{t_j,r}\) be the rough
flows generated by \(\mathbf x_n\) and \(\mathbf x\),
respectively, and let \(z_n\) and \(z\) be the
corresponding Doss--Sussmann transforms.
By Lemma~\ref{lem3.2}, the It\^o--Lyons continuity
theorem and \eqref{eq2.40}, we have for any $r\in I_j$,
\[
\begin{aligned}
 \|u_n(r)-u(r)\|
 &\leq
 \left\|
 \psi^n_{t_j,r}\bigl(z_n(r)\bigr)
 -
 \psi^n_{t_j,r}\bigl(z(r)\bigr)
 \right\|+
 \left\|
 \psi^n_{t_j,r}\bigl(z(r)\bigr)
 -
 \psi_{t_j,r}\bigl(z(r)\bigr)
 \right\|\\
 &\leq
 C\|z_n-z\|_{\infty,I_j}
 +
 \sup_{\ell\in I_j}
 \left\|
 \psi^n_{t_j,\ell}\bigl(z(\ell)\bigr)
 -
 \psi_{t_j,\ell}\bigl(z(\ell)\bigr)
 \right\|\rightarrow0~~\mbox{as}
~n\rightarrow\infty.
\end{aligned}
\]
Starting from \(u_n(0)\to u(0)\) and proceeding
successively over the finite partition, we obtain
$\|u_n-u\|_{\infty,J_0}\rightarrow0$.
Repeating this argument over the finitely many
method-of-steps intervals contained in \([0,T]\), we
conclude that
\[
 \|u_n-u\|_{\infty,[-\tau,T]}
 \rightarrow0
 \qquad\text{as }n\to\infty.
\]
Hence the solution map $(\xi,\mathbf{x})
\mapsto
u|_{[-\tau,T]}$ is continuous from the product of $\mathcal H$ and the local
$p$-rough-path topology into $C([-\tau,T];U)$. This completes the proof.
\end{proof}

\begin{theorem}\label{thm3.9}
For \(t\geq0\), \(\omega\in\Omega_0\), and \(\xi\in\Hh\), define
\[\Phi(t,\omega)\xi:=
u_t(\cdot;\omega,\xi).\]
Then \(\Phi\) is a continuous random dynamical system on \(\Hh\).
Moreover, for every
\(\omega\in\Omega_0\), the map
$(t,\xi)\longmapsto\Phi(t,\omega)\xi$
is continuous from \([0,\infty)\times\Hh\) into \(\Hh\).

\end{theorem}

\begin{proof}
By Theorem~\ref{thm3.8}, for every
\(\omega\in\Omega_0\) and \(\xi\in\Hh\),
$u(\cdot;\omega,\xi)\in C([-\tau,\infty);U)$.
Consequently, \(u_t(\cdot;\omega,\xi)\in\Hh\) for every
\(t\geq0\), and hence
\(\Phi(t,\omega):\Hh\to\Hh\) is well defined.

Fix $s\geq0$ and set
$v(r):=u(s+r;\omega,\xi)$. By
Assumptions~\textbf{(A4)},
$d\mathbf x_{s+r}(\omega)=d\mathbf x_r(\theta_s\omega)$. Consequently,
the function $v$ solves the same delay equation driven by $\theta_s\omega$:
\[
dv(r)
=
\left[
-\mathcal L v(r)
-\lambda v(r)
+F\bigl(v(r),v(r-\tau)\bigr)
\right]\,dr
+
h\bigl(v(r)\bigr)\,d\mathbf{x}_{r}
(\theta_s\omega),
\]
where its initial history is given by
$v(\iota)=u(s+\iota;\omega,\xi)=
\bigl[\Phi(s,\omega)\xi\bigr](\iota)$,
 $\iota\in[-\tau,0]$.
Define $\bar{v}(r):=u\bigl(r;\theta_s\omega,
\Phi(s,\omega)\xi\bigr)$. By the existence and uniqueness of solution to the delay equation
\eqref{eq:vector}, established in Theorem~\ref{thm3.8}, we obtain
$v(r)=\bar{v}(r)$ for any $r\geq -\tau$.
For every $\xi\in\mathcal H$ and
$\iota\in[-\tau,0]$, we find that
$[\Phi(0,\omega)\xi](\iota)
=u_0(\iota;\omega,\xi)
=u(\iota;\omega,\xi)
=\xi(\iota)$. Then $\Phi(0,\omega)=\operatorname{Id}_{\mathcal H}$. In addition, we have for any $t\geq0$ and
$\iota\in[-\tau,0]$,
\begin{align*}
v(t+\iota)
&=u(s+t+\iota;\omega,\xi)
=u_{s+t}(\iota;\omega,\xi)
=[\Phi(s+t,\omega)\xi](\iota),\\
\bar{v}(t+\iota)
&=u\bigl(t+\iota;\theta_s\omega,
\Phi(s,\omega)\xi\bigr)=
[\Phi(t,\theta_s\omega)
\bigl(\Phi(s,\omega)\xi\bigr)](\iota).
\end{align*}
In view of $v(r)=\bar{v}(r)$ for any $r\geq -\tau$, we have
$\Phi(s+t,\omega)\xi
=
\Phi(t,\theta_s\omega)
\bigl(\Phi(s,\omega)\xi\bigr)$.
Since this identity holds for every $\xi\in\mathcal H$, it follows that
\[
\Phi(s+t,\omega)
=
\Phi(t,\theta_s\omega)\circ\Phi(s,\omega),
\qquad s,t\geq0.
\]

Let \(t_n\to t\) in \([0,T]\), \(\xi_n\to\xi\) in \(\Hh\), and
\(\mathbf{x}_n\to\mathbf{x}\) in the local \(p\)-rough-path topology.
Denote by \(u_n\) and \(u\) the solutions corresponding to
\((\xi_n,\mathbf{x}_n)\) and \((\xi,\mathbf{x})\), respectively. Then using Theorem~\ref{thm3.8} and the uniform continuity of \(u\) on
\([-\tau,T]\), we obtain that
\[
\begin{aligned}
\|(u_n)_{t_n}-u_t\|_{\Hh}
&=
\sup_{\iota\in[-\tau,0]}
\|u_n(t_n+\iota)-u(t+\iota)\|
\\
&\leq
\|u_n-u\|_{\infty,[-\tau,T]}
+
\sup_{\iota\in[-\tau,0]}
\|u(t_n+\iota)-u(t+\iota)\|
\rightarrow0\qquad\text{as }n\to\infty.
\end{aligned}
\]
Therefore, for every \(T>0\), the map
$(t,\xi,\mathbf{x})
\mapsto
u_t(\cdot;\mathbf{x},\xi)$
is jointly continuous from the product of \([0,T]\), \(\Hh\), and
the local \(p\)-rough-path topology into \(\Hh\).
By Assumption~\textbf{(A4)}, we know that
$\omega\mapsto\mathbf x(\omega)$
is measurable. Since the solution map
$(t,\xi,\mathbf{x})
\mapsto
u_t(\cdot;\mathbf{x},\xi)$
is continuous, their composition
$(t,\omega,\xi)\mapsto
 u_t(\cdot;\mathbf x(\omega),\xi)
 =\Phi(t,\omega)\xi$
is jointly measurable on
\(\mathbb{R}_+\times\Omega_0\times\Hh\). This completes the proof.
\end{proof}

\section{Random pullback attractor}

\subsection{Pullback absorbing set}

We first consider the deterministic delay equation obtained from
\eqref{eq:vector} by removing the rough forcing. Let \(\mu\) denote
the solution of
\begin{equation}\label{eq:deterministic}
\dot{\mu}(t)
=
-\mathcal L\mu(t)
-\lambda\mu(t)
+
F\bigl(\mu(t),\mu(t-\tau)\bigr),
\qquad t>0,
\end{equation}
subject to the initial history
$\mu(\theta)
=
\xi(\theta)$,
$\theta\in[-\tau,0]$.
Since \(\lambda>L_0+L_1\), if \(L_1>0\), we may choose
\(\sigma>0\) sufficiently small so that
$L_1e^{\sigma\tau}<
\lambda-L_0$;
if \(L_1=0\), choose any
$\sigma\in(0,\lambda-L_0)$.
Define
$\delta_*:=\min\{
\lambda-L_0-L_1e^{\sigma\tau},
\sigma\}$.

\begin{lemma}\label{lem4.1}
For \(t\geq0\), define
\[
 \mathcal E_\mu(t)
 :=
 \norm{\mu(t)}^2
 +
 L_1e^{\sigma\tau}
 \int_{t-\tau}^{t}
 e^{-\sigma(t-s)}
 \norm{\mu(s)}^2\,ds.
\]
Then
\begin{equation}\label{eq4.1}
 \mathcal E_\mu(t)
 \leq
 e^{-\delta_*t}\mathcal E_\mu(0)
 +
 \frac{\norm{F_0}^2}
 {(\lambda-L_0-L_1)\delta_*},
 \qquad t\geq0.
\end{equation}
Moreover, for every \(t\geq\tau\),
\begin{equation}\label{eq4.2}
 \norm{\mu(t+\cdot)}_{\Hh}^2
 \leq
 \bigl(1+\tau L_1e^{\sigma\tau}\bigr)
 e^{-\delta_*(t-\tau)}
 \norm{\xi}_{\Hh}^2
 +
 \frac{\norm{F_0}^2}
 {(\lambda-L_0-L_1)\delta_*}.
\end{equation}
\end{lemma}

\begin{proof}
By \eqref{eq2.14}, $\ip{\mathcal L\mu}{\mu}\geq0$. Using \eqref{eq2.15} and Young's inequality, we obtain
\begin{align}\label{eq4.2-1}
 \frac{d}{dt}\|\mu(t)\|^2
 &= -2\ip{\mathcal L\mu(t)}{\mu(t)}-2\lambda\|\mu(t)\|^2
 +2\ip{F(\mu(t),\mu(t-\tau))-
 F_0}{\mu(t)}+2\ip{F_0}{\mu(t)}
 \nonumber\\
 &\leq-2\lambda\|\mu(t)\|^2+2L_0\|\mu(t)\|^2
 +2L_1\norm{\mu(t-\tau)}\norm{\mu(t)}
 \nonumber\\
 &+(\lambda-L_0-L_1)\norm{\mu(t)}^2+
 (\lambda-L_0-L_1)^{-1}\norm{F_0}^2
 \nonumber\\
 &\leq-(\lambda-L_0)\norm{\mu(t)}^2+
 L_1\norm{\mu(t-\tau)}^2+
 (\lambda-L_0-L_1)^{-1}\norm{F_0}^2.
\end{align}
In addition, it holds that
\[
 \frac{d}{dt}\int_{t-\tau}^{t}e^{-\sigma(t-s)}\norm{\mu(s)}^2ds
 =\norm{\mu(t)}^2-e^{-\sigma\tau}\norm{\mu(t-\tau)}^2
 -\sigma\int_{t-\tau}^{t}e^{-\sigma(t-s)}\norm{\mu(s)}^2ds.
\]
Consequently,
\[
 \frac{d}{dt}\mathcal E_\mu(t)
\leq-\delta_*\mathcal E_\mu(t)+(\lambda-L_0-L_1)^{-1}
\norm{F_0}^2.
\]
Applying Gronwall's inequality to above inequality yields \eqref{eq4.1}. Notice that $\mathcal E_\mu(0)\leq(1+\tau L_1e^{\sigma\tau})\norm\xi_{\Hh}^2$. Then inserting this into \eqref{eq4.1} gives  \eqref{eq4.2}.
\end{proof}

\begin{lemma}\label{lem4.2}
Let \(u\) and \(\mu\) be the solutions of \eqref{eq:vector} and
\eqref{eq:deterministic}, respectively, with the same initial history
\(\xi\).  Then, for every \(r>0\) and
\(\varepsilon\in(0,1/4)\), there exist constants
\(C_r>0\) and \(C_{r,\varepsilon}>0\) such that
\begin{equation}\label{eq4.4}
 \norm{u-\mu}_{\infty,[0,r]}
 \leq C_r\varepsilon\bigl(1+\norm\xi_{\Hh}\bigr)
 +C_{r,\varepsilon}\bigl(1+\norm{\mathbf x}_{\pvar,[0,r]}^{p}\bigr).
\end{equation}

\end{lemma}

\begin{proof}
By Lemma~\ref{lem4.1}, we obtain that for $0\leq s\leq t\leq r$,
\begin{equation}\label{eq4.5}
 \norm\mu_{\infty,[-\tau,r]}\leq C_r (1+\norm\xi_{\Hh})~~
\mbox{and}~~
 \tnorm\mu_{1\text{-var},[s,t]}\leq C_r(1+\norm\xi_{\Hh})(t-s).
\end{equation}
Let \(\gamma_0\) be given in Lemma~\ref{lem3.2} and let  \(\gamma_\varepsilon\in(0,\gamma_0]\) be chosen sufficiently small so that $\gamma_\varepsilon^{1/p}
\leq\varepsilon$.
Starting from the greedy partition of \([0,r]\) associated with the control \(\omega_{\mathbf{x}}\) at level \(\gamma_\varepsilon\), we adjoin all points of the form \(j\tau\) lying in \([0,r]\) and further refine the resulting partition so that every subinterval has length at most \(\tau\wedge1\).
Denote the resulting partition by
$0=\tau_0<\tau_1<\cdots<\tau_N=r$
and set
$I_i:=[\tau_i,\tau_{i+1}]$,
$\rho_i:=
\omega_{\mathbf{x}}(\tau_i,\tau_{i+1}
)^{1/p}$, $i=0,\ldots,N-1$.
Then $\rho_i\leq\gamma_\varepsilon^{1/p}$, $i=0,\ldots,N-1$.
By using a similar argument as in \eqref{eq3.37}, we have
\begin{equation}\label{eq4.6}
 N\leq C_{r,\varepsilon}\left(1+
\tnorm{x}_{\pvar,[0,r]}^p
 +\tnorm{\mathbb X}_{\qvar,[0,r]}^q\right)
\leq C_{r,\varepsilon}\left(1+
 \tnorm{\mathbf x}_{\pvar,[0,r]}^p
 \right).
\end{equation}

Fix \(i\in\{0,\ldots,N-1\}\). For \(t\in I_i\), let $
\psi_{\tau_i,t}$
denote the local rough flow, and define
$z(t):=
\psi_{\tau_i,t}^{-1}(u(t))$,
$e(t):=z(t)-\mu(t)$.
It follows from \eqref{eq:local-integral-0}, \eqref{eq3.7} and \eqref{eq:deterministic} that
\begin{align*}
\dot e(t)
={}&
D(\psi_{\tau_i,t}^{-1})(u(t))
\Bigl[
G\bigl(u(t),u(t-\tau)\bigr)
-
G\bigl(\mu(t),\mu(t-\tau)\bigr)
\Bigr]
\\
&+
\Bigl[
D(\psi_{\tau_i,t}^{-1})(u(t))
-
\operatorname{Id}_U
\Bigr]
G\bigl(\mu(t),\mu(t-\tau)\bigr)
\end{align*}
for almost every \(t\in I_i\).
By Lemma \ref{lem3.2}, \eqref{eq2.40}, \eqref{eq4.5} and
$u(t)=\psi_{\tau_i,t}(z(t))$, we deduce that
\begin{equation}\label{eq4.7}
\|\dot e(t)\|
\leq
C\|e(t)\|
+
C\|u(t-\tau)-\mu(t-\tau)\|
+
C\rho_i
+
C_r\rho_i
\bigl(1+\|\xi\|_{\Hh}\bigr)
\end{equation}
for almost every \(t\in I_i\),
where we have used
$\norm{u(t)-\mu(t)}
\leq\norm{e(t)}+C\rho_i$.

Since
$|I_i|
\leq
\tau,$
we have
$
-\tau
\leq
t-\tau
\leq
\tau_i$ for
$t\in I_i.$
Set
$
K_i
:=
\sup_{-\tau\leq s\leq\tau_i}
\|u(s)-\mu(s)\|$. It follows that
$\|u(t-\tau)-\mu(t-\tau)\|
\leq
K_i$ for
$t\in I_i$,
and
$
\|e(\tau_i)\|=\|u(\tau_i)-\mu(\tau_i)\|
\leq
K_i.
$
Applying Gronwall's inequality to \eqref{eq4.7}, and using
\(|I_i|\leq1\), we obtain that
\begin{align*}
\sup_{t\in I_i}\|e(t)\|
\leq
\bigl(1+C|I_i|\bigr)K_i
+
C\rho_i|I_i|
+
C_r\rho_i|I_i|
\bigl(1+\|\xi\|_{\Hh}\bigr).
\end{align*}
By $\norm{u(t)-\mu(t)}\leq\norm{e(t)}+C\rho_i$ for $t\in I_i$ and $
K_{i+1}
=\max
\left\{
K_i,
\sup_{t\in I_i}\|u(t)-\mu(t)\|
\right\}$, we have
\begin{equation}\label{eq4.9}
 K_{i+1}\leq(1+C|I_i|)K_i+C\rho_i
 +C_r\rho_i|I_i|(1+\norm\xi_{\Hh}).
\end{equation}
In view of \(K_0=0\), iterating
\eqref{eq4.9}, and using
$\sum_i|I_i|=r$ and
$\rho_i\leq\gamma_\varepsilon^{1/p}$,
the discrete Gronwall inequality gives
\begin{align}\label{eq4.10}
 \norm {u-\mu}_{\infty,[0,r]}
 \leq C_r\sum_{i=0}^{N-1}\rho_i
 +C_r(1+\norm\xi_{\Hh})
 \sum_{i=0}^{N-1}\rho_i|I_i|
 \leq C_r\gamma_\varepsilon^{1/p} (1+\norm\xi_{\Hh})+
 C_r\gamma_\varepsilon^{1/p}N.
\end{align}
Inserting \eqref{eq4.6} into \eqref{eq4.10} proves \eqref{eq4.4}.
\end{proof}

\begin{proposition}\label{prop4.3}
There exist $r>\tau$, $\eta\in(0,1)$ and \(\varepsilon\in(0,1/4)\), such that for every \(\omega\in\Omega_0\) and
\(\xi\in\Hh\),
\begin{equation}\label{eq4.11}
 \norm{\Phi(r,\omega)\xi}_{\Hh}^2
 \leq\eta\norm\xi_{\Hh}^2+P_{r,\varepsilon}(\omega),
\end{equation}
where $P_{r, \varepsilon}(\omega)=C_{r, \varepsilon}\bigl(1+\tnorm{\mathbf x(\omega)}_{\pvar,[0,r]}^{2p}\bigr)$ for some constant $C_{r, \varepsilon}$.
\end{proposition}

\begin{proof}
For every \(t\in[r-\tau,r]\), in view of Lemmas~\ref{lem4.1} and
\ref{lem4.2}, we have
\[
 \norm{u(t)}\leq
 \norm{\mu(t)}+\norm{u(t)-\mu(t)}
 \leq Ce^{-\delta_*(r-\tau)/2}\norm\xi_{\Hh}
 +C_r\varepsilon\norm\xi_{\Hh}
 +C_{r,\varepsilon}(1+\tnorm{\mathbf x}_{\pvar,[0,r]}^{p}).
\]
Notice that $\|\Phi(r,\omega)\xi\|_{\Hh}
=
\|u_r\|_{\Hh}
=
\sup_{t\in[r-\tau,r]}
\|u(t)\|$. Then
\[
\begin{aligned}
\|\Phi(r,\omega)\xi\|_{\Hh}^2
\leq{}&
\left(
Ce^{-\delta_*(r-\tau)}
+
C_r\varepsilon^2
\right)
\|\xi\|_{\Hh}^2
+
C_{r,\varepsilon}
\left(
1+\tnorm{\mathbf{x}}_{\pvar,[0,r]}^{p}
\right)^2.
\end{aligned}
\]
By first choosing \(r>\tau\) sufficiently large so that
$Ce^{-\delta_*(r-\tau)}
<
\frac14$,
and then choosing \(\varepsilon>0\) sufficiently small so that
$C_r\varepsilon^2
<
\frac14$, we obtain \eqref{eq4.11} for some \(\eta\in(0,1)\).
\end{proof}

\begin{lemma}\label{lem4.3}
Let \(R:\Omega_0\to(0,\infty)\) be a positive tempered random variable,
and assume that the orbit map
$s\mapsto R(\theta_s\omega)$ is continuous for every $\omega\in\Omega_0$. Then for every  $L>0$,
$R_L^{\#}(\omega):=\sup_{|s|\leq L}R(\theta_s\omega)$
is a finite, measurable and tempered random variable.
\end{lemma}

\begin{proof}
For each \(\omega\in\Omega_0\), the map
\(s\mapsto R(\theta_s\omega)\) is continuous on the
compact interval \([-L,L]\). Hence its maximum is
attained, and therefore \(R_L^\#(\omega)<\infty\).
Moreover, the continuity of the orbit map
\(s\mapsto R(\theta_s\omega)\), together with the
density of \(\mathbb Q\) in \(\mathbb R\), implies that
$R_L^\#(\omega)
 =
 \sup_{q\in\mathbb Q\cap[-L,L]}
 R(\theta_q\omega)$.
Since \(\omega\mapsto R(\theta_q\omega)\) is measurable
for every fixed \(q\), the countable supremum
\(R_L^\#\) is measurable.

Fix \(c>0\). By the temperedness of \(R\), for almost
every \(\omega\) there exists \(T=T(c,\omega)>0\) such
that
$R(\theta_{-v}\omega)\leq e^{cv/2}$ for all
$v\geq T$.
If \(t\geq T+L\) and \(|s|\leq L\), then \(t-s\geq T\),
and hence
\[
 e^{-ct}R(\theta_{s-t}\omega)
 \leq
 e^{-ct}e^{c(t-s)/2}
 \leq
 e^{-ct/2+cL/2}.
\]
Using the group property of \(\theta\), we obtain that
\[
\begin{aligned}
 0
 \leq
 e^{-ct}R_L^\#(\theta_{-t}\omega)
 =
 \sup_{|s|\leq L}
 e^{-ct}R(\theta_{s-t}\omega)
 \leq
 e^{-ct/2+cL/2}
 \rightarrow0
\end{aligned}
\]
as \(t\to\infty\), which implies \(R_L^\#\) is tempered,
since \(c>0\) was arbitrary.
\end{proof}

\begin{theorem}\label{thm4.6}
Let $r>\tau$, $\eta\in(0,1)$, and $P_{r,\varepsilon}$ be as in
Proposition~\ref{prop4.3}. Let
$\mathcal D(\Hh)$ denote the universe of tempered random
sets introduced in Definition~\ref{def2.1}. Define
\begin{equation}\label{eq4.15-1}
R_0(\omega)
:=
2
+
2\sum_{j=0}^{\infty}
\eta^j
P_{r,\varepsilon}\bigl(\theta_{-(j+1)r}\omega\bigr).
\end{equation}
Then there exists a positive tempered random variable
$C_r^{\#}(\omega)$, whose orbit map
$s\mapsto C_r^\#(\theta_s\omega)$ is continuous, such that
\begin{equation}\label{eq4.15}
\left\|
\Phi(s,\theta_{-s}\omega)\zeta
\right\|_{\Hh}^2
\leq
C_r^{\#}(\omega)
\left(
1+\|\zeta\|_{\Hh}^2
\right)
\end{equation}
for every $\zeta\in\Hh$ and every $s\in[0,r]$. Moreover, the random
variable
\begin{equation}\label{eq4.16}
 \overline R(\omega):=
 2+2C_r^{\#}(\omega)
 \left(1+\sup_{0\leq s\leq r}R_0(\theta_{-s}\omega)\right)
\end{equation}
is measurable and tempered. The closed random ball
\begin{equation}\label{eq4.17}
 B(\omega):=\{\xi\in\Hh:\norm\xi_{\Hh}^2\leq\overline R(\omega)\}
\end{equation}
is pullback absorbing with respect to $\mathcal D(\Hh)$. More precisely, for
every $D\in\mathcal D(\Hh)$ and every $\omega\in\Omega_0$, there exists a
random time $T_D(\omega)>0$ such that
$\Phi(t,\theta_{-t}\omega)
D(\theta_{-t}\omega)
\subset B(\omega)$
for all $t\geq T_D(\omega)$.

\end{theorem}

\begin{proof}
Since \(r>0\) is fixed and \(\eta\in(0,1)\), we may choose
\(\kappa>0\) sufficiently small such that $\eta e^{\kappa r}<1$.
Fix \(L>0\) and define
$P_{r,L}^{\#}(\omega)
:=
\sup_{|s|\leq L}
P_{r,\varepsilon}(\theta_s\omega)$.
Then by Assumption~\textbf{(A4)} and Lemma~\ref{lem4.3}, we obtain that \(P_{r,L}^{\#}\) is tempered. Moreover, by
the group property, $\sup_{|s|\leq L}
P_{r,\varepsilon}(\theta_{s-(j+1)r}\omega)
=
P_{r,L}^{\#}(\theta_{-(j+1)r}\omega)$.
Thanks to the temperedness of \(P_{r,L}^{\#}\), there exists a positive random
variable \(C_{\kappa,L}(\omega)\) such that
$P_{r,L}^{\#}(\theta_{-(j+1)r}\omega)
\leq
C_{\kappa,L}(\omega)e^{\kappa(j+1)r}$ for all $j\in\mathbb N$.
Consequently,
\[
\sup_{|s|\leq L}
\eta^jP_{r,\varepsilon}(\theta_{s-(j+1)r}\omega)
\leq
C_{\kappa,L}(\omega)
\eta^je^{\kappa(j+1)r}
=
C_{\kappa,L}(\omega)e^{\kappa r}
\left(
\eta e^{\kappa r}
\right)^j.
\]
In view of
$\eta e^{\kappa r}<1$,
the right-hand side is summable with respect to \(j\). Hence the series in \(R_0(\theta_s\omega)\)
converges uniformly on \(|s|\leq L\). Since \(L>0\)
was arbitrary, the series converges locally uniformly in
\(s\in\mathbb R\).
This together with the continuity of $s\mapsto P_{r,\varepsilon}(\theta_{s}\omega)$ yields $s\mapsto R_0(\theta_s\omega)$ is continuous.
In addition, the temperedness of \(P_{r,\varepsilon}\) gives that $R_0$ is tempered. In fact,
fix \(\iota>0\) and choose \(\kappa\in(0,\iota)\) sufficiently
small such that
$\eta e^{\kappa r}<1$.
By the temperedness of \(P_{r,\varepsilon}\), there
exists a finite positive random variable
\(C_\kappa(\omega)\) such that
$P_{r,\varepsilon}(\theta_{-v}\omega)
 \leq
 C_\kappa(\omega)e^{\kappa v}$
for $v\geq0$.
Consequently, for \(t\geq0\),
\begin{align*}
 R_0(\theta_{-t}\omega)
 \leq
 2+
 2C_\kappa(\omega)e^{\kappa t}
 \sum_{j=0}^{\infty}
 \eta^je^{\kappa(j+1)r}
 =
 2+
 \frac{
 2C_\kappa(\omega)e^{\kappa(t+r)}
 }{
 1-\eta e^{\kappa r}
 }.
\end{align*}
It follows that
$e^{-\iota t}R_0(\theta_{-t}\omega)
 \rightarrow0$ as $t\to\infty$.
Since \(\iota>0\) was arbitrary, \(R_0\) is tempered.

Using Corollary~\ref{cor3.5}, we deduce that for $0\leq s\leq r$,
\begin{align*}
\left\|
\Phi(s,\theta_{-s}\omega)\zeta
\right\|_{\mathcal H}
&\leq
\sup_{-\tau\leq \iota\leq s}
\left\|
u(\iota; \theta_{-s}\omega, \zeta)
\right\|\leq C_s\left(1+\|\zeta\|_{\mathcal H}\right)
\left(
1+
\tnorm{
\mathbf{x}(\theta_{-s}\omega)}_{p\text{-var},[0,s]}^{m_s}
\right)\\
&\leq C_r\left(1+\|\zeta\|_{\mathcal H}\right)
\left(
1+
\tnorm{
\mathbf{x}(\omega)}
_{p\text{-var},[-r,0]}^{m_r}
\right),
\end{align*}
where $m_s=5p\lceil s/\tau\rceil$ and
$m_r=5p\lceil r/\tau\rceil$.
Therefore \eqref{eq4.15} follows immediately by choosing $C_r^{\#}(\omega)=C_r
(1+\tnorm{\mathbf x(\omega)}_{\pvar,[-r,0]}^{2m_r})$, which is a continuous tempered random variable.
Lemma~\ref{lem4.3} now implies that
\eqref{eq4.16} is measurable and tempered.

Iterating \eqref{eq4.11} in pullback form, we conclude that
\begin{equation}\label{eq4.19}
 \norm{\Phi(nr,\theta_{-nr}\omega)\xi}_{\Hh}^2
 \leq\eta^n\norm\xi_{\Hh}^2
 +\sum_{j=0}^{n-1}
\eta^j P_{r,\varepsilon}\bigl(\theta_{-(j+1)r}\omega\bigr).
\end{equation}
Let $D\in\mathcal D(\Hh)$ and write
$d_D(\omega):=\sup_{z\in D(\omega)}\norm z_{\Hh}^2$. Choose the same $\kappa>0$
with $\eta e^{\kappa r}<1$. Since $d_D$ is tempered,
\begin{align}
 &\sup_{0\leq s<r}\eta^n d_D(\theta_{-(nr+s)}\omega)
\leq e^{\kappa r}(\eta e^{\kappa r})^n
 \sup_{0\leq s<r}
 \left[e^{-\kappa(nr+s)}d_D(\theta_{-(nr+s)}\omega)\right]
\rightarrow0~~\text{as }n\to\infty.
 \label{eq4.20}
\end{align}
Replacing $\omega$ by $\theta_{-s}\omega$ in \eqref{eq4.19} and using
\eqref{eq4.20} together with \eqref{eq4.15-1}, we find an integer
$n_0=n_0(\omega,D)$ such that, for every $n\geq n_0$, $s\in[0,r)$ and
$\xi\in D(\theta_{-(nr+s)}\omega)$,
\begin{align}\label{eq4.21}
\left\|
 \Phi\bigl(nr,\theta_{-(nr+s)}\omega\bigr)\xi
 \right\|_{\Hh}^{2}
 &\leq
 \eta^{n}\|\xi\|_{\Hh}^{2}
 +
 \sum_{j=0}^{n-1}
 \eta^{j}
 P_{r,\varepsilon}
 \bigl(\theta_{-s-(j+1)r}\omega\bigr)
 \nonumber\\
 &\leq \eta^n d_D(\theta_{-(nr+s)}\omega)
 +
 \sum_{j=0}^{\infty}
 \eta^{j}
 P_{r,\varepsilon}
 \bigl(\theta_{-s-(j+1)r}\omega\bigr)
 \leq R_{0}(\theta_{-s}\omega).
\end{align}
Set $t=nr+s$. By the cocycle identity, we have
$\Phi(t,\theta_{-t}\omega)
 =\Phi(s,\theta_{-s}\omega)
\Phi(nr,\theta_{-(nr+s)}\omega)$.
For $\xi\in D(\theta_{-t}\omega)$, set
$\zeta:=\Phi(nr,\theta_{-(nr+s)}\omega)\xi$. Using \eqref{eq4.15}
and \eqref{eq4.21}, we have
\[\norm{\Phi(s,\theta_{-s}\omega)
\Phi(nr,\theta_{-(nr+s)}\omega)\xi}_{\Hh}^2
\leq C_r^{\#}(\omega)
(1+\norm{\zeta}_{\Hh}^2)
\leq C_r^{\#}(\omega)
(1+R_0(\theta_{-s}\omega)).\]
Therefore, there exists $T_{D}(\omega)>0$ such that for all $t\geq T_{D}(\omega)$,
\[
 \sup_{\xi\in D(\theta_{-t}\omega)}
 \norm{\Phi(t,\theta_{-t}\omega)\xi}_{\Hh}^2
 \leq C_r^{\#}(\omega)
 \left(1+\sup_{0\leq v\leq r}R_0(\theta_{-v}\omega)\right)
 <\overline R(\omega),
\]
which implies that
$\Phi(t,\theta_{-t}\omega)
D(\theta_{-t}\omega)\subset B(\omega)$ for all $t\geq T_{D}(\omega)$.
Consequently, for every tempered family $D$, there exists a random
time $T_D(\omega)>0$ such that
$\Phi(t,\theta_{-t}\omega)D(\theta_{-t}\omega)
\subset B(\omega)$
for all $t\geq T_D(\omega)$.
Hence, $B$ is a pullback absorbing set for the universe of tempered
random sets.
\end{proof}

To obtain uniform estimates along pullback trajectories, we introduce a
positively invariant pathwise absorbing family \(\{\widehat B(\omega)\}_{\omega\in\Omega_0}\).

\begin{lemma}\label{lem4.7}
There exist a pathwise family $\widehat B
 :=
 \{\widehat B(\omega)\}_{\omega\in\Omega_0}$ of nonempty closed bounded subsets in $\Hh$ and a positive
measurable tempered random variable $\widehat R$ such that the following
properties hold.

\begin{enumerate}
\item The family $\widehat B$ is positively invariant, that is for $t\geq0$ and $\omega\in\Omega_0$,
$\Phi(t,\omega)\widehat B(\omega)
 \subset \widehat B(\theta_t\omega)$.

\item The family $\widehat B$ is pullback absorbing with respect to
$\mathcal D(\Hh)$. More precisely, for every $D\in\mathcal D(\Hh)$ and
$\omega\in\Omega_0$, there exists
$T_D(\omega)>0$ such that for any $t\geq T_D(\omega)$,
$\Phi(t,\theta_{-t}\omega)D(\theta_{-t}\omega)
 \subset \widehat B(\omega)$.

\item The family $\widehat B$ satisfies the pathwise bound:
$\sup_{\xi\in\widehat B(\omega)}
 \norm{\xi}_{\Hh}^{2}
 \leq \widehat R(\omega)$ for
 $\omega\in\Omega_0$.
Moreover, for every $L>0$,
$\widehat R_L^{\#}(\omega)
 :=
 \sup_{|s|\leq L}\widehat R(\theta_s\omega)$
is finite, measurable, and tempered.
\end{enumerate}
\end{lemma}

\begin{proof}
Let $r>\tau$ and $\eta\in(0,1)$ be as in
Proposition~\ref{prop4.3}. Define
\[
 \tilde{B}(\omega)
 :=
 \left\{
 \xi\in\mathcal H:
 \|\xi\|_{\mathcal H}^2\leq R_0(\omega)
 \right\},
\]
where $R_0(\omega)$ is given in \eqref{eq4.15-1}. By Theorem \ref{thm4.6}, $R_0$ is a
positive measurable tempered random variable with continuous orbit
map. Thanks to
\begin{align*}
R_0(\theta_r\omega)
=
2+2P_{r,\varepsilon}(\omega)
+2\eta\sum_{j=0}^{\infty}
\eta^jP_{r,\varepsilon}\bigl(\theta_{-(j+1)r}\omega
\bigr)\geq
\eta R_0(\omega)+P_{r,\varepsilon}(\omega),
\end{align*}
from \eqref{eq4.11} we obtain that for every $\xi\in \tilde{B}(\omega)$,
\[
 \|\Phi(r,\omega)\xi\|_{\mathcal H}^2
 \leq
 \eta\|\xi\|_{\mathcal H}^2+P_{r,\varepsilon}(\omega)
 \leq
 \eta R_0(\omega)+P_{r,\varepsilon}(\omega)
 \leq
 R_0(\theta_r\omega).
\]
Consequently,
\begin{equation}\label{eq4.27}
 \Phi(r,\omega)\tilde{B}(\omega)
 \subset \tilde{B}(\theta_r\omega).
\end{equation}
Define
\begin{equation}\label{eq4.28}
 \widehat B(\omega)
 :=
 \overline{
 \bigcup_{0\leq s\leq r}
 \Phi(s,\theta_{-s}\omega)
 \tilde{B}(\theta_{-s}\omega)
 }^{\mathcal H},
\end{equation}
where the superscript $\mathcal H$ denotes closure in $\mathcal H$.
For $0\leq s\leq r$,
let $\psi_{1}\in \tilde{B}(\theta_{-s}\omega)$ and $\psi_{2}:=\Phi(s,\theta_{-s}\omega)
\psi_{1}$.
For any $t\geq0$ and $s\in [0,r)$, there exist $q\in[0,r)$ and $n\in\mathbb N_0$ such that
$t+s=nr+q$. By the cocycle identity and repeated application of
\eqref{eq4.27}, we have
\begin{align*}
 \Phi(t,\omega)\psi_{2}
 &=
 \Phi(q,\theta_{nr-s}\omega)
 \Phi(nr,\theta_{-s}\omega)\psi_{1}
 \in
 \Phi(q,\theta_{t-q}\omega)
 \tilde{B}(\theta_{t-q}\omega)\subset
 \widehat B(\theta_t\omega).
\end{align*}
Since $\Phi(t,\omega)$ is continuous with respect to the initial value,
this inclusion extends to the closure in \eqref{eq4.28}. Therefore,
$\Phi(t,\omega)\widehat B(\omega)
 \subset\widehat B(\theta_t\omega)$,
$t\geq0.$

Let $D\in\mathcal D(\mathcal H)$. By \eqref{eq4.21}, there exists
$n_0=n_0(\omega,D)\in\mathbb N$ such that, whenever
$t=nr+s$ with $n\geq n_0$ and $s\in[0,r)$, the cocycle identity
yields
\begin{align*}
 \Phi(t,\theta_{-t}\omega)D(\theta_{-t}\omega)
 &=
 \Phi(s,\theta_{-s}\omega)
 \Phi\bigl(nr,\theta_{-(nr+s)}\omega\bigr)
 D\bigl(\theta_{-(nr+s)}\omega\bigr)
 \subset
 \Phi(s,\theta_{-s}\omega)\tilde{B}(\theta_{-s}\omega)
 \subset\widehat B(\omega).
\end{align*}
Thus, $\widehat B$ pullback absorbs every $D\in\mathcal D(\mathcal H)$.

By Theorem~\ref{thm4.6}, $R_0$ is tempered and the map
$s\mapsto R_0(\theta_s\omega)$ is continuous. Hence,
Lemma~\ref{lem4.3} implies that, for every $L>0$,
$
 R_{0,L}^{\#}(\omega)
 :=
 \sup_{|s|\leq L}R_0(\theta_s\omega)
$
is a finite, measurable, and tempered random variable. Applying
Corollary~\ref{cor3.5} on intervals of length at most $r$ and using the rough-path shift relation, we obtain that
\begin{align}\label{eq4.30}
 \sup_{0\leq s\leq r}
 \sup_{\xi\in \tilde{B}(\theta_{-s}\omega)}
 \norm{\Phi(s,\theta_{-s}\omega)\xi}_{\Hh}^{2}
 \leq
 C\bigl(1+R_{0,r}^{\#}(\omega)\bigr)
 \left(
 1+\tnorm{\mathbf x(\omega)}_{\pvar,[-r,r]}^{2m_r}
 \right).
\end{align}
Define
$\widehat R(\omega)
 :=
 1+
 C\bigl(1+R_{0,r}^{\#}(\omega)\bigr)
 (
 1+\tnorm{\mathbf x(\omega)}_{\pvar,[-r,r]}^{2m_r})$.
Since $\widehat B(\omega)$ is the closure of the union in
\eqref{eq4.28}, the preceding estimate \eqref{eq4.30} yields
$\sup_{\xi\in\widehat B(\omega)}
 \norm{\xi}_{\Hh}^{2}
 \leq\widehat R(\omega)$.
Moreover, $\widehat R$ is a positive measurable tempered random
variable with a continuous orbit map. Therefore, another application
of Lemma~\ref{lem4.3} shows that, for every $L>0$,
$\widehat R_L^{\#}(\omega)
 :=
 \sup_{|s|\leq L}\widehat R(\theta_s\omega)$
is finite, measurable, and tempered.
\end{proof}

\subsection{Pullback asymptotic compactness}

Choose $\rho\in C^1([0,\infty);[0,1])$ such that
\[
 \rho(s)=0\quad(0\leq s\leq1),\qquad
 \rho(s)=1\quad(s\geq2),\qquad
 \norm{\rho'}_\infty\leq C_\rho,
\]
and set $\chi=\rho^2$, so $\norm{\chi'}_\infty\leq C_\chi:=2C_\rho$. For \(K\geq1\), define the diagonal operators
\[
 (S_Ku)_i:=
 \rho\left(\frac{|i|}{K}\right)
 u_i,~~~
 (C_Ku)_i:=\chi\left(\frac{|i|}{K}\right)
 u_i,
\]
so that \(C_K=S_K^2\).
Moreover,
\begin{equation}\label{eq4.31}
 V_K(u):=\ip{u}{C_Ku}=\norm{S_Ku}^2
 =\sum_{i\in\Z}
 \chi\left(\frac{|i|}{K}\right)|u_i|^2.
\end{equation}
Let $[S_K,\mathcal L]:=S_K\mathcal L-\mathcal L S_K$ denote the commutator of \(S_K\) and \(\mathcal L\). For every $R\geq0$, define the projections $P_R,Q_R:U\to U$ by
\begin{equation}\label{eq4.32-0}
(P_Ru)_i
:=
\mathbf 1_{\{|i|\leq R\}}u_i,
\qquad
(Q_Ru)_i
:=
\mathbf 1_{\{|i|>R\}}u_i,
\qquad i\in\mathbb Z.
\end{equation}
Then
\begin{equation}\label{eq4.32}
 \norm{Q_{2K}u}^2
 \leq V_K(u)\leq\norm{Q_Ku}^2.
\end{equation}

\begin{lemma}\label{lem4.8}
For every \(K\geq1\) and \(u,v\in U\), the following
estimates hold:
\begin{align*}
-\ip{\mathcal Lu}{C_Ku}
 &\leq\frac{\nu C_\chi \norm u^2}{2K}
 \bigg(\sup_{i}\sum_{j}w_{ij}|i-j|\bigg),\\
\norm{[S_K,\mathcal L]u}
 &\leq\frac{\nu C_\rho \norm u }{K}
 \bigg(\sup_{i}\sum_{j}w_{ij}|i-j|\bigg),
 \\
\norm{S_K\mathcal Lu}
 &\leq\norm{\mathcal L}\,V_K(u)^{1/2}
 +\frac{\nu C_\rho \norm u}{K}
 \bigg(\sup_{i}\sum_{j}w_{ij}|i-j|
 \bigg),\\
 2\ip{F(u,v)}{C_Ku}
 &\leq(\lambda+L_0)V_K(u)
 +L_1V_K(v)+(\lambda-L_0-L_1)^{-1}
 \norm{Q_KF_0}^2.
\end{align*}
\end{lemma}

\begin{proof}
In view of \eqref{eq2.14}, it holds that
\begin{align}
 \ip{\mathcal Lu}{C_Ku}
 &=\frac{\nu}{2}\sum_{i,j}
 w_{ij}(u_i-u_j)
 \left[\chi\left(\frac{|i|}{K}\right)
 u_i-\chi\left(\frac{|j|}{K}\right)
 u_j\right]\notag\\
 &=\frac{\nu}{4}\sum_{i,j}w_{ij}
 \left[\chi\left(\frac{|i|}{K}\right)+
 \chi\left(\frac{|j|}{K}\right)\right]
 (u_i-u_j)^2
 +\frac{\nu}{4}\sum_{i,j}
 w_{ij}
 \left[\chi\left(\frac{|i|}{K}\right)-
 \chi\left(\frac{|j|}{K}\right)\right](u_i^2-u_j^2).
 \nonumber
\end{align}
The mean-value theorem yields
$|\chi(|i|/K)-
\chi(|j|/K)|
 \leq\frac{C_\chi}{K}\bigl||i|-|j|\bigr|
 \leq\frac{C_\chi}{K}|i-j|$.
Hence, using \(w_{ij}=w_{ji}\), we have
\begin{align*}
-\ip{\mathcal Lu}{C_Ku}
 &\leq\frac{\nu C_\chi}{4K}
 \sum_{i,j}w_{ij}|i-j|\,|u_i^2-u_j^2|
 \leq\frac{\nu C_\chi}{4K}
 \sum_{i,j}w_{ij}|i-j|(u_i^2+u_j^2)\\
 &=\frac{\nu C_\chi}{2K}
 \sum_i u_i^2\sum_jw_{ij}|i-j|
 \leq\frac{\nu C_\chi \norm u^2}{2K}
 \bigg(\sup_{i}\sum_{j}w_{ij}|i-j|\bigg).
\end{align*}

A direct calculation gives $([S_K,\mathcal L]u)_i
 =\nu\sum_jw_{ij}(
 \rho(|j|/K)-\rho(|i|/K))u_j$.
Then by Cauchy-Schwarz inequality and $|\rho(|i|/K)-\rho(|j|/K)|
\leq
\frac{C_\rho}{K}|i-j|$, we deduce that
\begin{align*}
 \norm{[S_K,\mathcal L]u}^2
 &\leq\nu^2\sum_i
 \left(\sum_jw_{ij}
 \left|\rho\left(\frac{|i|}{K}\right)
 -\rho\left(\frac{|j|}{K}\right)
 \right|\right)
 \left(\sum_jw_{ij}\left|
 \rho\left(\frac{|i|}{K}\right)
 -\rho\left(\frac{|j|}{K}\right)
 \right||u_j|^2\right)\\
 &\leq\frac{\nu^2C_\rho }{K}
 \bigg(\sup_{i}\sum_{j}w_{ij}|i-j|\bigg)
 \sum_j|u_j|^2\sum_iw_{ij}
 \left|\rho\left(\frac{|i|}{K}\right)-
 \rho\left(\frac{|j|}{K}\right)\right|\\
 &\leq\frac{\nu^2C_\rho^2 \norm u^2}{K^2}
 \bigg(\sup_{i}\sum_{j}w_{ij}|i-j|\bigg)^{2}.
\end{align*}
The third estimate follows from
$S_K\mathcal L u=\mathcal L S_Ku+[S_K,\mathcal L]u$ and
$\norm{S_Ku}=V_K(u)^{1/2}$.

Finally, recall that  $F(u,v)=(f_{i}(u_{i},v_{i}))_{i\in \mathbb{Z}}$ and
$|f_i(r,s)-f_i(0,0)|\leq L_0|r|+L_1|s|$. Applying the Cauchy--Schwarz and Young inequalities yields
\begin{align}\label{eq4.35--}
 &2\ip{F(u,v)-F_0}{C_Ku}
 \leq2L_0\sum_i
 \chi\left(\frac{|i|}{K}\right)|u_i|^2
 +2L_1\sum_i
 \chi\left(\frac{|i|}{K}\right)|v_i||u_i|\nonumber\\
 &\leq2L_0V_K(u)+2L_1V_K(v)^{1/2}
 V_K(u)^{1/2}
 \leq(2L_0+L_1)V_K(u)+L_1V_K(v).
\end{align}
Since $C_KF_0=C_KQ_KF_0$, it follows from Young's inequality that
\begin{align}\label{eq4.36--}
 2\ip{F_0}{C_Ku}
 &\leq
 2\norm{S_KQ_KF_0}\norm{S_Ku}
 \leq
 2\norm{Q_KF_0}V_K(u)^{1/2}\nonumber\\
 &\leq
 (\lambda-L_0-L_1)V_K(u)+
 (\lambda-L_0-L_1)^{-1}\norm{Q_KF_0}^{2}.
\end{align}
Combining \eqref{eq4.35--} and \eqref{eq4.36--}
proves the fourth estimate and completes the proof.
\end{proof}

Let \(u\) be the controlled solution of
\eqref{eq:vector}. For each \(K\geq1\), define the
tail Lyapunov--Krasovskii functional by
\begin{equation}\label{eq4.34-3}
 \Psi_K(t)
 :=
 V_K(u(t))
 +
 L_1e^{\sigma\tau}
 \int_{t-\tau}^{t}
 e^{-\sigma(t-s)}V_K(u(s))\,ds,
 \qquad t\geq0.
\end{equation}
Here, when \(L_1>0\), \(\sigma>0\) is chosen sufficiently
small that
$L_1e^{\sigma\tau}<\lambda-L_0$;
when \(L_1=0\), any
\(\sigma\in(0,\lambda-L_0)\) may be used.

\begin{lemma}\label{lem4.11}
Let \(u\) be the controlled solution of
\eqref{eq:vector}.
For the above choice of \(\sigma\), set
$\delta_*:=\min\left\{
 \lambda-L_0-L_1e^{\sigma\tau},
 \sigma
 \right\}>0$.
Then, for every
$K\in\mathbb N$ and $t>0$,
 \begin{align*}
 &2\ip{-\mathcal Lu(t)-\lambda u(t)
 +F(u(t),u(t-\tau))}{C_Ku(t)}
 +L_1e^{\sigma\tau}\frac{d}{dt}
 \int_{t-\tau}^{t}e^{-\sigma(t-s)}V_K(u(s))\,ds\\
 &\leq-\delta_*\Psi_K(t)
 +\frac{\nu C_\chi \norm{u(t)}^2}{K}
 \bigg(\sup_{i}\sum_{j}w_{ij}|i-j|\bigg)
 +(\lambda-L_0-L_1)^{-1}\norm{Q_KF_0}^2.
 \end{align*}
\end{lemma}

\begin{proof}
Using Lemma~\ref{lem4.8}, we have
\begin{align}\label{eq4.34}
 &2\ip{-\mathcal Lu-\lambda u+F(u,u(t-\tau))}{C_Ku}\leq-2\lambda V_K(u)
 +(\lambda+L_0)V_K(u)\nonumber\\
 &+L_1V_K(u(t-\tau))
 +\frac{\nu C_\chi \norm u^2 }{K}
 \bigg(\sup_{i}\sum_{j}w_{ij}|i-j|\bigg)
 +(\lambda-L_0-L_1)^{-1}
 \norm{Q_KF_0}^2\\
 &=-(\lambda-L_0)V_K(u)+L_1V_K(u(t-\tau))
 +\frac{\nu C_\chi \norm u^2 }{K}
 \bigg(\sup_{i}\sum_{j}w_{ij}|i-j|\bigg)
 +(\lambda-L_0-L_1)^{-1}\norm{Q_KF_0}^2.
\nonumber
\end{align}
Set $M_K(t):=\int_{t-\tau}^{t}
e^{-\sigma(t-s)}V_K(u(s))ds$. Then
\begin{equation}\label{eq4.35}
 \dot M_K(t)=V_K(u(t))-e^{-\sigma
 \tau}V_K(u(t-\tau))
 -\sigma M_K(t).
\end{equation}
Multiplying \eqref{eq4.35} by $L_1e^{\sigma\tau}$ and adding it to
\eqref{eq4.34}, we obtain that
\begin{align*}
 &2\ip{-\mathcal Lu-\lambda u+F(u,u(t-\tau))}{C_Ku}
 +L_1e^{\sigma\tau}\dot M_K\leq-(\lambda-L_{0}-L_1e^{\sigma\tau})
 V_K(u)\\
 &-L_1e^{\sigma\tau}\sigma M_K
 +\frac{\nu C_\chi \norm u^2}{K}
 \bigg(\sup_{i}\sum_{j}w_{ij}|i-j|\bigg)
 +(\lambda-L_0-L_1)^{-1}\norm{Q_KF_0}^2\\
 &\leq-\delta_*\bigl(V_K(u)+
 L_1e^{\sigma\tau} M_K\bigr)
 +\frac{\nu C_\chi \norm u^2}{K}
 \bigg(\sup_{i}\sum_{j}w_{ij}|i-j|\bigg)
 +(\lambda-L_0-L_1)^{-1}\norm{Q_KF_0}^2,
\end{align*}
where $\delta_*=\min\{
\lambda-L_0-L_1e^{\sigma\tau},
\sigma\}$.
\end{proof}

For each $v\in U$, define the linear and bilinear forms
\begin{align}\label{eq4.36-1}
I_K(v)z
 :=
 \ip{h(v)z}{C_Kv},~~
 J_K(v)(z_1,z_2)
 :=
 \ip{C_Kh(v)z_1}{h(v)z_2}
 ~~~~z,z_1,z_2\in U.
\end{align}
Let $u$ be the solution of \eqref{eq:vector}. Since $u$ is
controlled by $x$ with Gubinelli derivative $u'=h(u)$, the
path $I_K(u)$ is also controlled by $x$, with Gubinelli
derivative
\begin{align}\label{eq4.36}
 \bigl(I_K(u)\bigr)'(z_1,z_2)
 = DI_K(u)[h(u)z_1]z_2
 = \ip{ Dh(u)[h(u)z_1]z_2
 }{C_Ku}+ \ip{h(u)z_2}{C_Kh(u)z_1}.
\end{align}
Moreover, for $z,z_{1}, z_{2}\in U$,
 \begin{align*}
 I_K(u)z=I_K(u)Q_K z,~~~
 R^{I_K(u)}_{s,t} z=R^{I_K(u)}_{s,t}Q_K z,~~~ J_K(u)(z_{1}, z_{2})=J_K(u)(Q_Kz_{1}, Q_Kz_{2}),\\
 \bigl(I_K(u)\bigr)'(z_{1}, z_{2})
 =
 \bigl(I_K(u)\bigr)'
 (Q_K z_{1}, z_{2})
 =
 \bigl(I_K(u)\bigr)'
 (z_{1}, Q_K z_{2})
 =
 \bigl(I_K(u)\bigr)'(Q_Kz_{1}, Q_Kz_{2}).
 \end{align*}

Since
$V_K(u)=\ip{u}{C_Ku},
$
the first and second Fr\'echet derivatives of $V_K$ are given, for
\(u,z,z_1,z_2\in U\), by
$DV_K(u)[z]
 =
 2\ip{z}{C_Ku}$,
$D^2V_K(u)[z_1,z_2]
 =
 2\ip{C_Kz_1}{z_2}$.
Moreover, since $\mathbf x$ is weakly geometric,
$[\mathbf x]_{s,t}
 =
 x_{s,t}\otimes x_{s,t}
 -2\operatorname{Sym}\mathbb X_{s,t}
 =0$.
Therefore, applying the rough chain rule to \(V_K(u)\) and using
\eqref{eq:vector} and \eqref{eq:local-integral-0}, we obtain that for
\(a<t\),
\begin{align}\label{eq4.38-0}
 V_K(u(t))-V_K(u(a))
 =
 2\int_a^t
 \ip{G(u(s),u(s-\tau))}{C_Ku(s)}\,ds
 +
 2\int_a^t I_K(u(s))\,d\mathbf x(s).
\end{align}
Set $M_K(t)
 :=
 \int_{t-\tau}^{t}
 e^{-\sigma(t-s)}V_K(u(s))\,ds$ and
recall that
$\Psi_K(t)=
V_K(u(t))+L_1e^{\sigma\tau}M_K(t)$. Then by Lemma~\ref{lem4.11} and \eqref{eq4.38-0}, we have for $a<t$,
\begin{align}\label{eq4.38} e^{\delta_*t}\Psi_K(t)& \leq e^{\delta_*a}\Psi_K(a) +\int_a^te^{\delta_*s} \bigg(\frac{\nu C_\chi \norm{u(s)}^2}{K} \bigg(\sup_{i}\sum_{j}w_{ij}|i-j|\bigg) \nonumber\\ & +(\lambda-L_0-L_1)^{-1} \norm{Q_KF_0}^2\bigg)ds+ 2\int_a^te^{\delta_*s} I_K(u(s))dx(s).
\end{align}

\begin{lemma}\label{lem4.13}
Let \(Y\) be an operator-valued path controlled by \(x\) on
\([a,b]\), with Gubinelli derivative \(Y'\), and let \(Q_K\)
be the orthogonal projection defined above. Assume that
\[
 Y(s)=Y(s)Q_K,\qquad
 R^Y_{s,t}=R^Y_{s,t}Q_K,\qquad
 Y'(s)=Y'(s)(Q_K\otimes Q_K)
\]
for all \(a\leq s\leq t\leq b\). Then, for every
\(a\leq s<t\leq b\),
\begin{align}\label{eq4.41}
 &\left\|
 \int_s^tY(r)\,dx(r)
 -Y(s)Q_Kx_{s,t}
 -Y'(s)(Q_K\otimes Q_K)\mathbb X_{s,t}
 \right\|
 \notag\\
 &\quad\leq
 c_p
 \bigg(
 \tnorm{R^Y}_{\qvar,[s,t]}
 \tnorm{Q_Kx}_{\pvar,[s,t]}
 +
 \tnorm{Y'}_{\pvar,[s,t]}
 \tnorm{(Q_K\otimes Q_K)\mathbb X}_{\qvar,[s,t]}
 \bigg).
\end{align}
Moreover,
\begin{align}\label{eq4.42}
 \tnorm{(Q_K\otimes Q_K)\mathbb X}_{\qvar,[s,t]}
 \leq
 \min\bigl\{
 \tnorm{(Q_K\otimes\operatorname{Id}_U)
 \mathbb X}_{\qvar,[s,t]},
 \tnorm{(\operatorname{Id}_U\otimes Q_K)
 \mathbb X}_{\qvar,[s,t]}
 \bigr\}.
\end{align}
\end{lemma}

\begin{proof}
For $a\leq s\leq r \leq t\leq b$, set  $\Xi_{s,t}:=Y(s)x_{s,t}+Y'(s)\mathbb X_{s,t}$ and $\delta\Xi_{s,r,t}:=\Xi_{s,t}-\Xi_{s,r}
-\Xi_{r,t}$.
By Chen's relation \eqref{eq2.1}, we obtain
\begin{align}\label{eq4.43}
\delta\Xi_{s,r,t}
&=
\bigl(Y(s)-Y(r)\bigr)x_{r,t}
+
\bigl(Y'(s)-Y'(r)\bigr)\mathbb X_{r,t}
+
Y'(s)\bigl(x_{s,r}\otimes x_{r,t}\bigr)\nonumber\\
&=-R^Y_{s,r}x_{r,t}
 +\bigl(Y'(s)-Y'(r)\bigr)\mathbb X_{r,t}\nonumber\\
&=-R^Y_{s,r}Q_Kx_{r,t}
 +\bigl(Y'(s)-Y'(r)\bigr)(Q_K\otimes Q_K)\mathbb X_{r,t}.
\end{align}
The first term on the right-hand side of \eqref{eq4.43}
combines a $q$-variation increment of $R^Y$ with a
$p$-variation increment of $Q_Kx$, whereas the second combines
a $p$-variation increment of $Y'$ with a $q$-variation
increment of $(Q_K\otimes Q_K)\mathbb X$. Since $q=p/2$,
both terms have sewing exponent
$\frac{1}{q}+\frac{1}{p}=\frac{3}{p}>1$.
Hence, \eqref{eq4.41} follows immediately from the sewing lemma.
Finally, since \(Q_K\) is an orthogonal projection on \(U\),
we have \(\norm{Q_K}_{\mathcal L(U)}=1\) and
\[
 Q_K\otimes Q_K
 =
 (\operatorname{Id}_U\otimes Q_K)
 (Q_K\otimes\operatorname{Id}_U)
 =
 (Q_K\otimes\operatorname{Id}_U)
 (\operatorname{Id}_U\otimes Q_K).
\]
Hence, for every
\(\zeta\in U\widehat\otimes_\pi U\),
it holds that
\[
 \norm{(Q_K\otimes Q_K)\zeta}
 \leq
 \min\bigl\{
 \norm{(Q_K\otimes\operatorname{Id}_U)\zeta},
 \norm{(\operatorname{Id}_U\otimes Q_K)\zeta}
 \bigr\}.
\]
Taking the \(q\)-variation seminorm yields
\eqref{eq4.42}.
\end{proof}

To derive uniform spatial-tail estimates for the solution $u$ of
\eqref{eq:vector}, we impose the following tail condition on the
driving rough path \(\mathbf x=(x,\mathbb X)\). This condition will
be used to establish pullback asymptotic compactness.

\begin{enumerate}
\item[\textbf{(A5)}]
For $K\geq0$ and every bounded interval $I$, set
\begin{align*}
 \Gamma_K(\omega;I)
 &:=
 \tnorm{Q_K x(\omega)}_{\pvar,I}
 +\tnorm{(Q_K\otimes\operatorname{Id}_U)
       \mathbb X(\omega)}_{\qvar,I}
 +\tnorm{(\operatorname{Id}_U\otimes Q_K)\mathbb X(\omega)}_{\qvar,I}.
\end{align*}
Assume that, for every \(\omega\in\Omega_0\) and every
bounded interval \(I\), $\Gamma_K(\omega;I)\rightarrow0$ as $K\rightarrow\infty$.
\end{enumerate}

\begin{lemma}\label{lem4.14}
Let $\widehat B$ be the positively invariant pathwise
pullback absorbing family constructed in
Lemma~\ref{lem4.7}, and let \(u\) be the controlled solution of
\eqref{eq:vector}. For each $s\in I_n=[n,n+1]$, define
$u_s(\vartheta):=u(s+\vartheta)$ for
$\vartheta\in[-\tau,0]$.
Suppose that $u_s\in\widehat B(\theta_s\omega)$ for $s\in I_n$.
Then there exists a tempered random variable
$A(\omega)\geq1$, independent of $K$ and $n$, such that, for every
$K\in\mathbb N$,
\begin{equation}\label{eq:rough-tail}
 \left|
 \int_n^{n+1}
 e^{\delta_*s}I_K(u(s))\,dx(s)
 \right|
 \leq C
 e^{\delta_*(n+1)}
 A(\theta_n\omega)
 \gamma_{\lfloor K/2\rfloor}(\theta_n\omega),
\end{equation}
where $\gamma_{\lfloor K/2\rfloor}(\theta_n\omega)
 :=\Gamma_{\lfloor K/2\rfloor}(\theta_n\omega;[0,1])
 +\Gamma_{\lfloor K/2\rfloor}(\theta_n\omega;[0,1])^2$ and $C$ is independent of $K$.
\end{lemma}

\begin{proof}
By Lemmas~\ref{lem3.4} and \ref{lem4.7}, there exists a
tempered random variable $A_0(\omega)\geq1$, independent of $K$,
such that
\begin{equation}\label{eq4.46}
 \begin{aligned}
 &1+\sup_{s\in I_n}\norm{u_s}_{\Hh}
 +\tnorm{ u}_{\pvar,I_n}+\tnorm{R^u}_{\qvar,I_n}
 \leq A_0(\theta_n\omega).
 \end{aligned}
\end{equation}
Thanks to the definition of $I_K$ in \eqref{eq4.36-1}, Cauchy--Schwarz inequality and
Assumption~\textbf{(A3)}, we have
\begin{align}\label{eq4.46-0}
\norm{I_K(u)}_{\infty,I_n}
 \leq C\norm{u}_{\infty,I_n}
 \leq C A_0(\theta_n\omega).
\end{align}
In view of \eqref{eq4.36} and Assumption \textbf{(A3)}, it holds that $\norm{(I_K(u))'}_{\infty,I_n}
\leq C(1+\norm{u}_{\infty,I_n})$ and
\begin{align}\label{eq4.39}
\tnorm{(I_K(u))'}_{\pvar,I_{n}}
 \leq
 C(1+\norm{u}_{\infty,I_{n}})
 \tnorm{u}_{\pvar,I_{n}}
 \leq
 C A_0(\theta_n\omega)^{2}.
\end{align}
Moreover, we notice that $h(u(t))-h(u(s))
 =
 Dh(u(s))[u_{s,t}]
 +
 \mathcal R^h_{s,t}$ with
$\mathcal R^h_{s,t}
 =
 \int_0^1(1-r)
 D^2h(u(s)+ru_{s,t})
 [u_{s,t},u_{s,t}]\,dr$ and
 $u_{s,t}=
 h(u(s))x_{s,t}+R^u_{s,t}$. Then for every $z\in U$,
\begin{align*}
 \bigl(I_K(u(t))-I_K(u(s))\bigr)z
 &=
 \ip{(h(u(t))-h(u(s)))z}{C_Ku(t)}
 +
 \ip{h(u(s))z}{C_Ku_{s,t}}\\
 &=
 \bigl(I_K(u)\bigr)'(s)(x_{s,t},z)
 +
 R^{I_K(u)}_{s,t}z,
\end{align*}
where
\begin{equation}\label{eq4.40-0}
 \norm{R^{I_K(u)}_{s,t}}
 \leq
 C(1+\norm{u}_{\infty,I_{n}})
 \left(
 \norm{R^u_{s,t}}
 +
 \norm{u_{s,t}}^2
 \right).
\end{equation}
It follows that
\begin{equation}\label{eq4.40}
 \tnorm{R^{I_K(u)}}_{\qvar,I_{n}}
 \leq
 C(1+\norm{u}_{\infty,I_{n}})
 \left(
 \tnorm{R^u}_{\qvar,I_{n}}
 +
 \tnorm{u}_{\pvar,I_{n}}^2
 \right)\leq
 C A_0(\theta_n\omega)^{3}.
\end{equation}

 Set $Y_K(s):=e^{\delta_*s}I_K(u(s))$. Then \(Y_K\) is controlled by \(x\), with Gubinelli derivative
$Y_K'(s)=e^{\delta_*s}(I_K(u))'(s)$, and
\begin{align*}
 R^{Y_K}_{s,t}
 =e^{\delta_*s}R^{I_K(u)}_{s,t}
 +(e^{\delta_*t}-e^{\delta_*s})I_K(u(t)).
\end{align*}
By \eqref{eq4.46-0}-\eqref{eq4.39} and \eqref{eq4.40},
we deduce that
\begin{align}
|Y_K(n)Q_{\lfloor K/2\rfloor}x_{n,n+1}|
&\leq \|Y_K(n)\|
 \tnorm{Q_{\lfloor K/2\rfloor}x}_{\pvar,I_n}\nonumber\\
 &\leq
 Ce^{\delta_*(n+1)}A_0(\theta_n\omega)
 \tnorm{Q_{\lfloor K/2\rfloor}x}_{\pvar,I_n},
 \label{eq4.47}\\
|Y_K'(n)(Q_{\lfloor K/2\rfloor}\otimes Q_{\lfloor K/2\rfloor})\mathbb X_{n,n+1}|
&\leq
\|Y_K'(n)\|
\|(Q_{\lfloor K/2\rfloor}\otimes Q_{\lfloor K/2\rfloor})\mathbb X_{n,n+1}\|\nonumber\\
 &\leq C
e^{\delta_*(n+1)}A_0(\theta_n\omega)^2
 \tnorm{(Q_{\lfloor K/2\rfloor}\otimes \operatorname{Id}_U)\mathbb X}_{\qvar,I_n},
 \label{eq4.48}\\
\tnorm{Y_K'}_{\pvar,I_n}
 +\tnorm{R^{Y_K}}_{\qvar,I_n}
 &\leq Ce^{\delta_*(n+1)} A_0(\theta_n\omega)^3.
 \label{eq4.49}
\end{align}
Moreover, by $Q_{K}=Q_{K}Q_{\lfloor K/2\rfloor}$, it holds that
$Y_K(s)=Y_K(s)Q_{\lfloor K/2\rfloor}$,
$R^{Y_K}_{s,t}=R^{Y_K}_{s,t}
Q_{\lfloor K/2\rfloor}$ and
$Y_K'(s)=Y_K'(s)(Q_{\lfloor K/2\rfloor}\otimes Q_{\lfloor K/2\rfloor}).$
Applying Lemma~\ref{lem4.13}, and using \eqref{eq4.42}, \eqref{eq4.47}-\eqref{eq4.49} and Assumption~\textbf{(A4)},  we have
\begin{align*}
&\bigg|
\int_n^{n+1}Y_K(s)dx(s)
\bigg|
\leq
\left|
Y_K(n)Q_{\lfloor K/2\rfloor}x_{n,n+1}
\right|+
\left|
Y_K'(n)
(Q_{\lfloor K/2\rfloor}\otimes Q_{\lfloor K/2\rfloor})
\mathbb X_{n,n+1}
\right|\\
&+
c_p
\tnorm{R^{Y_K}}_{\qvar,I_n}
\tnorm{Q_{\lfloor K/2\rfloor}x}_{\pvar,I_n}+
c_p
\tnorm{Y_K'}_{\pvar,I_n}
\tnorm{
(Q_{\lfloor K/2\rfloor}\otimes Q_{\lfloor K/2\rfloor})\mathbb X
}_{\qvar,I_n}\\
&\leq C e^{\delta_*(n+1)}
 \gamma_{\lfloor K/2\rfloor}(\theta_n\omega)
 A_0(\theta_n\omega)^3.
\end{align*}
Finally, the desired estimate \eqref{eq:rough-tail}
follows immediately by setting
\(A(\omega):=A_0(\omega)^3\), which is tempered and
independent of \(n\) and \(K\).
\end{proof}

\begin{theorem}\label{thm4.15}
Let $D\in\mathcal D(\mathcal H)$. For every $\varepsilon>0$ and
$\omega\in\Omega_0$, there exist
$K_0=K_0(\varepsilon,\omega,D)$ and
$T_0=T_0(\varepsilon,\omega,D)$ such that
\begin{equation}\label{eq4.52}
 \sup_{\xi\in D(\theta_{-t}\omega)}
 \sup_{\vartheta\in[-\tau,0]}
 \sum_{|i|\geq2K_0}
 \left|
 \bigl(\Phi(t,\theta_{-t}\omega)\xi\bigr)(\vartheta)_i
 \right|^2
 <\varepsilon
\end{equation}
for all $t\geq T_0$.
\end{theorem}

\begin{proof}
Fix \(t>0\) and \(\xi\in D(\theta_{-t}\omega)\), and define
$v_t(s):=u(t+s;\theta_{-t}\omega,\xi)$ for $-t-\tau\leq s\leq0$.
By Assumption~\textbf{(A4)}, we have
$\mathbf x_{t+s_1,t+s_2}(\theta_{-t}\omega)
 =
 \mathbf x_{s_1,s_2}(\omega)$.
Then, it follows from  \eqref{eq:vector} that \(v_t\) solves \[
 dv_t(s)
 =
 \left[
 -\mathcal L v_t(s)
 -\lambda v_t(s)
 +F\bigl(v_t(s),v_t(s-\tau)\bigr)
 \right]\,ds
 +
 h\bigl(v_t(s)\bigr)\,dx(s;\omega),
 ~~~~s\in[-t,0].
\]

Let $N>\tau$ be an integer to be chosen later. Define $(v_t)_{-N}(s):=v_t(-N+
s)$ for $s\in [-\tau,0]$.
Then from Lemma \ref{lem4.7}, there exists
$T_D(\theta_{-N}\omega)>0$ such that for  $t-N\geq T_D(\theta_{-N}\omega)$,
\[(v_t)_{-N}=
\Phi(t-N,\theta_{-t}\omega)\xi=
\Phi(t-N,\theta_{-(t-N)}
(\theta_{-N}\omega))\xi\in
\widehat B(\theta_{-N}\omega),\]
which implies $\norm{(v_t)_{-N}}_{\Hh}^2
 \leq
 \widehat R(\theta_{-N}\omega)$.
By the positive invariance of $\widehat B$, we have for any $\xi\in D(\theta_{-t}\omega)$ and $s\in [-N,0]$,
\begin{align}\label{eq4.53}
 (v_t)_s
 =\Phi(t+s,\theta_{-t}\omega)\xi
 &=\Phi(s+N,\theta_{-N}\omega)
 \Phi(t-N,\theta_{-t}\omega)\xi
 \nonumber\\
 &=
 \Phi(s+N,\theta_{-N}\omega)(v_t)_{-N}
 \in
 \widehat B(\theta_s\omega).
\end{align}
For $s\in[-\tau,0]$, applying
\eqref{eq4.34-3} and \eqref{eq4.38} to $v_t$ gives
\begin{align}
 \Psi_K(s)
 &\leq
 e^{-\delta_*(s+N)}\Psi_K(-N)
 +\frac{\nu C_\chi }{K}
 \bigg(\sup_{i}\sum_{j}w_{ij}|i-j|\bigg)
 \int_{-N}^{s}
 e^{-\delta_*(s-r)}\norm{v_t(r)}^2\,dr
 \notag\\
 &\quad+
 \frac{\norm{Q_KF_0}^2}
 {(\lambda-L_0-L_1)\delta_*}
 +e^{-\delta_*s}
 \left|
 2\int_{-N}^{s}
 e^{\delta_*r}I_K(v_t(r))\,dx(r)
 \right|.
 \label{eq4.54}
\end{align}

\textbf{Step~1.} We first estimate the first term on the right-hand side
of \eqref{eq4.54}.
In view of \eqref{eq4.31}, we have
$V_K(v_t(r))
 \leq
 \norm{v_t(r)}^2
 \leq
 \norm{(v_t)_{-N}}_{\Hh}^2$
for each $r\in[-N-\tau,-N]$.
Consequently,
it follows from \eqref{eq4.34-3} that
\begin{align*}
 \Psi_K(-N)
 &=
 V_K(v_t(-N))
 +L_1e^{\sigma\tau}
 \int_{-N-\tau}^{-N}
 e^{-\sigma(-N-r)}
 V_K(v_t(r))\,dr\\
 &\leq
 (1+L_1e^{\sigma\tau}\tau)
 \norm{(v_t)_{-N}}_{\Hh}^2
 \leq
 (1+L_1e^{\sigma\tau}\tau)
 \widehat R(\theta_{-N}\omega).
\end{align*}
Since $\widehat R$ is tempered, we have
\begin{equation}\label{eq4.56}
 \sup_{s\in[-\tau,0]}
 e^{-\delta_*(s+N)}\Psi_K(-N)
 \leq
 e^{\delta_*\tau}e^{-N\delta_*}
 (1+L_1e^{\sigma\tau}\tau)\widehat R(\theta_{-N}\omega)
 \rightarrow0~~\mbox{as}~~N\to+\infty.
\end{equation}

\textbf{Step~2.}
We next estimate the second term on the right-hand side
of \eqref{eq4.54}. By \eqref{eq4.53} and
Lemma~\ref{lem4.7}, we have
\begin{equation}\label{eq4.57}
 \norm{v_t(r)}^2
 \leq
 \norm{(v_t)_r}_{\Hh}^2
 \leq
 \widehat R(\theta_r\omega),
 \qquad -N\leq r\leq0.
\end{equation}
For $j\in\mathbb N_0$, let
$I_j:=[-j-1,-j]$.
If $r\in I_j$, then
\begin{align*}
 \widehat R(\theta_r\omega)
 &=
 \widehat R\bigl(
 \theta_{r+j}(\theta_{-j}\omega)
 \bigr)\leq
 \sup_{|\rho|\leq1}
 \widehat R\bigl(
 \theta_\rho(\theta_{-j}\omega)
 \bigr)
 =
 \widehat R_1^{\#}(\theta_{-j}\omega),
\end{align*}
where $\widehat R_1^{\#}(\omega)$ is given in
Lemma~\ref{lem4.7}.
Splitting $[-N,s]$ into its intersections with the unit
intervals $\{I_j\}_{j\geq0}$ and using
\eqref{eq4.57}, we obtain that
\begin{align}
 &\int_{-N}^{s}
 e^{-\delta_*(s-r)}
 \norm{v_t(r)}^2\,dr=
 \sum_{j=0}^{N-1}
 \int_{[-N,s]\cap I_j}
 e^{-\delta_*(s-r)}
 \norm{v_t(r)}^2\,dr\nonumber\\
 &\leq
 C_{\tau,\delta_*}
 \sum_{j=0}^{\infty}
 e^{-\delta_*j}
 \bigl(
 1+\widehat R_1^{\#}(\theta_{-j}\omega)
 \bigr)=
 C_{\tau,\delta_*}\mathfrak R(\omega),
 \label{eq4.58}
\end{align}
where the constant $C_{\tau,\delta_*}$ is independent of
$t$, $N$, $K$, $\xi$ and $s\in[-\tau,0]$.
Since \(\widehat R_1^{\#}\) is tempered,
\(\mathfrak R(\omega)<\infty\) for every
\(\omega\in\Omega_0\).

\textbf{Step~3.} We next estimate the rough integral term on the
right-hand side of \eqref{eq4.54}.
By Lemma~\ref{lem4.14}, we find that
\begin{align}
&\sup_{s\in[-\tau,0]}
\bigg\{e^{-\delta_*s}
 \left|
 2\int_{-N}^{s}
 e^{\delta_*r}I_K(v_t(r))\,dx(r)\right|
 \bigg\}\nonumber\\
 &\leq \sup_{s\in[-\tau,0]}
\bigg\{2
 e^{-\delta_*s}\sum_{j=0}^{N-1}
 \left|
 \int_{[-N,s]\cap I_j}
 e^{\delta_*r}I_K(v_t(r))\,dx(r)\right|
 \bigg\}
 \nonumber\\
 &\leq\sup_{s\in[-\tau,0]}
\bigg\{
 Ce^{-\delta_*s}
 \sum_{j=0}^{N-1}
 e^{-\delta_*j}
 A(\theta_{-j-1}\omega)
 \gamma_{\lfloor K/2\rfloor}
 (\theta_{-j-1}\omega)\bigg\}
 \nonumber\\
&\leq C_{\tau,\delta_*}
\sum_{j=0}^{\infty}
 e^{-\delta_*j}
 A(\theta_{-j-1}\omega)
 \gamma_{\lfloor K/2\rfloor}
 (\theta_{-j-1}\omega)
:=C_{\tau,\delta_*}
\mathfrak G_K(\omega).
 \label{eq4.59}
\end{align}
Set
$M(\omega)
 :=
 \tnorm{x(\omega)}_{\pvar,[0,1]}
 +2\tnorm{\mathbb X(\omega)}_{\qvar,[0,1]}$.
Then
$\sup_{K\geq1}
 \gamma_{\lfloor K/2\rfloor}(\omega)
 \leq M(\omega)+M(\omega)^2$.
Fix $\varsigma\in(0,\delta_*/3)$. Since $A(\omega)$ and $M(\omega)+M(\omega)^2$ are
tempered, there exists a finite random variable
\(\Lambda_\varsigma(\omega)\) such that, for every
\(j\in\mathbb N_0\),
\[
 A(\theta_{-j-1}\omega)
 \leq
 \Lambda_\varsigma(\omega)e^{\varsigma(j+1)},
 \qquad
 \sup_{K\geq1}
 \gamma_{\lfloor K/2\rfloor}
 (\theta_{-j-1}\omega)
 \leq
 \Lambda_\varsigma(\omega)e^{\varsigma(j+1)}.
\]
Consequently,
\begin{align*}
e^{-\delta_*j}
A(\theta_{-j-1}\omega)
\gamma_{\lfloor K/2\rfloor}
(\theta_{-j-1}\omega)
\leq
\Lambda_\varsigma(\omega)^{2}e^{2\varsigma}
e^{-(\delta_*-2\varsigma)j}.
\end{align*}
Since $\delta_*-2\varsigma>0$, the right-hand side is summable
in $j$ and independent of $K$. Thus
$\mathfrak G_K(\omega)<\infty$ for every $K\geq1$.
For each fixed $j$, Assumption~\textbf{(A5)} implies
$\gamma_{\lfloor K/2\rfloor}
 (\theta_{-j-1}\omega)
 \rightarrow0$ as
$K\to\infty$.
Hence, using the dominated convergence theorem yields
\begin{equation}\label{eq:GK-zero}
 \mathfrak G_K(\omega)\rightarrow0
 \qquad\text{as }K\to\infty
\end{equation}
for every $\omega\in\Omega_0$.

\textbf{Step~4.}
Combining \eqref{eq4.54} and \eqref{eq4.58}--\eqref{eq4.59}, we obtain that for
$N>\tau$ and
$t\geq N+T_D(\theta_{-N}\omega)$,
\begin{align}\label{eq4.60}
 \sup_{s\in[-\tau,0]}\Psi_K(s)
 &\leq
 \sup_{s\in[-\tau,0]}
 e^{-\delta_*(s+N)}\Psi_K(-N)
 +\frac{\nu C_\chi C_{\tau,\delta_*}\mathfrak R(\omega)}{K}
 \bigg(\sup_{i}\sum_{j}w_{ij}|i-j|\bigg)
 \nonumber\\
 &+\frac{\norm{Q_KF_0}^2}
 {(\lambda-L_0-L_1)\delta_*}
 +C_{\tau,\delta_*}
\mathfrak G_K(\omega).
\end{align}
Since $\mathfrak R(\omega)<\infty$,
$\norm{Q_KF_0}\to0$, and
$\mathfrak G_K(\omega)\to0$ as $K\to\infty$, we may choose
$K_0$ sufficiently large such that the sum of the last three
terms on the right-hand side of \eqref{eq4.60} is less than
$\varepsilon/2$.
By \eqref{eq4.56}, we may then choose an integer $N>\tau$
sufficiently large such that
$e^{\delta_*\tau}e^{-\delta_*N}
 \bigl(1+L_1e^{\sigma\tau}\tau\bigr)
 \widehat R(\theta_{-N}\omega)
 <
 \frac{\varepsilon}{2}$.
Set
$T_0:=N+T_D(\theta_{-N}\omega)$.
Then, for every $t\geq T_0$ and every
$\xi\in D(\theta_{-t}\omega)$, the pullback absorption
property gives
$(v_t)_{-N}\in\widehat B(\theta_{-N}\omega)$.
Consequently,
$\sup_{s\in[-\tau,0]}
 e^{-\delta_*(s+N)}\Psi_{K_0}(-N)
 <
 \frac{\varepsilon}{2}$.
It now follows from \eqref{eq4.60} that
$\sup_{s\in[-\tau,0]}
\Psi_{K_0}(s)<\varepsilon$.
Finally, \eqref{eq4.32} yields
\[\sup_{s\in[-\tau,0]}
 \sum_{|i|\geq2K_0}|v_t(s)_i|^2
 \leq
 \sup_{s\in[-\tau,0]}V_{K_0}(v_t(s))\leq
 \sup_{s\in[-\tau,0]}\Psi_{K_0}(s)
 <\varepsilon.\]
Since
$v_t(\vartheta)
 =
 \bigl(\Phi(t,\theta_{-t}\omega)
 \xi\bigr)(\vartheta)$,
 $\vartheta\in[-\tau,0]$,
we obtain \eqref{eq4.52}. This completes the proof.
\end{proof}

\begin{lemma}\label{lem4.16}
Fix $\omega\in\Omega_0$ and let $D\in\mathcal D(\mathcal H)$. Suppose
that $t_n\to\infty$ and
$\xi_n\in D(\theta_{-t_n}\omega)$, and set
\[
 z_n:=\Phi(t_n,\theta_{-t_n}\omega)\xi_n
 \in\Hh.
\]
Then the sequence $\{z_n\}$ is eventually equicontinuous on
$[-\tau,0]$. More precisely, there exist $n_0\in\mathbb N$
and a function
$\rho_\omega:[0,\tau]\to[0,\infty)$, independent of
$n\geq n_0$, such that
$\lim_{\delta\downarrow0}\rho_\omega(\delta)=0
$
and
\begin{align*}
 \sup_{n\geq n_0}
 \norm{z_n(\vartheta_2)-z_n(\vartheta_1)}
 \leq
 \rho_\omega\bigl(|\vartheta_2-\vartheta_1|\bigr)
\end{align*}
for all
\(\vartheta_1,\vartheta_2\in[-\tau,0]\).
\end{lemma}

\begin{proof}
Set $u^{(n)}(r):=u(r;\theta_{-t_n}
\omega,\xi_n)$.
Then $z_{n}(\vartheta)=u^{(n)}(\vartheta+t_n)$.
By the pullback absorbing property of \(\widehat B\), there
exists \(n_0\in\mathbb N\) such that, for every
\(n\geq n_0\),
\[
 (u^{(n)})_{t_n-2\tau-1}
 \in
 \widehat B(\theta_{-2\tau-1}\omega).
\]
Using the cocycle identity of $\Phi$ and
the positive invariance of \(\widehat B\), we obtain that for \(s\in[-2\tau-1,0]\),
\[
\begin{aligned}
 (u^{(n)})_{t_n+s}
 =
 \Phi(t_n+s,\theta_{-t_n}\omega)\xi_n
 &=
 \Phi(s+2\tau+1,\theta_{-2\tau-1}\omega)
 \Phi(t_n-2\tau-1,\theta_{-t_n}\omega)\xi_n\\
 &=
 \Phi(s+2\tau+1,\theta_{-2\tau-1}\omega)
 (u^{(n)})_{t_n-2\tau-1}
 \in
 \widehat B(\theta_s\omega).
\end{aligned}
\]
In view of Lemma~\ref{lem4.7}, we have $\|(u^{(n)})_{t_n+s}\|_{\Hh}^{2}\leq \widehat{R}(\theta_{s}\omega)$. Define $\widehat{R}_{2\tau+1}^{\#}(\omega):=
\sup_{|s|\leq 2\tau+1}
\widehat{R}(\theta_{s}\omega)$. Then
\begin{equation}\label{eq4.62-0}
 \begin{aligned}
\sup_{n\geq n_0}
\sup_{-2\tau-1\leq s\leq 0}
\bigl\|(u^{(n)})_{t_n+s}\bigr\|_{\Hh}
\leq
\left(
\widehat{R}_{2\tau+1}^{\#}(\omega)
\right)^{1/2}.
 \end{aligned}
\end{equation}
Furthermore, using \eqref{eq4.62-0},  Lemma~\ref{lem3.4} and the stationarity identity
$\tnorm{\mathbf x(\theta_{-t_n}\omega)}
 _{p\text{-var},[t_n-\tau,t_n]}=
 \tnorm{\mathbf x(\omega)}
 _{p\text{-var},[-\tau,0]}$,
we find that there exists a finite random variable
\(C(\omega)>0\), independent of \(n\), such that
\begin{equation}\label{eq4.62}
 \begin{aligned}
 &\sup_{n\geq n_0}\Bigl(
 \norm{u^{(n)}}_{\infty,[t_n-2\tau-1,t_n]}
 +\tnorm{u^{(n)}}_{\pvar,[t_n-\tau,t_n]}
 +\tnorm{R^{u^{(n)}}}_{\qvar,[t_n-\tau,t_n]}
 \Bigr)
 \leq C(\omega).
 \end{aligned}
\end{equation}
For
\(\vartheta_1,\vartheta_2\in[-\tau,0]\).
Without loss of generality, we may assume that
\(\vartheta_1\leq\vartheta_2\).
Let $r_{j}:=t_{n}+\vartheta_j$, $j=1,2$, and set $\mathbf{x}^{(n)}
:=\mathbf{x}(\theta_{-t_n}\omega)
=\bigl(x^{(n)},\mathbb{X}^{(n)}\bigr)$.
In view of \eqref{eq:vector}, it holds that
\[
 \norm{z_n(\vartheta_2)-z_n(\vartheta_1)}
 \leq
 \left\|
 \int_{r_1}^{r_2}G(u^{(n)}(r),u^{(n)}(r-\tau)
 )dr
 \right\|
 +
 \left\|
 \int_{r_1}^{r_2}
 h\bigl(u^{(n)}(r)\bigr)\,dx^{(n)}(r)
 \right\|.
\]
Moreover, applying the sewing estimate to the
rough integral, and using \eqref{eq4.62}, Lemma \ref{lem3.1}, Assumptions~\textbf{(A3)}-\textbf{(A4)}, we conclude that
\begin{align}
 \norm{z_n(\vartheta_2)-z_n(\vartheta_1)}
 &\leq C(\omega)(\vartheta_2-\vartheta_1)
+C_h\tnorm{x}_{\pvar,[\vartheta_1,\vartheta_2]}
 +C_h^2\tnorm{\mathbb X}_{\qvar,[\vartheta_1,\vartheta_2]}
 \notag\\
 &\quad+C(\omega)
 \left(
 \tnorm{\mathbb X}_{\qvar,[\vartheta_1,\vartheta_2]}
 +\tnorm{x}_{\pvar,[\vartheta_1,\vartheta_2]}
 \right),
 \label{eq4.63}
\end{align}
from which, we define
\begin{align*}
 \rho_\omega(\delta)
 &:=
 C(\omega)\delta+
 \bigl(C_h+C(\omega)\bigr)
 \sup_{\substack{-\tau\leq a\leq b\leq0\\
 b-a\leq\delta}}
 \tnorm{x}_{\pvar,[a,b]}+
 \bigl(C_h^2+C(\omega)\bigr)
 \sup_{\substack{-\tau\leq a\leq b\leq0\\
 b-a\leq\delta}}
 \tnorm{\mathbb X}_{\qvar,[a,b]}
\end{align*}
for $\delta\in[0,\tau]$. Then \eqref{eq4.63} immediately implies
\[
 \norm{z_n(\vartheta_2)-z_n(\vartheta_1)}
 \leq
 \rho_\omega\bigl(|\vartheta_2-\vartheta_1|\bigr).
\]
Since the maps
$(a,b)\mapsto
 \tnorm{x}_{p\text{-var},[a,b]}^p$
 and
$(a,b)\mapsto
 \tnorm{\mathbb X}_{q\text{-var},[a,b]}^q$
are continuous on
\(\{(a,b):-\tau\leq a\leq b\leq0\}\) and vanish whenever
\(a=b\), we have
\(\rho_\omega(\delta)\to0\) as \(\delta\downarrow0\).
This completes the proof.
\end{proof}

\begin{theorem}\label{thm4.18}
The cocycle $\Phi$ is pullback asymptotically compact in $\Hh$.
\end{theorem}

\begin{proof}
Fix $\omega\in\Omega_0$ and $D\in\mathcal D(\mathcal H)$. Let
$t_n\to\infty$, $\xi_n\in D(\theta_{-t_n}\omega)$, and set
$z_n:=\Phi(t_n,\theta_{-t_n}\omega)\xi_n$.
By Lemma~\ref{lem4.7},
there
exists \(n_0\in\mathbb N\) such that, for every
\(n\geq n_0\),
$z_n
 \in\widehat B(\omega)$.
Hence,
\begin{equation}\label{eq4.64}
 \sup_{n\geq n_0}\norm{z_n}_{\Hh}^2
 \leq\widehat R(\omega).
\end{equation}
For each $m\in\mathbb N$, by Theorem~\ref{thm4.15}, there
exist $K_m\in\mathbb N$ and $N_m\in\mathbb N$ such that
\begin{equation}\label{eq4.65}
 \sup_{n\geq N_m}
 \sup_{\vartheta\in[-\tau,0]}
 \norm{Q_{2K_m}z_n(\vartheta)}
 \leq2^{-m-3}.
\end{equation}

For each fixed $m$, the sequence
$\{P_{2K_m}z_n\}_{n\geq n_0}$ takes values in the
finite-dimensional space $P_{2K_m}\ell^2(\mathbb{Z})$.
After increasing \(n_0\), if necessary, we obtain from
\eqref{eq4.64}, Lemma~\ref{lem4.16}, and the contraction
property of \(P_{2K_m}\) that
\begin{align*}
 \sup_{n\geq n_0}
 \norm{P_{2K_m}z_n}_{\Hh}
 &\leq\widehat R(\omega)^{1/2},\\
\sup_{n\geq n_0}
 \norm{
 P_{2K_m}z_n(\vartheta_2)
 -
 P_{2K_m}z_n(\vartheta_1)}
 &\leq
 \rho_\omega\bigl(
 |\vartheta_2-\vartheta_1|
 \bigr).
\end{align*}
Thus, for each fixed \(m\),
\(\{P_{2K_m}z_n\}_{n\geq n_0}\) is uniformly bounded
and equicontinuous. Applying the Arzel\`a--Ascoli theorem
for each \(m\), followed by a diagonal argument, we extract
a subsequence
$\{z_{n_j}\}$ such that, for every fixed $m$,
$P_{2K_m}z_{n_j}$
converges in
$C([-\tau,0];P_{2K_m}\ell^2(\mathbb{Z}))$
as $j\to\infty$.

Let $\varepsilon>0$ and choose $m\in\mathbb N$ such that
$2^{-m-2}<\frac{\varepsilon}{2}$.
There exists $j_0\in\mathbb N$ such that,
for all $j,k\geq j_0$,
\[
 \sup_{\vartheta\in[-\tau,0]}
 \norm{
 P_{2K_m}\bigl(
 z_{n_j}(\vartheta)-z_{n_k}(\vartheta)
 \bigr)}
 <
 \frac{\varepsilon}{2}.
\]
Increasing $j_0$ if necessary, we may also assume that
$n_j,n_k\geq N_m$ whenever $j,k\geq j_0$. Then
by \eqref{eq4.65},
we have
\begin{align*}
 \norm{z_{n_j}-z_{n_k}}_{\Hh}
 &\leq\sup_{\vartheta\in[-\tau,0]}\norm{P_{2K_m}
 (z_{n_j}(\vartheta)-z_{n_k}(\vartheta))}+\sup_{\vartheta\in[-\tau,0]}\norm{Q_{2K_m}
 z_{n_j}(\vartheta)}\\
 &\quad
 +\sup_{\vartheta\in[-\tau,0]}\norm{Q_{2K_m}z_{n_k}
 (\vartheta)}
 <\frac{\varepsilon}{2}+2^{-m-2}<\varepsilon.
\end{align*}
Thus \(\{z_{n_j}\}\) is a Cauchy sequence in \(\Hh\).
Since \(\Hh\) is complete, it converges in \(\Hh\). This proves the pullback
asymptotic compactness of $\Phi$.
\end{proof}

Combining the preceding results, we obtain the main result of this section.

\begin{theorem}\label{thm:attractor}
The cocycle \(\Phi\) admits a unique
\(\mathcal D(\Hh)\)-random pullback attractor
\(\mathcal A=\{\mathcal A(\omega)\}_{\omega\in\Omega_0}\)
in \(\Hh\), given by
\begin{equation}\label{eq4.70}
 \mathcal A(\omega)
 =
 \bigcap_{s\geq0}
 \overline{
 \bigcup_{t\geq s}
 \Phi(t,\theta_{-t}\omega)
 B(\theta_{-t}\omega)
 }^{\Hh},
\end{equation}
where \(B(\omega)\) is defined in \eqref{eq4.17}.
\end{theorem}

\begin{proof}
By Theorem~\ref{thm3.9}, \(\Phi\) is a continuous cocycle
on \(\Hh\). Theorem~\ref{thm4.6} provides a closed
measurable pullback absorbing set
$B=\{B(\omega)\}_{\omega\in\Omega_0}
 \in\mathcal D(\Hh)$,
and Theorem~\ref{thm4.18} shows that \(\Phi\) is
\(\mathcal D(\Hh)\)-pullback asymptotically compact in
\(\Hh\). Since the tempered universe $\mathcal D(\mathcal H)$ is
inclusion-closed, by the standard existence theorem for random pullback
attractors
\cite[Proposition~2.7]{BatesLuWang2009},
we conclude that \(\Phi\) admits a unique $\mathcal D(\mathcal H)$-random pullback attractor
$\mathcal A\in\mathcal D(\mathcal H)$, which is given by
\eqref{eq4.70}.
\end{proof}

\section{Euler and finite-dimensional approximations}

Let $m\in\N$, $\Delta:=\tau/m$, $t_k:=k\Delta$, and write
$x_k:=x_{t_k,t_{k+1}}$,
$\mathbb X_k:=\mathbb X_{t_k,t_{k+1}}$.
The compensated rough Euler scheme is given by
\begin{equation}\label{eq:euler}
 \begin{aligned}
 u_{k+1}^{\Delta}
 =&u_k^{\Delta}
 +\Delta\bigl[-\mathcal Lu_k^{\Delta}-\lambda u_k^{\Delta}
 +F(u_k^{\Delta},u_{k-m}^{\Delta})\bigr]
 +h(u_k^{\Delta})x_k+H(u_k^{\Delta})\mathbb X_k,
 ~~~~ k\geq0,
 \end{aligned}
\end{equation}
with $u_k^\Delta=\xi(t_k)$ for $-m\leq k\leq0$.
Since the delay interval $[-\tau,0]$
contains the $m+1$ grid points $j\Delta$, $-m\leq j\leq0$,
the state of the delay
scheme at each discrete time $t_k$ is represented by the discrete history vector
$\mathbf u_k^\Delta
 :=
 (u_{k-m}^\Delta,\ldots,u_k^\Delta)
 \in U^{m+1}$.
This motivates the discrete history space
\[
 \Hh_\Delta=U^{m+1},
 \qquad
 \norm{\zeta}_{\Hh_\Delta}
 :=
 \max_{-m\leq j\leq0}\norm{\zeta_j},
\]
whose norm is the discrete counterpart of the supremum norm
on $\Hh=C([-\tau,0];\ell^2(\mathbb{Z}))$.
Let
$P_\Delta:\Hh\to\Hh_\Delta$ and
$J_\Delta:\Hh_\Delta\to\Hh$
denote the sampling and piecewise linear interpolation operators, respectively. The $k$-step
solution map is denoted by $\Phi_\Delta(k,\omega)$ and satisfies
\begin{equation}\label{eq:discrete-cocycle}
 \Phi_\Delta(k+\ell,\omega)
 =\Phi_\Delta(k,\theta_{\ell\Delta}\omega)
 \circ\Phi_\Delta(\ell,\omega).
\end{equation}
For fixed $\Delta>0$, a discrete random bounded family
$\{D_\Delta(\omega)\}_{\omega\in\Omega}\subset\Hh_\Delta$
is called tempered if its radius
$R_\Delta(\omega):=
\sup_{\zeta\in D_\Delta(\omega)}
\|\zeta\|_{\Hh_\Delta}$
satisfies, for every $c>0$,
\begin{equation}\label{eq:discrete-tempered}
e^{-ck\Delta}R_\Delta(\theta_{-k\Delta}\omega)
\rightarrow0
\qquad\text{as }k\to\infty.
\end{equation}

Let $\xi\in\Hh$, and let
$u=u(\,\cdot\,;\omega,\xi)$ be the solution of
\eqref{eq:vector} with initial history
$u(s)=\xi(s)$, $s\in[-\tau,0]$.
For every $k\geq0$, by the integral formulation of
\eqref{eq:vector}, we obtain that
\begin{equation}\label{eq5.1}
 \begin{aligned}
 u(t_{k+1})
 =u(t_k)
 +\Delta G\bigl(u(t_k),u(t_k-\tau)\bigr)
 +h\bigl(u(t_k)\bigr)x_k
 +H\bigl(u(t_k)\bigr)\mathbb X_k+\varrho_k,
 \end{aligned}
\end{equation}
where
\begin{align*}
\varrho_k&=
 \int_{t_k}^{t_{k+1}}
 \Bigl[
 G\bigl(u(r),u(r-\tau)\bigr)
 -
 G\bigl(u(t_k),u(t_k-\tau)\bigr)
 \Bigr]dr\\
&+
 \int_{t_k}^{t_{k+1}}
 h\bigl(u(r)\bigr)\,dx(r)
 -h\bigl(u(t_k)\bigr)x_k
 -H\bigl(u(t_k)\bigr)\mathbb X_k:=\varrho_k^{(d)}+\varrho_k^{(r)}.
\end{align*}

\begin{lemma}\label{lem5.1}
Let $T>0$, $S\subset\Hh$ be compact, and let
\(\beta>1/p\) be as in Assumption~\textbf{(A4)}. Fix
$\alpha\in(1/3,\beta)$.
Then, for every $\omega\in\Omega_0$, there exists a
constant $C_{S,T}(\omega)>0$, independent of $\Delta$,
such that, for all sufficiently small $\Delta=\tau/m$,
the remainder $\varrho_k$ in \eqref{eq5.1} satisfies
\begin{equation}\label{eq5.2}
 \norm{\varrho_k}
 \leq
 C_{S,T}(\omega)
 \left(
 \Delta^{1+\alpha}
 +\Delta^{3\alpha}
 +\Delta\omega_S(\Delta)
 \right),
\end{equation}
uniformly over \(\xi\in S\) and \(k\geq0\) satisfying
\(t_{k+1}\leq T\), where \[\omega_S(\delta)
 :=
 \sup_{\xi\in S}
 \sup_{r,s\in[-\tau,0],|r-s|\leq\delta}
 \norm{\xi(r)-\xi(s)}.\]
\end{lemma}

\begin{proof}
Let $s\in [t_{k}, t_{k+1}]$ and let $\psi_{t_k,s}$ denote the local rough flow generated by $d\psi=h(\psi)\,dx$. Set
$z(s):=\psi_{t_k,s}^{-1}(u(s))$.
By using Lemma \ref{lem3.2}, Corollary \ref{cor3.5}, the boundedness of $\mathcal L$, the linear
growth of $F$ and the
$\beta$-H\"{o}lder regularity of $\mathbf x$, we have
\begin{align*}
 &\norm{u(s)-u(t_{k})}
 \leq
 \norm{\psi_{t_{k},s}(z(s))-
 \psi_{t_{k},s}(z(t_{k}))}
 +
 \norm{\psi_{t_{k},s}
 (z(t_{k}))-z(t_{k})}\\
 &\leq C\|z(s)-z(t_{k})\|+
 \norm{h(z(t_k))}\norm{x_{t_k,s}}+
\norm{H(z(t_k))}\norm{\mathbb{X}_{t_k,s}}
+
\norm{R^\psi_{t_k,s}(z(t_k))}
 \\
 &\leq C_{S,T}(\omega)(s-t_{k})+
 C_{T}(\omega)(s-t_{k})^{\beta}.
\end{align*}
Then from  $\alpha<\beta<1$ and $\Delta\leq1$, we deduce that
\begin{align}\label{eq:positive-delay-increment}
 \sup_{s\in[t_k,t_{k+1}]}\norm{u(s)-u(t_k)}
 \leq C_{S,T}(\omega)\Delta^\alpha,
\end{align}

We next estimate the delayed increment. If $t_k\geq\tau$,
then
$[t_k-\tau,t_{k+1}-\tau]\subset[0,T]$.
Applying \eqref{eq:positive-delay-increment} on
$[t_{k-m},t_{k-m+1}]$ therefore yields
\begin{align}
\label{eq:positive-delay-increment-1}
 \sup_{s\in[t_k,t_{k+1}]}
 \norm{u(s-\tau)-u(t_k-\tau)}
 =
 \sup_{r\in[t_{k-m},t_{k-m+1}]}
 \norm{u(r)-u(t_{k-m})}
 \leq
 C_{S,T}(\omega)\Delta^\alpha.
\end{align}
If $t_k<\tau$, then $k\leq m-1$ and
$t_{k+1}\leq\tau$. Consequently,
$[t_k-\tau,t_{k+1}-\tau]
\subseteq[-\tau,0]$.
Moreover, in view of the definition of
$\omega_S$, it holds that
\begin{equation}
\label{eq:positive-delay-increment-2}
 \sup_{r\in[t_k,t_{k+1}]}
 \norm{u(r-\tau)-u(t_k-\tau)}
 \leq\omega_S(\Delta).
\end{equation}
Thanks to \eqref{eq:positive-delay-increment}-\eqref{eq:positive-delay-increment-2},
we obtain that
\begin{align}
\label{eq:positive-delay-increment-3}
 \norm{\varrho_k^{(d)}}
\leq
 \Delta\left(
 C_{S,T}(\omega)\Delta^\alpha
 +L_1\omega_S(\Delta)
 \right)
 \leq
 C_{S,T}(\omega)
 \left(
 \Delta^{1+\alpha}
 +\Delta\omega_S(\Delta)
 \right).
\end{align}
Applying the sewing estimate to the rough integral, and using the local $\beta$-H\"older regularity of $\mathbf{x}$ and \eqref{eq3.2},
we conclude that
\begin{align*}
 \norm{\varrho_k^{(r)}}
 &\leq C_p\Bigl(
 \tnorm{H(u)}_{p\text{-var},[t_k,t_{k+1}]}
 \tnorm{\mathbb X}_{q\text{-var},[t_k,t_{k+1}]}\\
 &
 +
 \tnorm{R^{h(u)}}_{q\text{-var},[t_k,t_{k+1}]}
 \tnorm{x}_{p\text{-var},[t_k,t_{k+1}]}
 \Bigr)
 \leq
 C_{S,T}(\omega)
 \bigl(\Delta^{3\alpha}
 +\Delta^{1+\alpha}\bigr),
\end{align*}
where in the last inequality we have used the local estimates $\tnorm{u}_{p\text{-var},[t_k,t_{k+1}]}
\leq C_{S,T}(\omega)
(\Delta^\beta+\Delta)$ and
$\tnorm{R^u}_{q\text{-var},[t_k,t_{k+1}]}
\leq C_{S,T}(\omega)(\Delta+
\Delta^{2\beta})$, which follow by applying arguments analogous to those used in the proof
of Lemma~\ref{lem3.4} on the interval
$I_k=[t_k,t_{k+1}]$.
Combining this estimate with
\eqref{eq:positive-delay-increment-3} proves
\eqref{eq5.2}.

\end{proof}

Let $z=(z_j)_{j\in\mathbb Z}$ be a $U$-valued grid path, where  $z_j=z(t_{j})\in U$ and $t_j=j\Delta$. Set $z_{i,j}=z_{j}-z_{i}$.
For $i<j$, the discrete $p$-variation of $z$ over the grid interval
$[i,j]$ is defined by
\[
 \norm{z}_{p\text{-var},\Delta;[i,j]}
 :=
 \left(
 \sup_{i=i_0<\cdots<i_n=j}
 \sum_{r=0}^{n-1}
 \norm{z_{i_r,i_{r+1}}}^p
 \right)^{1/p},
\]
where the supremum is taken over all integer partitions consisting
of grid indices.
Let $z'=(z_i')_{i\in\mathbb Z}$ be a sequence with
$z_i'\in\mathcal L(U,U)$. For $i<j$, define
\[
R^{z,\Delta}_{i,j}
:=
z_{i,j}-z_i'x_{t_i,t_j}.
\]
Its discrete $q$-variation over $[i,j]$ is defined by
\[
\norm{R^{z,\Delta}}_{q\text{-var},\Delta;[i,j]}
:=
\left(
\sup_{i=i_0<\cdots<i_n=j}
\sum_{r=0}^{n-1}
\norm{R^{z,\Delta}_{i_r,i_{r+1}}}^q
\right)^{1/q}.
\]
If $z'$ has finite discrete $p$-variation and
$R^{z,\Delta}$ has finite discrete $q$-variation, then $z'$ is called
a discrete Gubinelli derivative of $z$ with respect to $x$.

\begin{lemma}\label{lem5.3}
Let $\bar{\omega}$ be a continuous control on a fixed
interval $[a,b]$. For every $\gamma>0$, there exists
$\Delta_0=\Delta_0(\bar{\omega},\gamma)>0$ such that,
for every $0<\Delta\leq\Delta_0$ and any integers
$n_0<n_1$ satisfying
$a\leq n_0\Delta<n_1\Delta\leq b$,
there exist grid indices
$n_0=k_0<k_1<\cdots<k_N=n_1$ such that, with
$t_k:=k\Delta$,
\[
 \bar{\omega}(t_{k_j},t_{k_{j+1}})
 \leq\gamma,
 \qquad j=0,\ldots,N-1,
\]
and
\[
 N\leq
 2+2\gamma^{-1}
 \bar{\omega}(t_{n_0},t_{n_1}).
\]
\end{lemma}

\begin{proof}
The assertions follow from
the discrete version of the greedy-partition argument used in the
proof of Lemma~\ref{lem3.4}, together with the superadditivity of
$\bar{\omega}$; the standard details are omitted.
\end{proof}

We next provide a discrete sewing lemma for approximately additive
two-index grid maps.

\begin{lemma}\label{lem5.4}
Let $\bar{\omega}$ be the continuous control on $[t_{n_0},t_{n_1}]$, and let
$A_{i,j}\in U$,
$n_0\leq i\leq j\leq n_1$,
be a two-index grid map satisfying $A_{i,i}=0$. Define
$\delta A_{i,j,k}:=
A_{i,k}-A_{i,j}-A_{j,k}$.
Suppose that there exist constants $\vartheta>1$ and $C_A\geq0$,
independent of the grid indices, such that
$\norm{\delta A_{i,j,k}}
\leq
C_A\bar{\omega}(t_i,t_k)^\vartheta$
for all $n_0\leq i<j<k\leq n_1$.
Then there exists a unique additive two-index grid map
$\mathcal I$ satisfying
$\mathcal I_{i,i+1}=A_{i,i+1}$,
$n_0\leq i<n_1$,
and
\begin{equation}\label{eq5.25}
\norm{\mathcal I_{i,j}-A_{i,j}}
\leq
C_\vartheta C_A
\bar{\omega}(t_i,t_j)^\vartheta,
\qquad n_0\leq i<j\leq n_1,
\end{equation}
where $C_\vartheta>0$ depends only on $\vartheta$.
\end{lemma}

\begin{proof}
The unique additive extension is
$\mathcal I_{i,j}=\sum_{r=i}^{j-1}A_{r,r+1}$ for
$i<j$.
Let $\pi=\{i=i_0<i_1<\cdots<i_m=j\}$
be a partition of the grid indices and set $A^\pi_{i,j}:=
 \sum_{\ell=0}^{m-1}
 A_{i_\ell,i_{\ell+1}}$.
The odd--even decomposition and
superadditivity of $\bar{\omega}$ give
$\sum_{r=1}^{m-1}
 \bar{\omega}(t_{i_{r-1}},t_{i_{r+1}})
 \leq
 2\bar{\omega}(t_i,t_j)$.
Consequently, there exists at least one
$r\in\{1,\ldots,m-1\}$ such that
$\bar{\omega}(t_{i_{r-1}},t_{i_{r+1}})
 \leq
 \frac{2}{m-1}\bar{\omega}(t_i,t_j)$.
Let $\pi'$ be the partition obtained from $\pi$ by deleting
the index $i_r$. Then
\begin{align*}
 \norm{A^\pi_{i,j}-A^{\pi'}_{i,j}}
 =\norm{\delta A_{i_{r-1},i_r,i_{r+1}}}
 \leq
 C_A
 \bar{\omega}(t_{i_{r-1}},t_{i_{r+1}})^\vartheta
 \leq C_A
 2^\vartheta
 (m-1)^{-\vartheta}
 \bar{\omega}(t_i,t_j)^\vartheta.
\end{align*}
Starting from the finest partition
$\pi_0=\{i<i+1<\cdots<j\}$
and repeatedly deleting an interior index until only the
endpoints $i$ and $j$ remain, we obtain that
\begin{align*}
 \norm{\mathcal I_{i,j}-A_{i,j}}
 \leq
 2^\vartheta C_A\bar{\omega}(t_i,t_j)^\vartheta
 \sum_{n=1}^{j-i-1}n^{-\vartheta}
 \leq
 2^\vartheta C_A\bar{\omega}(t_i,t_j)^\vartheta
 \sum_{n=1}^{\infty}n^{-\vartheta}.
\end{align*}
Since $\vartheta>1$, the last series converges. Hence
\eqref{eq5.25} follows with
$C_\vartheta
 :=
 2^\vartheta
 \sum_{n=1}^{\infty}n^{-\vartheta}$,
which depends only on $\vartheta$ and is independent of the
grid and the mesh size.
\end{proof}

\begin{corollary}\label{cor5.5}
Let $A$ and $\delta A$ be as in Lemma~\ref{lem5.4}, and let
$\bar{\omega}_1$ and $\bar{\omega}_2$ be continuous controls on
$[t_{n_0},t_{n_1}]$. Suppose that, for some $\alpha_1,\alpha_2>0$ with
$\alpha_1+\alpha_2>1$ and some constant $C_A\geq0$ independent of the
grid indices,
$\norm{\delta A_{i,j,k}}
\leq
C_A\bar{\omega}_1(t_i,t_j)^{\alpha_1}
\bar{\omega}_2(t_j,t_k)^{\alpha_2}$,
$n_0\leq i<j<k\leq n_1$.
Then
\[
\norm{\mathcal I_{i,j}-A_{i,j}}
\leq
C_{\alpha_1,\alpha_2}C_A
\bar{\omega}_1(t_i,t_j)^{\alpha_1}
\bar{\omega}_2(t_i,t_j)^{\alpha_2},
\qquad n_0\leq i<j\leq n_1,
\]
where $C_{\alpha_1,\alpha_2}>0$ depends only on $\alpha_1$ and $\alpha_2$.
\end{corollary}
\begin{proof}
Set $\vartheta=\alpha_1+\alpha_2$ and
$\bar\omega
:=
\bar{\omega}_1^{\alpha_{1}/\vartheta}
\bar{\omega}_2^{\alpha_{2}/\vartheta}$.
By H\"older's inequality, we find that $\bar\omega$ is a continuous control, and thus
$\norm{\delta A_{i,j,k}}
\leq C_A\bar\omega(t_i,t_k)^\vartheta$.
Therefore the result follows immediately from Lemma~\ref{lem5.4}.
\end{proof}

\begin{lemma}\label{lem5.5}
Let $\omega_{\mathbf x}$ be the continuous control defined in
\eqref{eq2.4}, and let $M>0$. There exist constants
$\gamma_M\in(0,1]$ and $C_M>0$, independent of $\Delta$, such that
the following assertions hold.
Let $[k_0,k_1]$ be the grid block satisfying
$\omega_{\mathbf x}(t_{k_0},t_{k_1})\leq\gamma_M$.
Suppose that $u^\Delta=(u_j^\Delta)$ satisfies \eqref{eq:euler} for
$k_0\leq j<k_1$ and
$\max_{k_0-m\leq j\leq k_1}\|u_j^\Delta\|\leq M$.
Then, with discrete Gubinelli derivative
$(u^\Delta)'_j=h(u_j^\Delta)$ and remainder
$R^{u,\Delta}_{i,j}
 :=
 u_j^\Delta-u_i^\Delta
 -h(u_i^\Delta)x_{t_i,t_j}$,
 $k_0\leq i\leq j\leq k_1$,
we have
\begin{align}\label{eq5.26}
 \tnorm{u^\Delta}_{p\text{-var},\Delta;[k_0,k_1]}
 +
 \tnorm{R^{u,\Delta}}_{q\text{-var},\Delta;[k_0,k_1]}
 \leq
 C_M\omega_{\mathbf x}(t_{k_0},t_{k_1})^{1/p}.
\end{align}
Moreover, let $z^\Delta$ satisfy \eqref{eq:euler}, and let
$(r_j)_{j=k_0}^{k_1-1}\subset U$ be a sequence of perturbations. Suppose that
$y^\Delta$ satisfies the perturbed recursion
\[y_{j+1}^\Delta
 =y_j^\Delta
 +\Delta
 [-\mathcal Ly_j^\Delta-\lambda y_j^\Delta
 +F(y_j^\Delta,y_{j-m}^\Delta)\bigr]
 +h(y_j^\Delta)x_j
 +H(y_j^\Delta)\mathbb X_j+r_j,
 ~~k_0\leq j<k_1,\]
and
$\max_{k_0-m\leq j\leq k_1}
 \max\{\|y_j^\Delta\|,
 \|z_j^\Delta\|\}
 \leq M$.
Set $e_j:=y_j^\Delta-z_j^\Delta$,
$e'_j:=h(y_j^\Delta)-h(z_j^\Delta)$ and
$R^{e,\Delta}_{i,j}
 :=e_j-e_i-e'_i x_{t_i,t_j}$.
Then
\begin{align}\label{eq5.27}
 &\tnorm{e}_{p\text{-var},\Delta;[k_0,k_1]}
 +\tnorm{R^{e,\Delta}}_{q
 \text{-var},\Delta;[k_0,k_1]}\nonumber\\
 &\leq
 C_M\Bigl(
 \norm{e_{k_0}}
 +(t_{k_1}-t_{k_0})
 \max_{k_0-m\leq j\leq k_1-m}\norm{e_j}
 +\sum_{j=k_0}^{k_1-1}\norm{r_j}
 \Bigr).
\end{align}
\end{lemma}

\begin{proof}
Write
\[
 G(u_i^\Delta, u_{i-m}^\Delta)
 :=
 -\mathcal Lu_i^\Delta-\lambda u_i^\Delta
 +F(u_i^\Delta,u_{i-m}^\Delta)
 ~~\mbox{and}~~
 A^u_{i,j}
 :=
 h(u_i^\Delta)x_{t_i,t_j}
 +H(u_i^\Delta)\mathbb X_{t_i,t_j}.
\]
Set $\delta A^u_{a,b,c}
 :=A^u_{a,c}-A^u_{a,b}-A^u_{b,c}$.
Then applying Taylor's formula to $h(u_b^\Delta)$ around
$u_a^\Delta$ and using Chen's relation \eqref{eq2.1}, we have
\begin{align}\label{eq5.28}
 \delta A^u_{a,b,c}
 &=\bigl[h(u_a^\Delta)-h(u_b^\Delta)\bigr]
x_{t_b,t_c}+
\bigl[H(u_a^\Delta)-H(u_b^\Delta)\bigr]
\mathbb X_{t_b,t_c}+
H(u_a^\Delta)
\bigl(x_{t_a,t_b}\otimes x_{t_b,t_c}\bigr)\nonumber\\
&=
 -Dh(u_a^\Delta)R^{u,\Delta}_{a,b}x_{t_b,t_c}
 -\mathcal Q^h_{a,b}x_{t_b,t_c}
 +[H(u_a^\Delta)-H(u_b^\Delta)]
 \mathbb X_{t_b,t_c},
\end{align}
where
$\mathcal Q^h_{a,b}
 :=
 h(u_b^\Delta)-h(u_a^\Delta)
 -Dh(u_a^\Delta)u_{a,b}^\Delta$ satisfies
 $\|\mathcal Q^h_{a,b}\|
 \leq c_M\|u_{a,b}^\Delta\|^2$.

We next prove by induction on the integer $j-i$ that, for every
$k_0\leq i<j\leq k_1$,
\begin{align}\label{eq5.29}
 \norm{u_{i,j}^\Delta}
 &\leq K_M\omega_{\mathbf x}(t_i,t_j)^{1/p}~~\mbox{and}~~
 \norm{R^{u,\Delta}_{i,j}}
 \leq K_M\omega_{\mathbf x}(t_i,t_j)^{2/p},
\end{align}
where the constant $K_M\geq1$ will be chosen below.
In view of \eqref{eq:euler} and $\max_{k_0-m\leq j\leq k_1}\|u_j^\Delta\|\leq M$, we have
\[
\begin{aligned}
\norm{u_{i,i+1}^\Delta}
&\leq
(c+\|u_i^\Delta\|+\|u_{i-m}^\Delta\|)
\Bigl(
\Delta+\norm{x_{t_i,t_{i+1}}}
+\norm{\mathbb X_{t_i,t_{i+1}}}
\Bigr)\\
&\leq
c_M\bigl(
\omega_{\mathbf x}(t_i,t_{i+1})+
\omega_{\mathbf x}(t_i,t_{i+1})^{1/p}+
\omega_{\mathbf x}(t_i,t_{i+1})^{2/p}
\bigr)\leq
c_M\omega_{\mathbf x}(t_i,t_{i+1})^{1/p},\\
\|R^{u,\Delta}_{i,i+1}\|&=\|\Delta
G(u_i^\Delta, u_{i-m}^\Delta)
 +H(u_i^\Delta)\mathbb X_{t_i,t_{i+1}}\|
\leq
c_M\omega_{\mathbf x}(t_i,t_{i+1})^{2/p}.
 \end{aligned}
\]
Thus \eqref{eq5.29} holds for $j-i=1$, provided that $K_M$ is
chosen sufficiently large.

For the induction step, fix $k_0\leq i<j\leq k_1$ with
$j-i\geq2$, and assume that \eqref{eq5.29} holds for every
$k_0\leq a<b\leq k_1$ satisfying $b-a<j-i$. For each triple $i\leq a<b<c\leq j$,
we have $b-a<j-i$. Hence, by \eqref{eq5.28}-\eqref{eq5.29} we deduce that
\begin{align*}
 \norm{\delta A^u_{a,b,c}}
 \leq
 c_M(K_M+K_M^2)
 \omega_{\mathbf x}(t_a,t_c)^{3/p}.
\end{align*}
Since this estimate holds for every triple in $[i,j]$, define
$\mathcal J^u_{i,j}
:=
\sum_{\ell=i}^{j-1}A^u_{\ell,\ell+1}$.
Applying Lemma~\ref{lem5.4} with exponent $3/p>1$, we obtain
\[
\norm{\mathcal J^u_{i,j}-A^u_{i,j}}
\leq
c_M(K_M+K_M^2)
\omega_{\mathbf x}(t_i,t_j)^{3/p}.
\]
Summing \eqref{eq:euler} over $\ell=i,\ldots,j-1$ and using the
definition of $R^{u,\Delta}$, we have
\[
 \begin{aligned}
 u_{i,j}^\Delta
 &=
 \sum_{\ell=i}^{j-1}\Delta
G(u_\ell^\Delta, u_{\ell-m}^\Delta)
 +A^u_{i,j}
 +\bigl(\mathcal J^u_{i,j}-A^u_{i,j}\bigr),\\
 R^{u,\Delta}_{i,j}
 &=
 \sum_{\ell=i}^{j-1}\Delta
 G(u_\ell^\Delta, u_{\ell-m}^\Delta)
 +H(u_i^\Delta)\mathbb X_{t_i,t_j}
 +\bigl(\mathcal J^u_{i,j}-A^u_{i,j}\bigr).
 \end{aligned}
\]
In view of
$\omega_{\mathbf x}(t_i,t_j)\leq\gamma_M\leq1$,
these identities imply
\begin{align*}
 \norm{u_{i,j}^\Delta}
 &\leq
 \left[c_M+c_M(K_M+K_M^2)\gamma_M^{2/p}\right]
 \omega_{\mathbf x}(t_i,t_j)^{1/p},\\
 \norm{R^{u,\Delta}_{i,j}}
 &\leq
 \left[c_M+c_M(K_M+K_M^2)\gamma_M^{1/p}\right]
 \omega_{\mathbf x}(t_i,t_j)^{2/p}.
\end{align*}
Choose
$K_M\geq\max\{1,2c_M\}$,
and then choose $\gamma_M\in(0,1]$ sufficiently small that
$c_M(K_M+K_M^2)\gamma_M^{1/p}
\leq\frac{K_M}{2}$.
Thus, \eqref{eq5.29} holds on $[i,j]$, which completes the
induction.
Moreover, using the superadditivity of $\omega_{\mathbf x}$
and $q=p/2$, we infer from \eqref{eq5.29} that
\begin{align*}
 \tnorm{u^\Delta}_{p\text{-var},\Delta;[k_0,k_1]}
 &\leq
 K_M\omega_{\mathbf x}(t_{k_0},t_{k_1})^{1/p},
 \\
 \tnorm{R^{u,\Delta}}_{q\text{-var},\Delta;[k_0,k_1]}
 &\leq
 K_M\omega_{\mathbf x}(t_{k_0},t_{k_1})^{2/p}.
\end{align*}
Since $\omega_{\mathbf x}(t_{k_0},t_{k_1})\leq1$,
their sum is bounded by
$2K_M\omega_{\mathbf x}(t_{k_0},t_{k_1})^{1/p}$.
This proves \eqref{eq5.26} with a constant independent
of $\Delta$.

For the second assertion, set
\[
 H_j:=H(y_j^\Delta)-H(z_j^\Delta),
 \qquad
 A^e_{i,j}:=
 e'_ix_{t_i,t_j}+H_i\mathbb X_{t_i,t_j},
\]
and write
\[X:=
 \tnorm{e}_{p\text{-var},\Delta;[k_0,k_1]}
 +
 \tnorm{R^{e,\Delta}}_{q\text{-var},\Delta;[k_0,k_1]}.
\]
By the similar arguments as in \eqref{eq5.28}, we obtain that for
$k_0\leq i<j<k\leq k_1$,
\begin{align*}
 \norm{\delta A^e_{i,j,k}}
 &=\norm{(e'_i-e'_j)x_{t_j,t_k}+
(H_i-H_j)\mathbb X_{t_j,t_k}+
H_i\bigl(x_{t_i,t_j}\otimes x_{t_j,t_k}\bigr)}\\
 &\leq
 C_M\max_{k_0\leq j\leq k_1}\norm{e_j}\omega_{\mathbf x}(t_i,t_k)^{3/p}
 +
 C_M\norm{R^{e,\Delta}_{i,j}}
 \omega_{\mathbf x}(t_j,t_k)^{1/p}.
\end{align*}
Since
$\|R^{e,\Delta}_{i,j}\|
 \leq
 \tnorm{R^{e,\Delta}}_{q\text{-var},
 \Delta;[i,j]}$
and
$(i,j)\mapsto
\tnorm{R^{e,\Delta}}_{q\text{-var},
\Delta;[i,j]}^q$
defines a grid control, in view of
Lemma \ref{lem5.4} and Corollary \ref{cor5.5}, we have
\begin{align}
 \tnorm{\sum_{\ell=i}^{j-1}
 A^e_{\ell,\ell+1}-A^e_{i,j}
 }_{q\text{-var},\Delta;[k_0,k_1]}
 &\leq
 C_M\max_{k_0\leq j\leq k_1}\norm{e_j}\omega_{\mathbf x}(t_{k_0},t_{k_1})^{3/p}
 \nonumber\\
 &+
 C_M\omega_{\mathbf x}(t_{k_0},t_{k_1})^{1/p}
 \tnorm{R^{e,\Delta}}_{q\text{-var},
\Delta;[k_0,k_1]}.
 \label{eq5.33-1}
\end{align}
By subtracting the Euler recursion \eqref{eq:euler} for
$z^\Delta$ from the perturbed recursion for $y^\Delta$,
summing over $\ell=i,\ldots,j-1$, and using the definitions
of $R^{e,\Delta}_{i,j}$ and $A^e_{i,j}$, we obtain that
\[
 \begin{aligned}
 R^{e,\Delta}_{i,j}
 ={}&
 \sum_{\ell=i}^{j-1}\Delta
 \Bigl[
 -\mathcal Le_\ell-\lambda e_\ell
 +F(y_\ell^\Delta,y_{\ell-m}^\Delta)
 -F(z_\ell^\Delta,z_{\ell-m}^\Delta)
 \Bigr]\\
 &+K_i\mathbb X_{t_i,t_j}
 +\bigg(\sum_{\ell=i}^{j-1}
 A^e_{\ell,\ell+1}-A^e_{i,j}\bigg)
 +\sum_{\ell=i}^{j-1}r_\ell.
 \end{aligned}
\]
It follows from \eqref{eq5.33-1} that
\begin{align}\label{eq5.33-2}
&\tnorm{R^{e,\Delta}}_{q
\text{-var},\Delta;[k_0,k_1]}
\leq
\sum_{j=k_0}^{k_1-1}
C_M\Delta\bigl(\norm{e_j}
+\norm{e_{j-m}}\bigr)
+C_M \max_{k_0\leq j\leq k_1}\norm{e_j}
\tnorm{\mathbb X}_{q
\text{-var},\Delta;[k_0,k_1]}\nonumber\\
&+C_M\max_{k_0\leq j\leq k_1}\norm{e_j}\omega_{\mathbf x}(t_{k_0},t_{k_1})^{3/p}
+
C_M\omega_{\mathbf x}(t_{k_0},t_{k_1})^{1/p}
 \tnorm{R^{e,\Delta}}_{q\text{-var},
\Delta;[k_0,k_1]}+
\sum_{j=k_0}^{k_1-1}\norm{r_j}\nonumber\\
&\leq C_{M}(t_{k_1}-t_{k_0})
\max_{k_0-m\leq j\leq k_1-m}\norm{e_j}+
C_M\max_{k_0\leq j\leq k_1}\norm{e_j}\omega_{\mathbf x}(t_{k_0},t_{k_1})^{1/p}\nonumber\\
&+C_M\omega_{\mathbf x}(t_{k_0},t_{k_1})^{1/p}
 \tnorm{R^{e,\Delta}}_{q\text{-var},
\Delta;[k_0,k_1]}+
\sum_{j=k_0}^{k_1-1}\norm{r_j}.
\end{align}
In view of
$e_{i,j}=e'_ix_{t_i,t_j}+
R^{e,\Delta}_{i,j}$,
$\norm{e'_i}\leq C_M\norm{e_i}$, and $q<p$, we have
\begin{align}\label{eq5.33-3}
\tnorm{e}_{p\text{-var},\Delta;[k_0,k_1]}
&\leq
C_M\max_{k_0\leq j\leq k_1}\norm{e_j}
\omega_{\mathbf x}(t_{k_0},t_{k_1})^{1/p}
+\tnorm{R^{e,\Delta}}_{q\text{-var},\Delta;[k_0,k_1]}.
\end{align}
Using
$\max_{k_0\leq j\leq k_1}\norm{e_j}\leq\norm{e_{k_0}}+X$,
from \eqref{eq5.33-2}-\eqref{eq5.33-3}
we conclude that
\[X
 \leq
 C_M\Bigl(
 \norm{e_{k_0}}
 +(t_{k_1}-t_{k_0})
 \max_{k_0-m\leq j\leq k_1-m}\norm{e_j}
 +\sum_{j=k_0}^{k_1-1}\norm{r_j}
 \Bigr)+
 C_M\omega_{\mathbf x}(t_{k_0},t_{k_1})^{1/p}X.\]
Then \eqref{eq5.27} follows by choosing $\gamma_M$
sufficiently small that $C_M\gamma_M^{1/p}\leq1/2$.

\end{proof}

\begin{lemma}\label{lem5.6}
Under the assumptions and notation of Lemma~\ref{lem5.5},
and after decreasing $\gamma_M$ if necessary, the following
estimate holds:
\begin{equation}\label{eq5.34}
\max_{k_0\leq k\leq k_1}\norm{e_k}
\leq
2\norm{e_{k_0}}
+C_M(t_{k_1}-t_{k_0})
\max_{k_0-m\leq j\leq k_1-m}\norm{e_j}
+C_M\sum_{j=k_0}^{k_1-1}\norm{r_j}.
\end{equation}
\end{lemma}

\begin{proof}
By subtracting the Euler recursion \eqref{eq:euler} for
$z^\Delta$ from the perturbed recursion for $y^\Delta$
and summing over $j=k_0,\ldots,\ell-1$, we obtain that
\begin{align}
 &e_\ell-e_{k_0}
 =\sum_{j=k_0}^{\ell-1}\Delta
 \bigl[-\mathcal Le_j-\lambda e_j
 +F(y_j^\Delta,y_{j-m}^\Delta)-F(z_j^\Delta,z_{j-m}^\Delta)\bigr]\notag\\
 &+\sum_{j=k_0}^{\ell-1}
 \bigl([h(y_j^\Delta)-h(z_j^\Delta)]x_j+[H(y_j^\Delta)-H(z_j^\Delta)]\mathbb X_j\bigr)
 +\sum_{j=k_0}^{\ell-1}r_j
 :=\mathcal{I}_{1}(\ell)+\mathcal{I}_{2}
 (\ell)
 +\mathcal{I}_{3}(\ell).
 \label{eq5.35}
\end{align}
By the boundedness of $\mathcal L$ and the Lipschitz
continuity of $F$, it holds that
\[\|\mathcal{I}_{1}(\ell)\|\leq c(t_{k_1}-t_{k_0})\max_{k_0\leq j\leq k_1}\norm{e_j}+L_1
(t_{k_1}-t_{k_0})
\max_{k_0-m\leq j\leq k_1-m}\norm{e_j}.\]
Using \eqref{eq5.27}, \eqref{eq5.33-1}, the Lipschitz
continuity of $h$ and $H$, and
$\omega_{\mathbf x}(t_{k_0},t_{k_1})\leq1$, we have
\begin{align*}
\max_{k_0\leq \ell\leq k_1}\|\mathcal{I}_{2}(\ell)\|
&+
\max_{k_0\leq \ell\leq k_1}\|\mathcal{I}_{3}(\ell)\|\leq
C_M\omega_{\mathbf x}(t_{k_0},t_{k_1})^{1/p}
\max_{k_0\leq j\leq k_1}\norm{e_j}\\
&+C_M(t_{k_1}-t_{k_0})
\max_{k_0-m\leq j\leq k_1-m}\norm{e_j}
+C_M\sum_{j=k_0}^{k_1-1}\norm{r_j}.
\end{align*}
Taking the maximum in \eqref{eq5.35} and using
$t_{k_1}-t_{k_0}\leq
\omega_{\mathbf x}(t_{k_0},t_{k_1})\leq1$, we deduce that
\begin{align*}
 \max_{k_0\leq j\leq k_1}\norm{e_j}
 &\leq
 \norm{e_{k_0}}
 +C_M\omega_{\mathbf x}(t_{k_0},t_{k_1})^{1/p}
 \max_{k_0\leq j\leq k_1}\norm{e_j}\\
 &+
 C_M(t_{k_1}-t_{k_0})
 \max_{k_0-m\leq j\leq k_1-m}\norm{e_j}
 +C_M\sum_{j=k_0}^{k_1-1}\norm{r_j}.
 \end{align*}
Then \eqref{eq5.34} follows by choosing $\gamma_M$
sufficiently small that $C_M\gamma_M^{1/p}\leq1/2$.
\end{proof}

\begin{proposition}\label{prop5.7}
Let $T>0$, $S\subset\Hh$ be compact, and let
\(\beta>1/p\) be as in Assumption~\textbf{(A4)}. Fix
$\omega\in\Omega_0$ and $\alpha\in(1/3,\beta)$.
Then there exists a
constant $C_{S,T}(\omega)>0$, independent of $\Delta$,
such that, for all sufficiently small $\Delta=\tau/m$,
\begin{align*}
 \sup_{\xi\in S}\max_{0\leq k\Delta\leq T}
 \norm{J_\Delta\Phi_\Delta(k,\omega)P_\Delta\xi
 -\Phi(k\Delta,\omega)\xi}_{\Hh}
 \leq C_{S,T}(\omega)
 \left(\omega_S(\Delta)+
 \Delta^\alpha+\Delta^{3\alpha-1}\right),
\end{align*}
where
$P_\Delta:\Hh\to\Hh_\Delta$ and
$J_\Delta:\Hh_\Delta\to\Hh$
denote the sampling and piecewise linear interpolation operators, respectively, $\Phi_\Delta(k,\omega)$ is the discrete cocycle satisfying
\eqref{eq:discrete-cocycle},
$\omega_S(\delta)
 :=
 \sup_{\xi\in S}
 \sup_{r,s\in[-\tau,0],|r-s|\leq\delta}
 \norm{\xi(r)-\xi(s)}$.
\end{proposition}

\begin{proof}
Fix $\xi\in S$ and $\omega\in\Omega_0$. Let
$u(\cdot;\xi,\omega)$ denote the exact solution of
\eqref{eq:vector} with initial history $\xi$, and let
$(u_k^\Delta)_{k\geq-m}$ satisfy \eqref{eq:euler} with $u_k^\Delta=\xi(t_k)$,
$-m\leq k\leq0$.
For simplicity, write
$u(t_k):=u(t_k;\xi,\omega)$
and define the grid error by
$e_k:=u_k^\Delta-u(t_k)$ for
$k\geq-m$.
Since
$u_k^\Delta=\xi(t_k)=u(t_k)$ for $-m\leq k\leq0$,
we have $e_k=0$ on the initial grid. By the continuity of
the solution map and the compactness of
$[-\tau,T]\times S$, we have
\[
 M_T(\omega):=2+\sup_{\xi\in S}\sup_{-\tau\leq t\leq T}\norm{u(t;\xi,\omega)}<\infty.
\]

Define
$k_*:=\inf\left\{
0\leq k\leq\lfloor T/\Delta\rfloor:
\norm{e_k}\geq1
\right\}$,
with $k_*=\infty$ if the set is empty.
For $k<k_*$, we have
$\norm{u_k^\Delta}\leq M_T-1$.
In addition, we can verify that
$\norm{u_{k_*}^\Delta}\leq M_T(\omega)$ whenever $k_*\leq\lfloor T/\Delta\rfloor$.
Thus
\[\max_{-m\leq k\leq k_*\wedge\lfloor T/\Delta\rfloor}\|u_k^\Delta\|\leq M_T,
~~~~
\max_{-m\leq k\leq k_*\wedge\lfloor T/\Delta\rfloor}\|u(t_k)\|\leq M_T.
\]

Choose $0<\gamma\leq\gamma_{M_T}$, where $\gamma_{M_T}$
is the threshold in Lemma~\ref{lem5.5}. By
Lemma~\ref{lem5.3}, after inserting all points of the form
$i\tau$ lying in $[0,t_{k_*\wedge\lfloor T/\Delta\rfloor}]$,
the interval $[0,t_{k_*\wedge\lfloor T/\Delta\rfloor}]$ admits a
partition
$0=k_0<k_1<\cdots<k_N=k_*\wedge\lfloor T/\Delta\rfloor$
such that, with
$I_j=[t_{k_{j-1}},t_{k_j}]$,
\[
 \omega_{\mathbf x}(t_{k_{j-1}},t_{k_j})
 \leq\gamma,
 \qquad j=1,\ldots,N,
\]
and $N$ is bounded independently of $\Delta$.
Moreover, by Lemma \ref{lem5.6}, we obtain that
\begin{equation}\label{eq5.38}
\max_{k_{j-1}\leq k\leq k_j}\norm{e_k}
\leq
\begin{cases}
\begin{aligned}
&C|I_1|
\max_{k_0-m\leq k\leq k_1-m}\norm{e_k}
+C\displaystyle
\sum_{k=k_0}^{k_1-1}\norm{\varrho_k},
\end{aligned}
& j=1,\\[3ex]
\begin{aligned}
&2\max_{k_{j-2}\leq k\leq k_{j-1}}\norm{e_k}+C|I_j|
\max_{k_{j-1}-m\leq k\leq k_j-m}\norm{e_k}
+C\displaystyle\sum_{k=
k_{j-1}}^{k_j-1}\norm{\varrho_k},
\end{aligned}
& 2\leq j\leq N.
\end{cases}
\end{equation}

For $I_{j}\subset [0,\tau]$, we have
$\max_{k_{j-1}-m\leq k\leq k_j-m}\norm{e_k}=0$. Iterating \eqref{eq5.38} over all blocks contained in
$[0,\tau]$, we obtain that for every $j$ such that
$I_j\subset[0,\tau]$,
\[\max_{k_{j-1}\leq k\leq k_j}\norm{e_k}
 \leq
 C\sum_{\ell=1}^{j}2^{j-\ell}
 \sum_{k=k_{\ell-1}}^{k_\ell-
 1}\|\varrho_k\|.\]
Using Lemma~\ref{lem5.1} and the fact that $N$ is bounded independently of $\Delta$, we deduce that
 \begin{align*}
 \max_{0\leq k
 \leq\min\{m,k_*\wedge\lfloor T/\Delta\rfloor\}}\|e_k\|
 &\leq
 C_{S,T}(\omega)
 \sum_{k=0}^{\min\{m,k_*\wedge\lfloor T/\Delta\rfloor\}-1}\|\varrho_k\|\nonumber\\
 &\leq C_{S,T}(\omega)
 \bigl(\omega_S(\Delta)+\Delta^\alpha
 +\Delta^{3\alpha-1}\bigr).
\end{align*}
Assume inductively that, for some integer $i\geq1$,
\begin{align}\label{eq5.38-2}
 \max_{0\leq k\leq\min\{im,k_*\wedge\lfloor T/\Delta\rfloor\}}\norm{e_k}
 \leq
 C_{S,T,i}(\omega)\left(
 \omega_S(\Delta)+\Delta^\alpha+\Delta^{3\alpha-1}
 \right).
\end{align}
For every block contained in
$[i\tau,(i+1)\tau]$, its delayed indices lie in the
preceding interval $[(i-1)\tau,i\tau]$. Hence
$\max_{k_{j-1}-m\leq k\leq k_j-m}\norm{e_k}$ is bounded by the
right-hand side of \eqref{eq5.38-2}.
Iterating \eqref{eq5.38} over all blocks
$I_j\subset[i\tau,(i+1)\tau]$ and using \eqref{eq5.38-2} to control the delayed-error terms, we have
\begin{align*}
&\max_{0\leq k\leq\min\{(i+1)m,
k_*\wedge\lfloor T/\Delta\rfloor\}}\norm{e_k}
\leq
C\max_{0\leq k\leq\min\{im,
k_*\wedge\lfloor T/\Delta\rfloor\}}\norm{e_k}\\
&\quad
+C_{S,T,i}(\omega)
\sum_{I_j\subset[i\tau,(i+1)\tau]}|I_j|
\left(
 \omega_S(\Delta)+\Delta^\alpha+\Delta^{3\alpha-1}
 \right)
+C\sum_{k=im}^{\min\{(i+1)m,
k_*\wedge\lfloor T/\Delta\rfloor\}-1}\norm{\varrho_k}\\
&\leq
 C_{S,T,i+1}(\omega)\left(
 \omega_S(\Delta)+\Delta^\alpha+\Delta^{3\alpha-1}
 \right),
\end{align*}
where in the last inequality we have used Lemma~\ref{lem5.1}, $\sum_{I_j\subset[i\tau,(i+1)\tau]}
|I_j|\leq\tau$ and the fact that $N$ is bounded independently of $\Delta$.
This completes the induction step.
Since at most $\lceil T/\tau\rceil$ such
steps are required, we conclude that
\begin{equation}\label{eq5.39}
 \max_{0\leq k\leq k_*\wedge\lfloor T/\Delta\rfloor}\norm{e_k}
 \leq
 C_{S,T}(\omega)\left(
 \omega_S(\Delta)+\Delta^\alpha+\Delta^{3\alpha-1}
 \right).
\end{equation}
In view of $\alpha>1/3$, we have
$\omega_S(\Delta)+\Delta^\alpha+\Delta^{3\alpha-1}
 \to0$ as $\Delta\to0$. If
$k_*\leq\lfloor T/\Delta\rfloor$, then
\eqref{eq5.39} would give
$\norm{e_{k_*}}<1$ for all sufficiently small $\Delta$,
contradicting the definition of $k_*$. Consequently,
$k_*=\infty$, and hence
\begin{equation}\label{eq:grid-error}
 \max_{0\leq k\Delta\leq T}\norm{e_k}
 \leq
 C_{S,T}(\omega)\left(
 \omega_S(\Delta)+\Delta^\alpha+\Delta^{3\alpha-1}
 \right).
\end{equation}

Fix $0\leq k\Delta\leq T$ and
$\vartheta\in[-\tau,0]$, and set
$s=t_k+\vartheta$. Choose
$\ell\in\{k-m,\ldots,k-1\}$ such that
$s\in[t_\ell,t_{\ell+1}]$, and let
$\eta:=\frac{s-t_\ell}{\Delta}\in[0,1]$.
Recall that
$\mathbf u_k^\Delta
:=(u_{k-m}^\Delta,\ldots,u_k^\Delta)
\in U^{m+1}$
represents the discrete history of the delay scheme \eqref{eq:euler} at time $t_k$.
By the definition of $J_\Delta$,
\[
 (J_\Delta\mathbf u_k^\Delta)(\vartheta)
 =
 (1-\eta)u_\ell^\Delta+\eta u_{\ell+1}^\Delta.
\]
Therefore, using $u_\ell^\Delta=u(t_\ell)+e_\ell$ and
$u_{\ell+1}^\Delta=u(t_{\ell+1})+e_{\ell+1}$, we obtain that
 \begin{align*}
 \norm{(J_\Delta\mathbf u_k^\Delta)(\vartheta)-u(s)}
&\leq\max\{\norm{e_\ell},\norm{e_{\ell+1}}\}
 +\sup_{r,\tilde{r}\in[t_\ell,t_{\ell+1}]}
 \norm{u(r)-u(\tilde{r})}\\
 &\leq C_{S,T}(\omega)
 \left(\omega_S(\Delta)+\Delta^\alpha+\Delta^{3\alpha-1}\right),
\end{align*}
where in the last inequality we have used \eqref{eq:grid-error} and \eqref{eq:positive-delay-increment}.
Consequently,
\begin{align*}
&\sup_{\xi\in S}\max_{0\leq k\Delta\leq T}
 \norm{J_\Delta\Phi_\Delta(k,\omega)P_\Delta\xi
 -\Phi(k\Delta,\omega)\xi}_{\Hh}\\
&=\sup_{\xi\in S}\max_{0\leq k\Delta\leq T}\sup_{\vartheta\in[-\tau,0]}
\norm{(J_\Delta\mathbf u_k^\Delta)(\vartheta)-u(t_k+\vartheta)}
\leq C_{S,T}(\omega)
 \left(\omega_S(\Delta)+
 \Delta^\alpha+\Delta^{3\alpha-1}\right).
\end{align*}
This completes the proof.
\end{proof}

Let $\mu_k^\Delta$ solve the deterministic Euler recursion
\begin{equation}\label{eq:det-euler}
 \mu_{k+1}^\Delta
 =\mu_k^\Delta+\Delta
 \bigl[-\mathcal L\mu_k^\Delta-\lambda\mu_k^\Delta
 +F(\mu_k^\Delta,\mu_{k-m}^\Delta)\bigr].
\end{equation}
There exists a constant $C_*>0$, independent of
$\Delta$, such that
\begin{equation}\label{eq:G-square}
 \norm{-\mathcal L\mu_k^\Delta-\lambda\mu_k^\Delta
 +F(\mu_k^\Delta,\mu_{k-m}^\Delta)}^2
 \leq
 C_*
 \bigl(1+\norm{\mu_k^\Delta}^2+
 \norm{\mu_{k-m}^\Delta}^2\bigr).
\end{equation}
Recall from Assumption~\textbf{(A2)} that
$\lambda>L_0+L_1$, and fix $\sigma>0$ such that
$L_1e^{\sigma\tau}<\lambda-L_0$. Then there exist constants $\Delta_0>0$ and $c_0>0$
such that, for every $0<\Delta\leq\Delta_0$,
\begin{equation}
\label{eq:discrete-coercivity-constants}
 \lambda-L_0-C_*\Delta
 -
 (L_1+C_*\Delta)e^{\sigma\tau}
 \geq c_0,
 \qquad
 \frac{1-e^{-\sigma\Delta}}{\Delta}
 \geq\frac{\sigma}{2}.
\end{equation}

\begin{lemma}\label{lem5.8}
Let $\Delta_0$, $\sigma$, $C_*$, and $c_0$ be as specified
above, and set
$c:=\min\left\{c_0,\frac{\sigma}{2}\right\}$.
Let $m\in\mathbb N$ and $\Delta=\tau/m$ satisfy
$0<\Delta\leq\Delta_0$,
$\Delta\leq1$ and
$1-c\Delta\geq0$.
For the sequence $(\mu_k^\Delta)_{k\geq-m}$ defined above, set
\begin{equation}\label{eq:discrete-energy}
 \mathcal E_k^\Delta
 :=
 \norm{\mu_k^\Delta}^2
 +(L_1+C_*\Delta)e^{\sigma(\tau+\Delta)}
 \Delta
 \sum_{j=k-m}^{k-1}
 e^{-\sigma(t_k-t_j)}
 \norm{\mu_j^\Delta}^2,
 \qquad k\geq0.
\end{equation}
Then there exists a constant $C>0$, independent of $\Delta$,
such that
\begin{equation}\label{eq:discrete-deterministic-decay}
 \mathcal E_k^\Delta
 \leq
 e^{-ct_k}\mathcal E_0^\Delta+\frac{C}{c},
 \qquad k\geq0.
\end{equation}
\end{lemma}

\begin{proof}
Using \eqref{eq:G-square} and arguing as in
\eqref{eq4.2-1}, we have
\begin{align}
 \norm{\mu_{k+1}^\Delta}^2
 &=
 \norm{\mu_k^\Delta}^2
 +2\Delta
 \ip{
 -\mathcal L\mu_k^\Delta-\lambda\mu_k^\Delta
 +F(\mu_k^\Delta,\mu_{k-m}^\Delta)
 }{\mu_k^\Delta}\notag\\
 &\quad+
 \Delta^2
 \norm{
 -\mathcal L\mu_k^\Delta-\lambda\mu_k^\Delta
 +F(\mu_k^\Delta,\mu_{k-m}^\Delta)
 }^2\notag\\
 &\leq
 \bigl[1-(\lambda-L_0-C_*\Delta)\Delta\bigr]
 \norm{\mu_k^\Delta}^2
 +(L_1+C_*\Delta)\Delta
 \norm{\mu_{k-m}^\Delta}^2
 +C\Delta.
 \label{eq5.40}
\end{align}
Denote the second term on the right-hand side of
\eqref{eq:discrete-energy} by $M_k^\Delta$. It holds that
\begin{align}\label{eq5.41}
 M_{k+1}^\Delta-M_k^\Delta
 =
 -(1-e^{-\sigma\Delta})M_k^\Delta
 +(L_1+C_*\Delta)e^{\sigma\tau}\Delta
 \norm{\mu_k^\Delta}^2
 -
 (L_1+C_*\Delta)\Delta
 \norm{\mu_{k-m}^\Delta}^2.
\end{align}
Adding \eqref{eq5.40} and \eqref{eq5.41} yields
\begin{align}
 \mathcal E_{k+1}^\Delta-\mathcal E_k^\Delta
 &\leq
 -\Delta
 \bigl[
 \lambda-L_0-C_*\Delta
 -(L_1+C_*\Delta)e^{\sigma\tau}
 \bigr]
 \norm{\mu_k^\Delta}^2\notag\\
 &\quad-
 (1-e^{-\sigma\Delta})M_k^\Delta
 +C\Delta
 \leq
-c\Delta\mathcal E_k^\Delta+C\Delta,
 \label{eq5.42}
\end{align}
where we have used
\eqref{eq:discrete-coercivity-constants} and $0<c=
 \min\left\{c_0,\frac{\sigma}{2}\right\}$.
Iterating
\eqref{eq5.42} and using
$(1-c\Delta)^k
 \leq e^{-ck\Delta}
 =e^{-ct_k}$,
we obtain  \eqref{eq:discrete-deterministic-decay}.
\end{proof}

\begin{lemma}\label{lem5.10}
Let \(r>\tau\) and \(\varepsilon\in(0,1/4)\), and let $\Delta_0$ be as specified above, with
$\Delta_0\leq(\tau\wedge1)/2$. For \(\Delta=\tau/m\in(0,\Delta_0]\),
let \(u^\Delta=(u_k^\Delta)\) and
\(\mu^\Delta=(\mu_k^\Delta)\) satisfy
\eqref{eq:euler} and \eqref{eq:det-euler}, respectively,
with the same initial history \(\xi\in\Hh\). Then there
exist constants \(C_r>0\) and \(C_{r,\varepsilon}>0\),
independent of \(\Delta\) and \(\xi\), such that
\begin{align*}
 \max_{0\leq k\Delta\leq r}
 \norm{u_k^\Delta-\mu_k^\Delta}
 \leq
 C_r\varepsilon\bigl(1+\norm{\xi}_{\Hh}\bigr)
 +C_{r,\varepsilon}
 \left(
 1+\norm{x}_{\pvar,[0,r]}^{p}
 +\norm{\mathbb X}_{\qvar,[0,r]}^{q}
 \right).
\end{align*}
\end{lemma}

\begin{proof}

\textbf{Step~1.}
Let $\gamma_0$ be as in Lemma~\ref{lem3.2}.
Choose $\gamma_\varepsilon\in(0,\gamma_0]$ such that
$r\gamma_\varepsilon^{1/p}\leq\varepsilon$
and $\gamma_\varepsilon<1$.
Set $t_{j}=j\Delta$ and
\[\mathcal B:=
 \left\{
 0\leq j<\lfloor r/\Delta\rfloor:
 \omega_{\mathbf x}(t_j,t_{j+1})
 >\gamma_\varepsilon/2
 \right\}.\]
By the superadditivity of $\omega_{\mathbf x}$, we have
$\sum_{j\in\mathcal B}
 \omega_{\mathbf x}(t_j,t_{j+1})
 \leq \omega_{\mathbf x}(0,r)$.
Consequently,
\[\#\mathcal B
 \leq 2\gamma_\varepsilon^{-1}\omega_{\mathbf x}(0,r).\]
Retain each step indexed by $\mathcal B$ as a separate
block, and apply the discrete greedy partition at level
$\gamma_\varepsilon$ to each remaining consecutive run
of grid steps. Subdivide the resulting greedy blocks
at delay-grid points $n\tau$ and, if necessary, at
additional grid points so that their lengths do not
exceed $\tau\wedge1$. We thus obtain a partition
\[
 0=j_0<j_1<\cdots<j_N=\lfloor r/\Delta\rfloor,
\]
such that each block
$I_\ell=[t_{j_\ell},t_{j_{\ell+1}}]$
is either a single step indexed by $\mathcal B$
or a block indexed by $\ell\in\mathcal G$, where
$\mathcal G$ denotes the set of indices of the
refined greedy blocks, satisfying
\[
 |I_\ell|\leq\tau\wedge1,
 \qquad
 \rho_\ell:=
 \omega_{\mathbf x}(t_{j_\ell},t_{j_{\ell+1}})^{1/p}
 \leq\gamma_\varepsilon^{1/p},
 \qquad \ell\in\mathcal G.
\]
By the discrete greedy-partition argument, the
superadditivity of $\omega_{\mathbf x}$, and
the bound on $\#\mathcal B$, we conclude that
\begin{equation}\label{eq5.44}
 N
 \leq C_{r,\varepsilon}
 \left(
 1+\tnorm{x}_{\pvar,[0,r]}^p
 +\tnorm{\mathbb X}_{\qvar,[0,r]}^q
 \right).
\end{equation}
By \eqref{eq2.40}, \eqref{eq:discrete-deterministic-decay}
and $\mathcal E_0^\Delta\leq C\norm{\xi}_{\Hh}^{2}$, we further obtain that
\begin{equation}\label{eq5.45}
 \max_{0\leq j\Delta\leq r}\norm{\mu_j^\Delta}
 \leq C_r(1+\norm{\xi}_{\Hh})
 ~~~\mbox{and}~~~
 \sum_{j=j_\ell}^{j_{\ell+1}-1}
 \norm{\mu_{j+1}^\Delta-\mu_j^\Delta}
 \leq
 C_r(1+\norm{\xi}_{\Hh})
 (t_{j_{\ell+1}}-t_{j_\ell}).
\end{equation}

\textbf{Step~2.}
Fix $\ell\in\mathcal G$ and define
$z_j
 :=
 \psi_{t_{j_\ell},t_j}^{-1}(u_j^\Delta)$ for
$j\in [j_\ell,j_{\ell+1}]$.
For \(j_\ell\leq j<j_{\ell+1}\), let
\begin{equation}\label{eq5.46}
 \varepsilon_j^{\mathrm{fl}}
 :=
 u_j^\Delta+h(u_j^\Delta)x_j
 +H(u_j^\Delta)\mathbb X_j
 -\psi_{t_j,t_{j+1}}(u_j^\Delta).
\end{equation}
Applying the sewing estimate \eqref{eq2.12} to
\eqref{eq5.46}, together with the uniform flow
bounds in Lemma~\ref{lem3.2}, gives
\begin{align}
\label{eq5.47}
&\norm{\varepsilon_j^{\mathrm{fl}}}
=\norm{\int_{t_j}^{t_{j+1}}
h(\psi_{t_j,r}(u_j^\Delta))\,d\mathbf x_r-h(u_j^\Delta)x_j
 -H(u_j^\Delta)\mathbb X_j}\nonumber\\
 &\leq
C_p\Bigl(
 \tnorm{R^{h(\psi_{t_j,\cdot}(u_j^\Delta))}}_{q\text{-var},[t_j,t_{j+1}]}
 \tnorm{x}_{p\text{-var},[t_j,t_{j+1}]}
 +
 \tnorm{H(\psi_{t_j,\cdot}(u_j^\Delta))}_{p\text{-var},[t_j,t_{j+1}]}
 \tnorm{\mathbb X}_{q\text{-var},[t_j,t_{j+1}]}
\Bigr)\nonumber\\
 &\leq
 C\omega_{\mathbf x}(t_j,t_{j+1})^{3/p}.
\end{align}

Let $\Psi_{j+1}
 :=
 \psi_{t_{j_\ell},t_{j+1}}^{-1}$.
Then $
 \Psi_{j+1}
 \bigl(\psi_{t_j,t_{j+1}}(u_j^\Delta)\bigr)
 =z_j$.
 By \eqref{eq:euler} and \eqref{eq5.46}, we find that
\[
 u_{j+1}^\Delta
 =
 \psi_{t_j,t_{j+1}}(u_j^\Delta)
 +\Delta G(u_j^\Delta,u_{j-m}^\Delta)
 +\varepsilon_j^{\mathrm{fl}}.
\]
Consequently,
\begin{align}
\label{eq5.48}
 &z_{j+1}-z_j
=
 \Psi_{j+1}\Bigl(
 \psi_{t_j,t_{j+1}}(u_j^\Delta)
 +\Delta G(u_j^\Delta,u_{j-m}^\Delta)
 +\varepsilon_j^{\mathrm{fl}}
 \Bigr)
 -
 \Psi_{j+1}\Bigl(
 \psi_{t_j,t_{j+1}}(u_j^\Delta)
 \Bigr)\nonumber\\
 &=\Delta\int_0^1
 D\Psi_{j+1}
 \Bigl(
 \psi_{t_j,t_{j+1}}(u_j^\Delta)
 +\theta\bigl[
 \Delta G(u_j^\Delta,u_{j-m}^\Delta)
 +\varepsilon_j^{\mathrm{fl}}
 \bigr]
 \Bigr)
 G(u_j^\Delta,u_{j-m}^\Delta)d\theta
 \nonumber\\
 &~~+\int_0^1
 D\Psi_{j+1}
 \Bigl(
 \psi_{t_j,t_{j+1}}(u_j^\Delta)
 +\theta\bigl[
 \Delta G(u_j^\Delta,u_{j-m}^\Delta)
 +\varepsilon_j^{\mathrm{fl}}
 \bigr]
 \Bigr)
 \varepsilon_j^{\mathrm{fl}}\,d\theta
 :=\Delta\widetilde G_j+d_j.
\end{align}
Thanks to \eqref{eq3.5}-\eqref{eq3.6}, \eqref{eq2.40} and
\eqref{eq5.47}, it holds that
\begin{align}
 \norm{\widetilde G_j-G(z_j,u_{j-m}^\Delta)}
 &\leq
 C\rho_\ell
 \bigl(1+\norm{z_j}+\norm{u_{j-m}^\Delta}\bigr),
 \nonumber\\
 \norm{d_j}
 &\leq
 C\omega_{\mathbf x}(t_j,t_{j+1})^{3/p}.
 \label{eq5.50}
\end{align}
Since \(3/p>1\) and \(\omega_{\mathbf x}\) is
superadditive,
\begin{align}\label{eq5.51}
 \sum_{j=j_\ell}^{j_{\ell+1}-1}\norm{d_j}
 \leq
 C\sum_{j=j_\ell}^{j_{\ell+1}-1}
 \omega_{\mathbf x}(t_j,t_{j+1})^{3/p}
 \leq
 C\omega_{\mathbf x}
 (t_{j_\ell},t_{j_{\ell+1}})^{3/p}
 =
 C\rho_\ell^3.
\end{align}
Subtracting \eqref{eq:det-euler} from
\eqref{eq5.48},
using \eqref{eq5.45},
\eqref{eq5.50}, and the Lipschitz
continuity of \(G\), we have
\begin{align}\label{eq5.52}
 \norm{z_{j+1}-\mu_{j+1}^\Delta}
 &\leq
 \norm{z_j-\mu_j^\Delta}
 +\Delta\norm{
 \widetilde G_j
 -G(\mu_j^\Delta,\mu_{j-m}^\Delta)}+
 \|d_j\|\\
 &\leq (1+C\Delta)\norm{z_j-\mu_j^\Delta}
 +C\Delta\norm{u_{j-m}^\Delta-
 \mu_{j-m}^\Delta}
 +C_r\Delta\rho_\ell (1+\norm{\xi}_{\Hh})
 +\norm{d_j}.\nonumber
\end{align}
Since \(|I_\ell|\leq\tau=m\Delta\), we have $j-m\leq j_\ell$ for \(j_\ell\leq j<j_{\ell+1}\). Notice that $\|u_j^\Delta-\mu_j^\Delta\|
 \leq
 \|z_j-\mu_j^\Delta\|+C\rho_\ell$,
\[
\|u_{j-m}^\Delta-\mu_{j-m}^\Delta\|\leq \max_{-m\leq i\leq j_\ell}
 \norm{u_i^\Delta-\mu_i^\Delta}~~
 \mbox{and}~~
 \|z_{j_\ell}-\mu_{j_\ell}^\Delta\|
\leq \max_{-m\leq i\leq j_\ell}
 \norm{u_i^\Delta-\mu_i^\Delta}.
\]
Iterating \eqref{eq5.52} over
\(I_\ell\) and using \eqref{eq5.51}, we obtain that
\begin{equation}\label{eq5.53}
\max_{-m\leq i\leq j_{\ell+1}}
 \norm{u_i^\Delta-\mu_i^\Delta}
 \leq
 (1+C|I_\ell|)\max_{-m\leq i\leq j_\ell}
 \norm{u_i^\Delta-\mu_i^\Delta}
 +C\rho_\ell
 +C_r\rho_\ell|I_\ell|(1+\norm{\xi}_{\Hh}).
\end{equation}

\textbf{Step~3.}
For a single step indexed by $j\in\mathcal B$,
subtracting \eqref{eq:det-euler} from
\eqref{eq:euler}, and using the Lipschitz continuity
of $G$ and Assumption \textbf{(A3)},
we have
\begin{align*}
 \norm{u_{j+1}^\Delta-\mu_{j+1}^\Delta}
 \leq
 (1+C\Delta)\norm{u_j^\Delta
 -\mu_j^\Delta}
 +
 C\Delta
 \norm{u_{j-m}^\Delta-\mu_{j-m}^\Delta}
 +C\bigl(\norm{x_j}+\norm{\mathbb X_j}\bigr).
\end{align*}
Hence, after increasing $C$ if necessary,
\begin{equation}\label{eq5.53-1}
 \max_{-m\leq i\leq j+1}
 \norm{u_i^\Delta-\mu_i^\Delta}
 \leq
 (1+C\Delta)
 \max_{-m\leq i\leq j}
 \norm{u_i^\Delta-\mu_i^\Delta}
 +C\bigl(\norm{x_j}+\norm{\mathbb X_j}\bigr).
\end{equation}
Since $u^\Delta$ and $\mu^\Delta$ have the same
initial history, iterating \eqref{eq5.53} and
\eqref{eq5.53-1} over all blocks and using
$\sum_{\ell=0}^{N-1}|I_\ell|\leq r$, we obtain that
\begin{align*}
 \max_{0\leq j\Delta\leq r}
 \norm{u_j^\Delta-\mu_j^\Delta}
 &\leq
 C_r(1+\norm{\xi}_{\Hh})
 \sum_{\ell\in\mathcal G}\rho_\ell|I_\ell|
 +C_r\sum_{\ell\in\mathcal G}\rho_\ell
 +
 C_r\sum_{j\in\mathcal B}
 \bigl(\norm{x_j}+\norm{\mathbb X_j}\bigr)\\
  &\leq
 C_r\varepsilon(1+\norm{\xi}_{\Hh})
 +
 C_{r,\varepsilon}
 \left(
 1+\tnorm{x}_{\pvar,[0,r]}^{p}
 +\tnorm{\mathbb X}_{\qvar,[0,r]}^{q}
 \right),
\end{align*}
where we have used \eqref{eq5.44},
\[
 \sum_{\ell\in\mathcal G}\rho_\ell|I_\ell|
 \leq r\gamma_\varepsilon^{1/p}\leq\varepsilon,
 \qquad
 \sum_{\ell\in\mathcal G}\rho_\ell
 \leq\gamma_\varepsilon^{1/p}N,
\]
and, by H\"older's inequality and the bound on
$\#\mathcal B$,
\begin{align*}
 \sum_{j\in\mathcal B}
 \bigl(\norm{x_j}+\norm{\mathbb X_j}\bigr)
 &\leq
 (\#\mathcal B)^{1-1/p}
 \tnorm{x}_{\pvar,[0,r]}
 +
 (\#\mathcal B)^{1-1/q}
 \tnorm{\mathbb X}_{\qvar,[0,r]}
 \leq
 C_{r,\varepsilon}
 \bigl(1+\omega_{\mathbf x}(0,r)\bigr).
\end{align*}
The same estimate holds when $\mathcal B$ is empty.
This completes the proof.

\end{proof}

The following proposition establishes the discrete counterparts of
Proposition~\ref{prop4.3}, Theorem~\ref{thm4.6}, and
Lemma~\ref{lem4.7}.

\begin{proposition}\label{prop5.11}
Let $\Delta_0$ be as specified above.
There exist an integer $M\geq2$ and a constant
$\eta_0\in(0,1)$ such that,
for every admissible $\Delta=\tau/m\in(0,\Delta_0]$,
$\omega\in\Omega_0$ and
$\zeta\in\Hh_\Delta$,
\begin{equation}\label{eq:discrete-contraction}
 \norm{\Phi_\Delta(n_r,\omega)\zeta}_{\Hh_\Delta}^2
 \leq
 \eta_0\norm{\zeta}_{\Hh_\Delta}^2+
 P_r(\omega),
\end{equation}
where \(r:=M\tau\), $n_r:=\frac r\Delta=Mm$, $P_{r}(\omega):=C_{r,\varepsilon}
 \left(
 1+\norm{x}_{\pvar,[0,r]}^{2p}
 +\norm{\mathbb X}_{\qvar,[0,r]}^{2q}
 \right)$, $\varepsilon\in(0,1/4)$ is fixed sufficiently small.
Consequently, the discrete cocycles \(\Phi_\Delta\)
admit measurable closed tempered pullback absorbing
balls \(B_\Delta(\omega)\) and positively invariant
pathwise pullback absorbing hulls
\(\widehat B_\Delta(\omega)\) satisfying
\begin{equation}
\label{eq:uniform-discrete-radius}
 \sup_{
 0<\Delta\leq\Delta_0}
 \sup_{\zeta\in\widehat B_\Delta(\omega)}
 \norm{\zeta}_{\Hh_\Delta}^2
 \leq R_*(\omega),
\end{equation}
where \(R_*\) is a nonnegative tempered random variable
independent of \(\Delta\).
\end{proposition}

\begin{proof}
Fix an admissible $\Delta$ and $\zeta\in\Hh_\Delta$.
Apply Lemma~\ref{lem5.10} with initial history
$\xi=J_\Delta\zeta$, for which
$P_\Delta\xi=\zeta$ and
$\|\xi\|_{\Hh}=\|\zeta\|_{\Hh_\Delta}$.
Since \(t_k\geq r-\tau\) for \(n_r-m\leq k\leq n_r\), by
\eqref{eq:discrete-deterministic-decay} we have
\[\norm{\mu_{n_r+\cdot}^\Delta}_{\Hh_\Delta}
 =
 \max_{-m\leq \iota\leq0}
 \norm{\mu_{n_r+\iota}^\Delta}
 =\max_{n_r-m\leq k\leq n_r}
 \norm{\mu_{k}^\Delta}
 \leq C e^{-c(r-\tau)/2}\norm\zeta_{
 \Hh_\Delta}+C,\]
where we used
\((\mathcal E_0^\Delta)^{1/2}
\leq C(1+\norm{\zeta}_{\Hh_\Delta})\).
Combining this with Lemma~\ref{lem5.10}, we obtain that
\begin{align*}
 \norm{u_{n_r+\cdot}^\Delta}_{\Hh_\Delta}
 &\leq
 \norm{\mu_{n_r+\cdot}^\Delta}_{\Hh_\Delta}
 +
 \max_{0\leq k\Delta\leq r}
 \norm{u_k^\Delta-\mu_k^\Delta}\\
 &\leq
 \left[Ce^{-c(r-\tau)/2}+C_r\varepsilon\right]
 \norm{\zeta}_{\Hh_\Delta}
 +C_{r,\varepsilon}
 \left(
 1+\norm{x}_{\pvar,[0,r]}^{p}
 +\norm{\mathbb X}_{\qvar,[0,r]}^{q}
 \right).
\end{align*}
Choose \(M\geq2\) sufficiently large, with \(r=M\tau\), and
then choose \(\varepsilon\in(0,1/4)\) sufficiently small so that
$Ce^{-c(r-\tau)/2}+C_r
\varepsilon<1$.
Then there exists a constant $\eta_0\in(0,1)$ such that
$\|u_{n_r+\cdot}^\Delta\|_{\Hh_\Delta
}^2\leq \eta_0\norm{\zeta}_{\Hh_\Delta}^2
 +P_{r}(\omega)$,
which proves \eqref{eq:discrete-contraction}, since
$u_{n_r+\cdot}^\Delta=
\Phi_\Delta(n_r,\omega)\zeta$.

For the fixed \(r\) and \(\varepsilon\) chosen above,
it follows from \eqref{eq5.45} and Lemma~\ref{lem5.10} that there exists a positive tempered random variable \(C_r^*\), independent
of \(\Delta\) and \(\zeta\), whose orbit map
\(s\mapsto C_r^*(\theta_s\omega)\) is continuous, such that
\begin{align}
&\max_{0\leq k\Delta\leq r}
\norm{\Phi_\Delta(k,\omega)\zeta}_{\Hh_\Delta}^{2}
=
\max_{0\leq k\Delta\leq r}
\norm{u_{k+\cdot}^\Delta}_{\Hh_\Delta}^{2}
\nonumber\\
&\leq
2\max_{-m\leq j\leq n_r}
\norm{\mu_j^\Delta}^{2}
+
2\max_{0\leq j\leq n_r}
\norm{u_j^\Delta-\mu_j^\Delta}^{2}
\leq
C_r^*(\omega)
\left(
1+\norm{\zeta}_{\Hh_\Delta}^{2}
\right).
\label{eq5.58}
\end{align}
Define
\begin{align}
 R_{\mathrm d,*}(\omega)
 &:=2+2\sum_{j=0}^{\infty}\eta_0^j
 P_r(\theta_{-(j+1)r}\omega),
 \label{eq5.55}\\
 B_\Delta(\omega)
 &:=\{\zeta\in\Hh_\Delta:\norm\zeta_{\Hh_\Delta}^2\leq R_{\mathrm d,*}(\omega)\}.\nonumber
\end{align}
An argument analogous to that in the proof of
Theorem~\ref{thm4.6} shows that \(R_{\mathrm d,*}\) is
measurable and tempered and that
\(s\mapsto R_{\mathrm d,*}(\theta_s\omega)\) is continuous.
Let \(D_\Delta\) be an arbitrary discrete tempered family and
fix \(\omega\in\Omega_0\). For \(n\geq1\), take
$\zeta\in D_\Delta(\theta_{-nr}\omega)$.
By repeatedly applying \eqref{eq:discrete-contraction}, we obtain
\begin{align}
&\norm{
\Phi_\Delta(nn_r,\theta_{-nr}\omega)\zeta
}_{\Hh_\Delta}^{2}
\leq
\eta_0^n\norm{\zeta}_{\Hh_\Delta}^{2}
+
\sum_{j=0}^{n-1}
\eta_0^j
P_{r}\bigl(\theta_{-(j+1)r}\omega\bigr).
\label{eq:block-iteration}
\end{align}
Since \(D_\Delta\) is tempered and \(\eta_0\in(0,1)\), we have
\[
\eta_0^n
\sup_{\zeta\in D_\Delta(\theta_{-nr}\omega)}
\norm{\zeta}_{\Hh_\Delta}^{2}
\longrightarrow0
\qquad\text{as }n\to\infty.
\]
By the definition of \(R_{\mathrm d,*}\),
\begin{align*}
\sum_{j=0}^{n-1}
\eta_0^jP_{r}\bigl(\theta_{-(j+1)r}\omega\bigr)
\leq
\frac{R_{\mathrm d,*}(\omega)-2}{2}.
\end{align*}
It therefore follows from \eqref{eq:block-iteration} that there exists \(N_{D_\Delta}(\omega)\in\mathbb N\) such that for
every \(n\geq N_{D_\Delta}(\omega)\),
\begin{align*}
\sup_{\zeta\in D_\Delta(\theta_{-nr}\omega)}
\norm{
\Phi_\Delta(nn_r,\theta_{-nr}\omega)\zeta
}_{\Hh_\Delta}^{2}
\leq
R_{\mathrm d,*}(\omega).
\end{align*}
Consequently,
$\Phi_\Delta(nn_r,\theta_{-nr}\omega)
D_\Delta(\theta_{-nr}\omega)
\subseteq B_\Delta(\omega)$
for all sufficiently large \(n\).

For any $\iota\in\mathbb N_0$, write
$\iota=nn_r+\ell$, where $0\leq\ell<n_r$.
For $\zeta\in D_\Delta(\theta_{-\iota\Delta}\omega)$,
in view of the identity $\iota\Delta=nr+\ell\Delta$,
the cocycle property and \eqref{eq5.58}, we have
\begin{align*}
 &\Phi_\Delta(\iota,\theta_{-\iota\Delta}\omega)\zeta
 =
 \Phi_\Delta(nn_r,\theta_{-nr}\omega)
 \Phi_\Delta(\ell,\theta_{-\iota\Delta}\omega)\zeta,\\
 &\|\Phi_\Delta(\ell,\theta_{-\iota\Delta}\omega)
 \zeta\|_{\Hh_\Delta}^2
 \leq
 C_r^*(\theta_{-\iota\Delta}\omega)
 \bigl(1+\|\zeta\|_{\Hh_\Delta}^2\bigr).
\end{align*}
Applying \eqref{eq:block-iteration} yields
\begin{align*}
 \sup_{\zeta\in D_\Delta(\theta_{-\iota\Delta}\omega)}
 \|\Phi_\Delta(\iota,\theta_{-\iota\Delta}\omega)
 \zeta\|_{\Hh_\Delta}^2
\leq
 \eta_0^n C_r^*(\theta_{-\iota\Delta}\omega)
 \left(
 1+
 \sup_{\zeta\in D_\Delta(\theta_{-\iota\Delta}\omega)}
 \|\zeta\|_{\Hh_\Delta}^2
 \right)
 +\frac{R_{\mathrm d,*}(\omega)-2}{2}.
\end{align*}
By the temperedness of $C_r^*$ and $D_\Delta$,
the first term on the right-hand side tends to
zero as $n\to\infty$ for each fixed remainder
$\ell$. Since there are only finitely many
remainders, $B_\Delta$ pullback absorbs $D_\Delta$
for all sufficiently large integer times.

Moreover, \eqref{eq5.55} gives
$\eta_0R_{\mathrm d,*}(\omega)+P_{r}(\omega)
 \leq R_{\mathrm d,*}(\theta_r\omega)$.
It follows from \eqref{eq:discrete-contraction} that
\begin{equation}\label{eq5.57-00}
\Phi_\Delta(n_r,\omega)B_\Delta(\omega)
 \subseteq B_\Delta(\theta_r\omega).
 \end{equation}
Define
\begin{equation}\label{eq5.57}
 \widehat B_\Delta(\omega)
 :=
 \overline{
 \bigcup_{\ell=0}^{n_r-1}
 \Phi_\Delta(\ell,\theta_{-\ell\Delta}\omega)
 B_\Delta(\theta_{-\ell\Delta}\omega)
 }^{\Hh_\Delta}.
\end{equation}
By \eqref{eq5.57-00}, the cocycle identity, and the
pullback absorbing property of \(B_\Delta\), the set
\(\widehat B_\Delta\) defined in \eqref{eq5.57} is
positively invariant and pullback absorbing.
Finally, set
\begin{align*}
 R_*(\omega)
 :=
 2+
 2\left(
 \sup_{0\leq s\leq r}
 C_{r}^{*}(\theta_{-s}\omega)
 \right)
 \left(
 1+\sup_{0\leq s\leq r}
 R_{\mathrm d,*}(\theta_{-s}\omega)
 \right).
\end{align*}
In view of Lemma~\ref{lem4.3}, \(R_*\) is measurable and
tempered, and \(s\mapsto R_*(\theta_s\omega)\) is
continuous.
For \(0\leq\ell<n_r\) and
\(\zeta\in B_\Delta(\theta_{-\ell\Delta}\omega)\), by \eqref{eq5.58} we have
\begin{align*}
 &\norm{
 \Phi_\Delta
 \bigl(\ell,\theta_{-\ell\Delta}\omega\bigr)\zeta
 }_{\Hh_\Delta}^2
 \leq
 C_{r}^{*}(\theta_{-\ell\Delta}\omega)
 \left(1+\norm{\zeta}_{\Hh_\Delta}^2
 \right)
\leq
 C_{r}^{*}(\theta_{-\ell\Delta}\omega)
 \left(
 1+R_{\mathrm d,*}(\theta_{-\ell\Delta}\omega)
 \right)
 \leq R_*(\omega).
\end{align*}
Passing to the closure in \eqref{eq5.57} and then taking the
supremum over all admissible step sizes
\(\Delta=\tau/m\in(0,\Delta_0]\), we obtain
\eqref{eq:uniform-discrete-radius}.
\end{proof}

The following lemma provides a discrete sewing estimate for the
accumulated high-mode rough terms on a small grid block, uniformly
in \(K\) and \(\Delta\).

\begin{lemma}\label{lem5.11}
Let \(u^\Delta\) and the grid block \([k_0,k_1]\) satisfy the
hypotheses of Lemma~\ref{lem5.5}. For
\(k_0\leq i\leq j\leq k_1\), set
\[
 \Xi^K_{i,j}
 :=
 2I_K(u_i^\Delta)x_{t_i,t_j}
 +2I_K'(u_i^\Delta)\mathbb X_{t_i,t_j},
\]
where \(I_K\) and \(I_K'\) are given in
\eqref{eq4.36-1} and \eqref{eq4.36}, respectively.
Then there exists a constant \(C_M>0\), independent of
\(K\) and \(\Delta\), such that
\begin{align}
 &\max_{k_0\leq\ell\leq k_1}
 \left|
 \sum_{j=k_0}^{\ell-1}
 \left(
 2I_K(u_j^\Delta)x_j
 +2I_K'(u_j^\Delta)\mathbb X_j
 \right)
 -\Xi^K_{k_0,\ell}
 \right|
 \notag\\
 &\quad\leq
 C_M\Bigl[
 \omega_{\mathbf x}(t_{k_0},t_{k_1})^{2/p}
 \vartheta_{\lfloor K/2\rfloor}^{(1)}(t_{k_0},t_{k_1})^{1/p}
 +
 \omega_{\mathbf x}(t_{k_0},t_{k_1})^{1/p}
 \vartheta_{\lfloor K/2\rfloor}^{(2)}(t_{k_0},t_{k_1})^{1/q}
 \Bigr],
 \label{eq:discrete-sewing-tail}
\end{align}
where
$\vartheta_K^{(1)}(s,t)
 :=
 \norm{Q_Kx}_{\pvar,[s,t]}^p$
and
$\vartheta_K^{(2)}(s,t)
 :=
 \norm{(Q_K\otimes I)\mathbb X}_{\qvar,[s,t]}^q
 +
 \norm{(I\otimes Q_K)\mathbb X}_{\qvar,[s,t]}^q$.
\end{lemma}

\begin{proof}
For $k_0\leq i<j<k \leq k_1$, by Chen's relation \eqref{eq2.1} and the controlled-path
expansion of \(I_K(u^\Delta)\), we have
\begin{align}\label{eq5.60}
 \delta\Xi^K_{i,j,k}
 =&-2R^{I_K,\Delta}_{i,j}x_{t_j,t_k}
 +2\bigl(I_K'(u_i^\Delta)-I_K'(u_j^\Delta)\bigr)
 \mathbb X_{t_j,t_k}.
\end{align}
Arguing as in \eqref{eq4.39}-\eqref{eq4.40-0} and using \eqref{eq5.29}, we obtain that
\begin{align*}
 \norm{R^{I_K,\Delta}_{i,j}}
 &\leq
 C(1+M)
 \left(
 \norm{R_{i,j}^{u,\Delta}}
 +\norm{u_{i,j}^\Delta}^{2}
 \right)
 \leq
 C_M\omega_{\mathbf x}(t_i,t_j)^{2/p},\\
 \norm{I_K'(u_j^\Delta)-I_K'(u_i^\Delta)}
 &\leq
 C(1+M)\norm{u_{i,j}^\Delta}
 \leq
 C_M\omega_{\mathbf x}(t_i,t_j)^{1/p}.
\end{align*}
By the definitions of \(I_K\) and \(I_K'\), the increments
\(x_{t_j,t_k}\) and \(\mathbb X_{t_j,t_k}\) in \eqref{eq5.60} may be replaced by
\(Q_{\lfloor K/2\rfloor}x_{t_j,t_k}\) and
\((Q_{\lfloor K/2\rfloor}\otimes I)\mathbb X_{t_j,t_k}\),
respectively. Consequently,
\begin{align*}
 |\delta\Xi^K_{i,j,k}|
 \leq
 C_{M}
 \omega_{\mathbf x}(t_i,t_j)^{2/p}
 \vartheta_{\lfloor K/2\rfloor}^{(1)}(t_j,t_k)^{1/p}
 +
 C_{M}
 \omega_{\mathbf x}(t_i,t_j)^{1/p}
 \vartheta_{\lfloor K/2\rfloor}^{(2)}(t_j,t_k)^{1/q}.
\end{align*}
Since $q=p/2$, both product terms have total sewing exponent
$3/p>1$. Applying the product-form discrete sewing estimate in
Corollary~\ref{cor5.5} separately to the two terms on
$[k_0,\ell]$, using the monotonicity of the controls, and then
taking the maximum over $k_0\leq\ell\leq k_1$ yields
\eqref{eq:discrete-sewing-tail}.
\end{proof}

We next introduce the discrete counterpart of the tail
Lyapunov--Krasovskii functional in \eqref{eq4.34-3}.
For $K\geq1$ and $k\geq0$, set
\begin{equation}
\label{eq:discrete-tail-functional}
 \Psi_{K,k}^\Delta
 :=
 V_K(u_k^\Delta)
 +(L_1+C_*\Delta)e^{\sigma(\tau+\Delta)}
 \Delta
 \sum_{j=k-m}^{k-1}
 e^{-\sigma(t_k-t_j)}V_K(u_j^\Delta),
\end{equation}
where $V_K$ is defined in \eqref{eq4.31}, and $C_*$
is the constant appearing in \eqref{eq:G-square}.

\begin{lemma}\label{lem5.12}
Let \(u^\Delta\) and the grid block \([k_0,k_1]\) satisfy the
hypotheses of Lemma~\ref{lem5.5}. Then
\begin{align*}
 \Psi_{K,k_1}^\Delta
 \leq{}&
 e^{-c(t_{k_1}-t_{k_0})}
 \Psi_{K,k_0}^\Delta
 +
 C\sum_{j=k_0}^{k_1-1}
 e^{-c(t_{k_1}-t_{j+1})}
 \Delta\left(
 K^{-1}\norm{u_j^\Delta}^2
 +\norm{Q_KF_0}^2
 \right)\\
 &+
 C\left(
 \vartheta_{\lfloor K/2\rfloor}^{(1)}
 (t_{k_0},t_{k_1})^{1/p}
 +
 \vartheta_{\lfloor K/2\rfloor}^{(2)}
 (t_{k_0},t_{k_1})^{1/q}
 \right),
\end{align*}
where $0<c\leq\min
 \{c_0,\sigma/2\}$, and \(C>0\) is independent of \(K\) and \(\Delta\).
\end{lemma}

\begin{proof}
Set
\[
 b_j:=-\mathcal Lu_j^\Delta-\lambda u_j^\Delta
 +F(u_j^\Delta,u_{j-m}^\Delta)~~
 \mbox{and}~~
 r_j^\Delta
 :=h(u_j^\Delta)x_j+H(u_j^\Delta)\mathbb X_j.
\]
Then the Euler recursion \eqref{eq:euler} can be written as
$u_{j+1}^\Delta=u_j^\Delta+\Delta b_j+r_j^\Delta$.
Since \(C_K\) is self-adjoint and
\(V_K(u)=\ip{u}{C_Ku}\), it follows that
\begin{equation}\label{eq5.61}
\begin{aligned}
 V_K(u_{j+1}^\Delta)-V_K(u_j^\Delta)
 ={}&
 2\Delta\ip{b_j}{C_Ku_j^\Delta}
 +\Delta^2\ip{b_j}{C_Kb_j}\\
 &+
 2\ip{r_j^\Delta}{C_Ku_j^\Delta}
 +2\Delta\ip{b_j}{C_Kr_j^\Delta}
 +\ip{r_j^\Delta}{C_Kr_j^\Delta}.
\end{aligned}
\end{equation}
An argument analogous to that used in the derivation of
\eqref{eq4.34} yields
\begin{align}\label{eq5.62}
 2\Delta\ip{b_j}{C_Ku_j^\Delta}
 &\leq-(\lambda-L_{0})\Delta V_K(u_j^\Delta)
 +L_{1}\Delta V_K(u_{j-m}^\Delta)
 \nonumber\\
 &+\frac{\nu C_\chi\Delta\|u_j^\Delta\|^2 }{K}
 \bigg(\sup_{i}\sum_{\ell}
 w_{i\ell}|i-\ell|\bigg)
 +\frac{\Delta\norm{Q_KF_0}^2}
 {\lambda-L_{0}-L_{1}}.
\end{align}
Using \(C_K=S_K^2\), \(S_K=S_KQ_K\),
the localized Lipschitz estimate for \(F\), and
Lemma~\ref{lem4.8}, we obtain
\begin{align}\label{eq5.63}
 &\Delta^2\ip{b_j}{C_Kb_j}
 =\Delta^2\norm{S_Kb_j}^2\nonumber\\
 &\leq4\Delta^2\left(\norm{S_K\mathcal Lu_j^\Delta}^{2}
 +\lambda^{2}\norm{S_Ku_j^\Delta}^{2}
 +\norm{S_K(F(u_j^\Delta,
 u_{j-m}^\Delta)-F_0)}^{2}
 +\norm{S_KF_0}^{2}\right)
 \nonumber\\
 &\leq C\Delta^2\left(
 V_K(u_j^\Delta)+V_K(u_{j-m}^\Delta)
 +K^{-2}\norm{u_j^\Delta}^2
 +\norm{Q_KF_0}^2\right).
\end{align}
Here and below, \(C>0\) is independent of \(K\) and
\(\Delta\).
Denote the second term on the right-hand side of
\eqref{eq:discrete-tail-functional} by $M_{K,j}^\Delta$. Then
$\Psi_{K,j}^\Delta=
 V_K(u_j^\Delta)+M_{K,j}^\Delta$.
Arguing as in \eqref{eq5.41}, we see that
\begin{align}\label{eq5.64}
 M_{K,j+1}^\Delta-M_{K,j}^\Delta
=&-(1-e^{-\sigma\Delta})M_{K,j}^\Delta
 +(L_1+C_*\Delta)e^{\sigma\tau}
 \Delta V_K(u_j^\Delta)
 \nonumber\\
 &-(L_1+C_*\Delta)\Delta V_K(u_{j-m}^\Delta).
\end{align}
Substituting \eqref{eq5.62} and \eqref{eq5.63} into
\eqref{eq5.61} and then adding \eqref{eq5.64}, we find
that the terms involving \(V_K(u_{j-m}^\Delta)\) cancel.
Using \eqref{eq:discrete-coercivity-constants}, \(\Delta\leq1\), \(K\geq1\), and enlarging $C_*$ once if necessary, we arrive at
\begin{align}\label{eq5.65}
 \Psi_{K,j+1}^\Delta-\Psi_{K,j}^\Delta
 \leq{}&
 -\Delta\Bigl[
 \lambda-L_0-C_*\Delta
 -(L_1+C_*\Delta)e^{\sigma\tau}
 \Bigr]V_K(u_j^\Delta)
 \nonumber\\
 &-
 (1-e^{-\sigma\Delta})M_{K,j}^\Delta
 +
 C\Delta\left(
 K^{-1}\norm{u_j^\Delta}^{2}
 +\norm{Q_KF_0}^{2}
 \right)
 +
 \mathcal R_{K,j}^\Delta
 \nonumber\\
 \leq{}&
 -c\Delta\Psi_{K,j}^\Delta
 +
 C\Delta\left(
 K^{-1}\norm{u_j^\Delta}^{2}
 +\norm{Q_KF_0}^{2}
 \right)
 +
 \mathcal R_{K,j}^\Delta,
\end{align}
where $\mathcal R_{K,j}^\Delta:=
2\langle r_j^\Delta, C_Ku_j^\Delta\rangle
 +2\Delta\langle b_j, C_Kr_j^\Delta\rangle
 +\langle r_j^\Delta, C_Kr_j^\Delta\rangle$ and $0<c\leq\min
 \{c_0,\sigma/2\}$.

By the definition of \(r_j^\Delta\), the
self-adjointness of \(C_K\), the symmetry of \(J_K\),
the weakly geometric property of \(\mathbf x\), and
\eqref{eq4.36-1}-\eqref{eq4.36}, we obtain that
\begin{align} \label{eq5.66}
 \mathcal R_{K,j}^\Delta
 ={}&
 2\ip{h(u_j^\Delta)x_j}{C_Ku_j^\Delta}
 +2\ip{H(u_j^\Delta)\mathbb X_j}{C_Ku_j^\Delta}
 +\ip{h(u_j^\Delta)x_j}
 {C_Kh(u_j^\Delta)x_j}
 \nonumber\\
 &+
 2\Delta\ip{b_j}{C_Kh(u_j^\Delta)x_j}
 +2\Delta\ip{b_j}
 {C_KH(u_j^\Delta)\mathbb X_j}
 \nonumber\\
 &+
 2\ip{h(u_j^\Delta)x_j}
 {C_KH(u_j^\Delta)\mathbb X_j}
 +\ip{H(u_j^\Delta)\mathbb X_j}
 {C_KH(u_j^\Delta)\mathbb X_j}
 \nonumber\\
 ={}&
 2I_K(u_j^\Delta)x_j
 +2I_K'(u_j^\Delta)\mathbb X_j
 +r_{K,j}^\Delta,
\end{align}
where
\begin{align}\label{eq5.67}
 r_{K,j}^\Delta
 ={}&
 2\Delta\ip{b_j}{C_Kh(u_j^\Delta)x_j}
 +\ip{2\Delta b_j+ 2 h(u_j^\Delta)x_j+ H(u_j^\Delta)\mathbb X_j}
 {C_KH(u_j^\Delta)\mathbb X_j}.
\end{align}

Notice that $C_Kh(u_j^\Delta)x_j=C_Kh(
u_j^\Delta)
Q_{\lfloor K/2\rfloor}x_j$ and
$C_KH(u_j^\Delta)\mathbb X_j
 =C_KH(u_j^\Delta)
 \bigl((Q_{\lfloor K/2\rfloor}\otimes I)\mathbb X_j\bigr)$.
Thanks to H\"older's inequality, it holds that
\begin{align}
 \sum_{j=k_0}^{k_1-1}\Delta
 \norm{Q_{\lfloor K/2\rfloor}x_j}
 &\leq (t_{k_1}-t_{k_0})^{1-1/p}
 \Delta^{1/p}\norm{Q_{\lfloor K/2\rfloor}x}_{
 \pvar,[t_{k_0},t_{k_1}]},\notag\\
 \sum_{j=k_0}^{k_1-1}
 \Delta\norm{(Q_{\lfloor K/2\rfloor}\otimes I)\mathbb X_j}
 &\leq C\Delta^{1/q}
 \norm{(Q_{\lfloor K/2\rfloor}\otimes I)\mathbb X}_{\qvar,[t_{k_0},t_{k_1}]},\notag\\
 \sum_{j=k_0}^{k_1-1}
 \norm{Q_{\lfloor K/2\rfloor}x_j}\norm{\mathbb X_j}
 &\leq
 \norm{Q_{\lfloor K/2\rfloor}x}_{\pvar,[t_{k_0},t_{k_1}]}
 \norm{\mathbb X}_{\qvar,[t_{k_0},t_{k_1}]},\notag\\
 \sum_{j=k_0}^{k_1-1}
 \norm{(Q_{\lfloor K/2\rfloor}\otimes I)\mathbb X_j}\norm{\mathbb X_j}
 &\leq
 \norm{(Q_{\lfloor K/2\rfloor}\otimes I)\mathbb X}_{\qvar,[t_{k_0},t_{k_1}]}
 \norm{\mathbb X}_{\qvar,[t_{k_0},t_{k_1}]}.
 \label{eq5.68}
\end{align}
Hence, by \eqref{eq5.67}-\eqref{eq5.68} and $\max_{k_0-m\leq j\leq k_1}\|u_j^\Delta\|\leq M$, we have
\begin{equation}\label{eq5.69}
 \max_{k_0\leq\ell\leq k_1}
 \left|\sum_{j=k_0}^{\ell-1}r_{K,j}^\Delta\right|
 \leq \sum_{j=k_0}^{k_1-1}|r_{K,j}^\Delta|
 \leq
 C_{M}
 \bigg(\vartheta_{\lfloor K/2\rfloor}^{(1)}
 (t_{k_0},t_{k_1})^{1/p}
 +
 \vartheta_{\lfloor K/2\rfloor}^{(2)}
 (t_{k_0},t_{k_1})^{1/q}\bigg).
\end{equation}
Applying Lemma~\ref{lem5.11} to the first two terms on the
right-hand side of \eqref{eq5.66} and invoking
\eqref{eq5.69}, we deduce that
\begin{align}\label{eq5.70}
 &\max_{k_0\leq\ell\leq k_1}
 \left|\sum_{j=k_0}^{\ell-1}\mathcal R_{K,j}^\Delta\right|
\leq
\max_{k_0\leq\ell\leq k_1}\left|
\sum_{j=k_0}^{\ell-1}
\left(
2I_K(u_j^\Delta)x_j
+2I_K'(u_j^\Delta)\mathbb X_j
\right)
-\Xi_{k_0,\ell}^K
\right|
\nonumber\\
&+
\max_{k_0\leq\ell\leq k_1}\left|\Xi_{k_0,\ell}^K\right|
+
\max_{k_0\leq\ell\leq k_1}\left|
\sum_{j=k_0}^{\ell-1}r_{K,j}^\Delta
\right|\leq C_{M} \bigg(\vartheta_{\lfloor K/2\rfloor}^{(1)}(t_{k_0},t_{k_1})^{1/p}
 +
 \vartheta_{\lfloor K/2\rfloor}^{(2)}
 (t_{k_0},t_{k_1})^{1/q}\bigg).
\end{align}
where we have used
$|\Xi_{k_0,\ell}^K|
\leq
C (
\|Q_{\lfloor K/2\rfloor}x_{t_{k_0},t_\ell}\|
+\|(Q_{\lfloor K/2\rfloor}\otimes I)
\mathbb X_{t_{k_0},t_\ell}\|
+\|(I\otimes Q_{\lfloor K/2\rfloor})
\mathbb X_{t_{k_0},t_\ell}\|)$
and $\omega_{\mathbf{x}}(t_{k_0},
t_{k_1})\leq\gamma_M$.
Iterating \eqref{eq5.65} yields
\begin{equation}\label{eq5.71}
\begin{aligned}
 \Psi_{K,k_1}^\Delta
 \leq{}&
 (1-c\Delta)^{k_1-k_0}
 \Psi_{K,k_0}^\Delta\\
 &+
 \sum_{j=k_0}^{k_1-1}
 (1-c\Delta)^{k_1-j-1}
 \left[
 C\Delta\left(
 K^{-1}\norm{u_j^\Delta}^2
 +\norm{Q_KF_0}^2
 \right)
 +\mathcal R_{K,j}^\Delta
 \right].
\end{aligned}
\end{equation}
For \(k_0\leq j\leq k_1-1\), define
$\varpi_j^\Delta
 :=
 (1-c\Delta)^{k_1-j-1}$ and
$S_\ell
 :=
 \sum_{j=k_0}^{\ell-1}\mathcal R_{K,j}^\Delta$.
The discrete Abel summation formula gives
\begin{align}\label{eq5.72}
 \left|\sum_{j=k_0}^{k_1-1}
 \varpi_j^\Delta\mathcal R_{K,j}^\Delta\right|
 \leq\left|\varpi_{k_1-1}^\Delta S_{k_1}\right|
 +\left|\sum_{j=k_0}^{k_1-2}
 (\varpi_j^\Delta-\varpi_{j+1}^\Delta)S_{j+1}\right|
 \leq2\max_{k_0\leq\ell\leq k_1}|S_\ell|.
\end{align}
Indeed, since \(0\leq1-c\Delta\leq1\), the sequence
\(\{\varpi_j^\Delta\}_{j=k_0}^{k_1-1}\) is nondecreasing
in \(j\),
$\varpi_{k_1-1}^\Delta=1$ and
$\sum_{j=k_0}^{k_1-2}|
 \varpi_j^\Delta-\varpi_{j+1}^\Delta|
 =1-\varpi_{k_0}^\Delta
 \leq1$.
Finally, substituting \eqref{eq5.70} into
\eqref{eq5.72} and then using the resulting estimate in
\eqref{eq5.71}, together with
$(1-c\Delta)^n\leq e^{-cn\Delta}$ for $n\geq0$, we obtain the asserted estimate.
\end{proof}

We next establish a discrete counterpart of
Theorem~\ref{thm4.15} for the absorbing families
$\widehat B_\Delta$, uniformly over all sufficiently small
admissible step sizes.

\begin{theorem}\label{thm5.13}
For every \(\varepsilon>0\) and
\(\omega\in\Omega_0\), there exist
\(\widetilde K=\widetilde K(\varepsilon,\omega)\in\mathbb N\), \(\widetilde T=\widetilde T(\varepsilon,\omega)>0\), and \(\Delta_0=\Delta_0(\omega)>0\)
such that, for every admissible step size $\Delta=\frac{\tau}{m}\leq\Delta_0$, with
$m\in\mathbb N$,
and every \(k\in\mathbb N_0\) satisfying
\(k\Delta\geq \widetilde T\), we have
\begin{align*}
 \sup_{\zeta\in
 \widehat B_\Delta(\theta_{-k\Delta}\omega)}
 \max_{-m\leq j\leq0}
 \norm{
 Q_{2\widetilde K}
 \bigl(
 \Phi_\Delta(k,\theta_{-k\Delta}\omega)\zeta
 \bigr)_j
 }^2
 <\varepsilon.
\end{align*}
\end{theorem}

\begin{proof}
Fix $\varepsilon>0$ and $\omega\in\Omega_0$.
Choose $\Delta_0>0$, independently of $\varepsilon$,
so that \eqref{eq:discrete-coercivity-constants}
and \eqref{eq:uniform-discrete-radius} hold
for every admissible $\Delta\leq\Delta_0$.
We may assume that $\Delta_0\leq1$.

\textbf{Step~1.}
Let $L=M\tau>\tau+1$, with $M\in\mathbb N$, be
sufficiently large, to be specified below. For $k\Delta\geq L$ and
$\zeta\in\widehat B_\Delta(\theta_{-k\Delta}\omega)$,
we have
$\mathbf u_j^\Delta
:=
(u_{j-m}^\Delta,\ldots,u_j^\Delta)
=
\Phi_\Delta(j+k,
\theta_{-k\Delta}\omega)\zeta$ for $-k\leq j\leq0$.
By the positive invariance of $\widehat B_\Delta$, it holds that
$\mathbf u_j^\Delta
\in
\widehat B_\Delta(\theta_{t_j}\omega)$ for
$-k\leq j\leq0$.
Hence, by \eqref{eq:uniform-discrete-radius},
$\|\mathbf u_j^\Delta\|_{\Hh_\Delta}^2
\leq
R_*(\theta_{t_j}\omega)$.
Set
\[
 M_L(\omega)
 :=
 1+\sup_{-L\leq s\leq0}
 R_*(\theta_s\omega)^{1/2}.
\]
The continuity of
$s\mapsto R_*(\theta_s\omega)$ implies that
$M_L(\omega)<\infty$, and $M_L(\omega)$ is independent of
$\Delta$ and $\zeta$.
Let $\gamma_{M_L(\omega)}$ be the threshold in
Lemma~\ref{lem5.5}. Since $L=M\tau$ and
$\Delta=\tau/m$, the point $-L$ belongs to every
admissible grid. By Lemma~\ref{lem5.3}, there exists
$\Delta_L(\omega)>0$ such that, for every admissible
$\Delta\leq\min\{\Delta_L(\omega),
\Delta_0\}$, there are grid indices
$-L/\Delta:=k_0<k_1<\cdots<k_N=0$,
whose blocks satisfy
$\omega_{\mathbf x}
 (t_{k_{\ell-1}},t_{k_\ell})
 \leq\gamma_{M_L(\omega)}$.
Refining the partition if necessary, we may also
assume that $|I_\ell|\leq\tau\wedge1$, where
$I_\ell=[t_{k_{\ell-1}},t_{k_\ell}]$.
The refined number of blocks satisfies $N\leq N_L(\omega)$ for some finite bound
$N_L(\omega)$ independent of $\Delta$, $k$,
and $\zeta$.

Fix $i\in\{-m,\ldots,0\}$.
Restrict the preceding partition to $[-L,t_i]$
and include $i$ as its final index. Denote the
resulting partition by
\[
 -L/\Delta=r_0<r_1<\cdots<r_J=i,
 \qquad J\leq N.
\]
For every block
$I_\ell=[t_{r_{\ell-1}},t_{r_\ell}]$, $\ell=1,\ldots,J$, it holds that
\begin{align*}
\max_{r_{\ell-1}-m\leq j\leq r_\ell}
\|u_j^\Delta\|
&\leq
\max_{r_{\ell-1}\leq j\leq r_\ell}
\|\mathbf u_j^\Delta\|_{\Hh_\Delta}
\leq
M_L(\omega).
\end{align*}
Iterating the estimate of Lemma~\ref{lem5.12} yields
\begin{align*}
 \Psi_{K,i}^\Delta
 \leq{}&
 e^{-c(t_i+L)}\Psi_{K,-L/\Delta}^\Delta
 +
 C\sum_{j=-L/\Delta}^{i-1}
 e^{-c(t_i-t_{j+1})}\Delta
 \left(
 K^{-1}\norm{u_j^\Delta}^2
 +\norm{Q_KF_0}^2
 \right)
 \\
 &+
 C_{M_L(\omega)}
 \sum_{\ell=1}^{J}
 e^{-c(t_i-t_{r_\ell})}
 \left[
 \vartheta_{\lfloor K/2\rfloor}^{(1)}
 (t_{r_{\ell-1}},t_{r_\ell})^{1/p}
 +
 \vartheta_{\lfloor K/2\rfloor}^{(2)}
 (t_{r_{\ell-1}},t_{r_\ell})^{1/q}
 \right]
 \\
 \leq{}&
Ce^{-c(L-\tau)}R_*(\theta_{-L}\omega)
 +
 CL\left(
 K^{-1}M_L(\omega)^2+\norm{Q_KF_0}^2
 \right)
 \notag\\
 &+
 C_{M_L(\omega)}
 \left[
 N_L(\omega)^{1-1/p}
 \vartheta_{\lfloor K/2\rfloor}^{(1)}(-L,0)^{1/p}
 +
 N_L(\omega)^{1-1/q}
 \vartheta_{\lfloor K/2\rfloor}^{(2)}(-L,0)^{1/q}
 \right],
\end{align*}
where we have used $\Psi_{K,-L/\Delta}^\Delta\leq
C\|\mathbf u_{-L/\Delta}^\Delta\|_{\Hh_\Delta}^2$, H\"older's inequality and the
superadditivity of the tail controls $\vartheta_{\lfloor K/2\rfloor}^{(1)}$ and $\vartheta_{\lfloor K/2\rfloor}^{(2)}$.

Taking the maximum over $-m\leq i\leq0$, we find that
\begin{align}
 \max_{-m\leq i\leq0}\Psi_{K,i}^\Delta
 \leq{}&
 Ce^{-c(L-\tau)}R_*(\theta_{-L}\omega)
 +
 C_L(\omega)
 \left(K^{-1}+\norm{Q_KF_0}^2\right)
 \notag\\
 &+
 C_L(\omega)
 \left[
 \vartheta_{\lfloor K/2\rfloor}^{(1)}(-L,0)^{1/p}
 +
 \vartheta_{\lfloor K/2\rfloor}^{(2)}(-L,0)^{1/q}
 \right],
 \label{eq:finite-window-tail}
\end{align}
where $C_L(\omega)$ is finite and independent of
$K$, $\Delta$, $k$, and $\zeta$.
 Since $R_*$ is tempered, we may now fix $L=M\tau$
sufficiently large that
$Ce^{-c(L-\tau)}R_*(\theta_{-L}\omega)
 <\frac{\varepsilon}{2}$.
For this fixed $L$, in view of  Assumption \textbf{(A5)}, we have
\[
 \vartheta_{\lfloor K/2\rfloor}^{(1)}(-L,0)^{1/p}
 +
 \vartheta_{\lfloor K/2\rfloor}^{(2)}(-L,0)^{1/q}
 \longrightarrow0
 \qquad\text{as }K\to\infty.
 \]
Together with $K^{-1}\to0$ and
$\norm{Q_KF_0}\to0$, this allows us to choose
$K_0\in\mathbb N$ sufficiently large that the
last two terms in \eqref{eq:finite-window-tail}
have sum less than $\varepsilon/2$.

Set $T_0:=L$. For every admissible
$\Delta\leq\min\{\Delta_L(\omega),
\Delta_0\}$, every $k$ satisfying
$k\Delta\geq T_0$, and every
$\zeta\in\widehat B_\Delta(\theta_{-k\Delta}\omega)$,
it follows from \eqref{eq4.32} and
\eqref{eq:discrete-tail-functional} that
\begin{align*}
 &\max_{-m\leq j\leq0}
 \norm{
 Q_{2K_0}
 \bigl(
 \Phi_\Delta(k,\theta_{-k\Delta}\omega)\zeta
 \bigr)_j
 }^2=
 \max_{-m\leq j\leq0}\norm{Q_{2K_0}u_j^\Delta}^2
 \leq
 \max_{-m\leq j\leq0}V_{K_0}(u_j^\Delta)
 \leq
 \max_{-m\leq j\leq0}\Psi_{K_0,j}^\Delta.
\end{align*}
The right-hand side has a common upper bound strictly
less than $\varepsilon$, independent of $\zeta$.
Thus, taking the supremum over $\zeta$ gives the
required tail estimate whenever
$\Delta\leq\min\{\Delta_L(\omega),
\Delta_0\}$
and $k\Delta\geq T_0$.

\textbf{Step~2.}
We establish a tail estimate for each fixed admissible
$\Delta\leq\Delta_0$, without requiring small rough-path
control on individual grid steps.
Write
\[
 b_j:=-\mathcal Lu_j^\Delta-\lambda u_j^\Delta
      +F(u_j^\Delta,u_{j-m}^\Delta),
 \qquad
 r_j^\Delta:=h(u_j^\Delta)x_j
      +H(u_j^\Delta)\mathbb X_j,
\]
and set $v_{j+1}^\Delta:=u_j^\Delta+\Delta b_j$.
Then $u_{j+1}^\Delta=v_{j+1}^\Delta+r_j^\Delta$.
Denote the second term on the right-hand side of  \eqref{eq:discrete-tail-functional} by
$M_{K,j}^\Delta$. By the similar arguments as in \eqref{eq5.65}, we have
\begin{align}
 V_K(v_{j+1}^\Delta)+M_{K,j+1}^\Delta
 \leq
 (1-c\Delta)\Psi_{K,j}^\Delta+
 C\Delta
 \left(
 K^{-1}\norm{u_j^\Delta}^2+\norm{Q_KF_0}^2
 \right),
 \label{eq:deterministic-tail-step}
\end{align}
where $c\leq\min\{c_0,\sigma/2\}$.
Notice that $1-c\Delta\leq1$ and
\[
 (1+c\Delta/2)(1-c\Delta)
 =
 1-c\Delta/2-c^2\Delta^2/2
 \leq1-c\Delta/2,
\]
Using the identity
$V_K(u)=\norm{S_Ku}^2$, Young's inequality and \eqref{eq:deterministic-tail-step}, we deduce that
\begin{align}
\label{eq:perturbed-tail-step}
 \Psi_{K,j+1}^\Delta
 &=
 V_K(v_{j+1}^\Delta+r_j^\Delta)
 +M_{K,j+1}^\Delta\nonumber\\
 &=
 V_K(v_{j+1}^\Delta)+2\langle S_Kv_{j+1}^\Delta,S_Kr_j^\Delta\rangle
 +\norm{S_Kr_j^\Delta}^2+M_{K,
 j+1}^\Delta\nonumber\\
 &\leq
 (1+c\Delta/2)
 \left[
 V_K(v_{j+1}^\Delta)+M_{K,j+1}^\Delta
 \right]+
 (1+2/(c\Delta))\norm{S_Kr_j^\Delta}^2
 \nonumber\\
&\leq
 (1-c\Delta/2)\Psi_{K,j}^\Delta
 +C\Delta
 \left(
 K^{-1}\norm{u_j^\Delta}^2+\norm{Q_KF_0}^2
 \right)\nonumber\\
 &~~~~+\frac C\Delta
 \left[
 \norm{Q_{\lfloor K/2\rfloor}x_j}^2
 +
 \norm{(Q_{\lfloor K/2\rfloor}\otimes I)\mathbb X_j}^2\right].
\end{align}
Fix this $\Delta$, let $k\Delta\geq\tau$. In view of $\Psi_{K,-k}^\Delta
 \leq C R_*(\theta_{-k\Delta}\omega)$,
iterating \eqref{eq:perturbed-tail-step} from $-k$
to each $i\in\{-m,\ldots,0\}$, using
$(1-c\Delta/2)^n\leq e^{-cn\Delta/2}$, and setting
$j=-\nu-1$ in the sums, we obtain
\begin{align}
 &\max_{-m\leq i\leq0}\Psi_{K,i}^\Delta
 \leq
 Ce^{-c(k\Delta-\tau)/2}
 R_*(\theta_{-k\Delta}\omega)
 +\frac{C\Delta e^{c\tau/2}}{K}
 \sum_{\nu=0}^{\infty}e^{-c\nu\Delta/2}
 R_*(\theta_{-(\nu+1)\Delta}\omega)
 +C\norm{Q_KF_0}^2
 \notag\\
 &~~+\frac{Ce^{c\tau/2}}{\Delta}
\sum_{\nu=0}^{\infty}e^{-c\nu\Delta/2}
 \left[
 \norm{Q_{\lfloor K/2\rfloor}
 x_{-(\nu+1)\Delta,-\nu\Delta}}^2
 \right.
 \left.
 +
 \norm{(Q_{\lfloor K/2\rfloor}\otimes I)
 \mathbb X_{-(\nu+1)\Delta,-\nu\Delta}}^2
 \right],
 \label{eq:fixed-step-tail-series}
\end{align}
where we used $-i\Delta\leq\tau$ and
$\Delta\sum_{\nu=0}^{\infty}
e^{-c\nu\Delta/2}
 =\frac{\Delta}{1-e^{-c\Delta/2}}
 \leq\frac{e^{c\Delta_0/2}}{c/2}$.

For this fixed $\Delta$, by the temperedness of $R_*$
and the rough-path variation norms on bounded
intervals, we find that both series in
\eqref{eq:fixed-step-tail-series} converge.
Assumption \textbf{(A5)} and dominated convergence
therefore imply that the final series in
\eqref{eq:fixed-step-tail-series} tends to zero
as $K\to\infty$.
Together with $K^{-1}\to0$ and
$\norm{Q_KF_0}\to0$, this allows us to choose
$K_\Delta\in\mathbb N$ so that the last three
terms in \eqref{eq:fixed-step-tail-series}
have sum less than $\varepsilon/2$.
Since $R_*$ is tempered, we may then choose
$T_\Delta\geq\tau$ so that the first term is less
than $\varepsilon/2$ whenever $k\Delta\geq T_\Delta$.
By \eqref{eq4.32} and \eqref{eq:discrete-tail-functional}, it follows that
\begin{align*}
 \sup_{\zeta\in
 \widehat B_\Delta(\theta_{-k\Delta}\omega)}
 \max_{-m\leq j\leq0}
 \left\|
 Q_{2K_\Delta}
 \bigl(
 \Phi_\Delta(k,\theta_{-k\Delta}\omega)\zeta
 \bigr)_j
 \right\|^2
 <\varepsilon
\end{align*}
whenever $k\Delta\geq T_\Delta$.

\textbf{Step~3.}
The remaining admissible step sizes form the finite set
\[
 \mathscr S
 :=
 \left\{\Delta=\tau/m:
 \Delta_L(\omega)<\Delta\leq\Delta_0,\
 m\in\mathbb N\right\}.
\]
Indeed, their corresponding integers satisfy
$\tau/\Delta_0\leq m<\tau/\Delta_L(\omega)$.
Set
\[
 \widetilde K
 :=
 \max\bigl(
 \{K_0\}\cup\{K_\Delta:\Delta\in\mathscr S\}
 \bigr),
 \qquad
 \widetilde T
 :=
 \max\bigl(
 \{T_0\}\cup\{T_\Delta:\Delta\in\mathscr S\}
 \bigr).
\]
Combining the tail estimates from Steps~1 and~2 and using
$\norm{Q_{2\widetilde K}u}\leq\norm{Q_{2K}u}$
for $\widetilde K\geq K$, we obtain the required estimate
for every admissible $\Delta\leq\Delta_0$ and
$k\Delta\geq\widetilde T$.

\end{proof}

\begin{theorem}\label{thm5.14}
For every sufficiently small admissible \(\Delta\), the
cocycle \(\Phi_\Delta\) possesses a unique pullback attractor
\(\mathcal A_\Delta\) with respect to the physical-time
tempered universe defined by
\eqref{eq:discrete-tempered}. Moreover, for every
\(\omega\in\Omega_0\),
$\mathcal A_\Delta(\omega)
 \subseteq
 \widehat B_\Delta(\omega)$.
\end{theorem}
\begin{proof}
Fix a sufficiently small admissible \(\Delta\), and let
\(D_\Delta\) be any tempered random bounded family.
Thanks to Proposition~\ref{prop5.11} and Theorem~\ref{thm5.13},
we find that the sets
$\Phi_\Delta(k,\theta_{-k\Delta}\omega)
 D_\Delta(\theta_{-k\Delta}\omega)$
are eventually bounded and have uniformly vanishing
lattice tails at every history node. Since \(m=\tau/\Delta<\infty\), the history space
\(\Hh_\Delta=(\ell^2)^{m+1}\) is a finite product of copies
of \(\ell^2\). Applying the compactness criterion in
\(\ell^2\) to each history component and then using a finite
diagonal argument, we show that \(\Phi_\Delta\) is
pullback asymptotically compact in \(\Hh_\Delta\).
Define the pullback \(\omega\)-limit set of
\(\widehat B_\Delta\) by
\[
\mathcal A_\Delta(\omega)
:=
\bigcap_{N\in\mathbb N_0}
\overline{
\bigcup_{k\geq N}
\Phi_\Delta(k,\theta_{-k\Delta}\omega)
\widehat B_\Delta(\theta_{-k\Delta}\omega)
}^{\,\Hh_\Delta}.
\]
By the continuity of \(\Phi_\Delta\), the pullback absorbing
property of \(\widehat B_\Delta\), and the pullback
asymptotic compactness established above, we conclude that \(\mathcal A_\Delta\) is a compact, strictly
invariant family that pullback attracts every tempered
random bounded family. Hence, it is the unique pullback
attractor of \(\Phi_\Delta\).
Moreover, since \(\widehat B_\Delta\) is closed and
positively invariant, we obtain that
$\mathcal A_\Delta(\omega)
\subseteq
\widehat B_\Delta(\omega)$.
\end{proof}

\begin{lemma}\label{lem5.15}
Let \(\Delta_0>0\) be sufficiently small so that, for
every admissible step size
\(\Delta=\tau/m\in(0,\Delta_0]\), the pullback attractor
\(\mathcal A_\Delta\) given by
Theorem~\ref{thm5.14} exists. Then, for every
\(\omega\in\Omega_0\), the set
$\bigcup_{0<\Delta\leq\Delta_0}
 J_\Delta\mathcal A_\Delta(\omega)$
is relatively compact in $\Hh$.
\end{lemma}

\begin{proof}
By the strict invariance of \(\mathcal A_\Delta\) under
\(\Phi_\Delta\), for every
\(\zeta_\Delta\in\mathcal A_\Delta(\omega)\) and
integer \(k\geq1\), there exists
$\widetilde{\zeta}_{\Delta,k}
 \in
 \mathcal A_\Delta(\theta_{-k\Delta}\omega)$
such that
$\zeta_\Delta
 =
 \Phi_\Delta(k,\theta_{-k\Delta}\omega)
 \widetilde{\zeta}_{\Delta,k}$.
Since
\(\mathcal A_\Delta(\omega)\subseteq
\widehat B_\Delta(\omega)\), by
\eqref{eq:uniform-discrete-radius} we have
\begin{equation}\label{eq:uniform-attractor-bound}
 \sup_{0<\Delta\leq\Delta_0}
 \sup_{\zeta_\Delta\in\mathcal A_\Delta(\omega)}
 \norm{\zeta_\Delta}_{\Hh_\Delta}^2
 \leq R_*(\omega).
 \end{equation}
Taking \(k\Delta\) sufficiently large and applying
Theorem~\ref{thm5.13}, we obtain that, for every
\(\varepsilon>0\) and \(\omega\in\Omega_0\), there exist
\(K=K(\varepsilon,\omega)\in\mathbb N\) and
\(\Delta_\varepsilon\in (0,\Delta_0]\) such that
\begin{equation}
\label{eq:uniform-attractor-tail}
 \sup_{0<\Delta\leq\Delta_\varepsilon}
 \sup_{\zeta_\Delta\in\mathcal A_\Delta(\omega)}
 \max_{-m\leq j\leq0}
 \norm{Q_{2K}(\zeta_\Delta)_j}^2
 <\varepsilon.
\end{equation}
For
\(-\tau\leq\vartheta_1<\vartheta_2\leq0\), thanks to the piecewise
linearity of \(J_\Delta\), \eqref{eq5.29}
and the superadditivity of \(\omega_{\mathbf x}\), we show that
\begin{equation}\label{eq:discrete-equicontinuity}
 \norm{
 \bigl(J_\Delta\zeta_\Delta\bigr)(\vartheta_2)
 -
 \bigl(J_\Delta\zeta_\Delta\bigr)(\vartheta_1)
 }
 \leq
 C(\omega)
 \omega_{\mathbf x}
 (\vartheta_1-\Delta,
 \vartheta_2+\Delta)^{1/p},
\end{equation}
where \(C(\omega)\) is independent of \(\Delta\) and
\(\zeta_\Delta\). Since \(\omega_{\mathbf x}\) is a
continuous control, the right-hand side of
\eqref{eq:discrete-equicontinuity} converges to zero
uniformly as
\(|\vartheta_2-\vartheta_1|+\Delta\to0\).

Consider any sequence of admissible steps
$\Delta_n=\tau/m_n$ and
$\zeta_{\Delta_n}\in\mathcal A_{\Delta_n}(\omega)$.
 If \((m_n)_{n\geq1}\) is bounded, then there exist
\(m_*\in\mathbb N\) and a subsequence \((m_{n_\ell})\)
such that \(m_{n_\ell}=m_*\), and hence
\(\Delta_{n_\ell}=\Delta_*:=\tau/m_*\).
Using the compactness of \(\mathcal A_{\Delta_*}(\omega)\) and
the continuity of \(J_{\Delta_*}\),
we obtain that
\(J_{\Delta_*}\mathcal A_{\Delta_*}(\omega)\) is compact in
\(\Hh\); hence
\((J_{\Delta_{n_\ell}}\zeta_{\Delta_{n_\ell}})\)
admits a convergent subsequence in \(\Hh\).

If \((m_n)_{n\geq1}\) is unbounded, then, after passing to a
subsequence and relabeling, we may assume that
\(\Delta_n=\tau/m_n\downarrow0\). By \eqref{eq:uniform-attractor-tail}, we find that, for every
\(\varepsilon>0\), there exists \(K=K(\varepsilon,\omega)\)
such that, for all sufficiently large \(n\),
\begin{equation}
\label{eq:discrete-equicontinuity-1}
\norm{Q_{2K}J_{\Delta_n}\zeta_{\Delta_n}}_{\Hh}
\leq\max_{-m_{n}\leq j\leq0}
 \norm{Q_{2K}(\zeta_{\Delta_n})_j} <\varepsilon^{\frac{1}{2}}.
\end{equation}
On the other hand,
by \eqref{eq:uniform-attractor-bound} and
\eqref{eq:discrete-equicontinuity}, the sequence
$(P_{2K}J_{\Delta_n}
\zeta_{\Delta_n})_{n\geq1}$
is uniformly bounded and equicontinuous. Since
\(P_{2K}\ell^2\) is finite dimensional, the
Arzel\`a--Ascoli theorem implies that this sequence is
relatively compact in \(\Hh\).
Combining the relative compactness of $(P_{2K}J_{\Delta_n}
\zeta_{\Delta_n})_{n\geq1}$
with the uniform tail estimate \eqref{eq:discrete-equicontinuity-1}, a standard diagonal
argument shows that there exists a subsequence
\((n_\ell)_{\ell\geq1}\) such that
$(J_{\Delta_{n_\ell}}
\zeta_{\Delta_{n_\ell}}
)_{\ell\geq1}$
converges in \(\Hh\).
Therefore every sequence in
\(\bigcup_{0<\Delta\leq\Delta_0}
J_\Delta\mathcal A_\Delta(\omega)\) has a convergent
subsequence, which proves the assertion.
\end{proof}

\begin{theorem}
\label{thm5.16}
Let \(\mathcal A\) be the pullback attractor given by
Theorem~\ref{thm:attractor}, and, for every sufficiently
small admissible step size \(\Delta\), let
\(\mathcal A_\Delta\) be the discrete pullback attractor
given by Theorem~\ref{thm5.14}. Then, for every
\(\omega\in\Omega_0\),
\begin{equation}\label{eq:euler-usc}
 \dist_{\Hh}\bigl(
 J_\Delta\mathcal A_\Delta(\omega),
 \mathcal A(\omega)
 \bigr)
 \rightarrow 0
 \qquad\text{as }\Delta=\frac{\tau}{m}\rightarrow0.
\end{equation}
\end{theorem}

\begin{proof}
Fix \(\omega\in\Omega_0\) and suppose that
\eqref{eq:euler-usc} fails. Then there exist
\(\Delta_n\downarrow0\),
\(\eta_n\in\mathcal A_{\Delta_n}(\omega)\) and
\(\varepsilon_0>0\) such that
\begin{equation}\label{eq:euler-usc-contradiction}
 \inf_{\chi\in\mathcal A(\omega)}
 \norm{J_{\Delta_n}\eta_n-\chi}_{\Hh}
 \geq\varepsilon_0.
\end{equation}
By Lemma~\ref{lem5.15}, after passing to a subsequence
and relabeling, there exists \(\eta\in\Hh\) such that
\begin{equation}\label{eq:an-limit}
 J_{\Delta_n}\eta_n\longrightarrow\eta
 \qquad\text{in }\Hh.
\end{equation}
Let \(m_n=\tau/\Delta_n\). By the strict invariance of
\(\mathcal A_{\Delta_n}\), we choose a
complete discrete orbit
\(\zeta_n:\mathbb Z\to\Hh_{\Delta_n}\) satisfying
$\zeta_n(0)=\eta_n$ and
$\zeta_n(k)\in
 \mathcal
 A_{\Delta_n}(\theta_{k\Delta_n}\omega)$,
$k\in\mathbb Z$.
For \(M\in\mathbb N\), set \(T_M=M\tau\) and
$\eta_{n,M}:=\zeta_n(-Mm_n)
 \in
 \mathcal A_{\Delta_n}(\theta_{-T_M}\omega)$.
Applying Lemma~\ref{lem5.15} at each
\(\theta_{-T_M}\omega\) and using a diagonal argument, we
find that there exist a common subsequence and \(\eta_M\in\Hh\) such that
\begin{equation}\label{eq:preimage-limits}
 J_{\Delta_n}\eta_{n,M}\longrightarrow\eta_M
 \qquad\text{in }\Hh
\end{equation}
for every fixed \(M\).
Note that
\(\eta_{n,0}=\eta_n\) and \(\eta_0=\eta\).
For each fixed \(M\), the set
$\mathcal K_M
:=
\{J_{\Delta_n}\eta_{n,M}:n\in\mathbb N\}
\cup\{\eta_M\}$
is compact in \(\Hh\). By
\(P_{\Delta_n}J_{\Delta_n}\eta_{n,M}=\eta_{n,M}\),
Proposition~\ref{prop5.7}, \eqref{eq:preimage-limits} and
the continuity of \(\Phi\), we obtain that
\[
J_{\Delta_n}\eta_n
=
J_{\Delta_n}\Phi_{\Delta_n}
\bigl(Mm_n,\theta_{-T_M}\omega\bigr)\eta_{n,M}
\longrightarrow
\Phi\bigl(T_M,\theta_{-T_M}\omega\bigr)\eta_M.
\]
Comparing this with \eqref{eq:an-limit} yields
\begin{equation}\label{eq:limit-preimage}
\eta
=
\Phi\bigl(T_M,\theta_{-T_M}\omega\bigr)\eta_M.
\end{equation}
Moreover, in view of $\eta_{n,M}\in
 \mathcal A_{\Delta_n}(\theta_{-T_M}\omega)
 \subseteq
 \widehat B_{\Delta_n}(\theta_{-T_M}\omega)$, \eqref{eq:uniform-discrete-radius} and
\eqref{eq:preimage-limits}, we have
$\norm{\eta_M}_{\Hh}^2
 \leq R_*(\theta_{-T_M}\omega)$.
Define the random ball
$D_*(\omega)
:=
\left\{
\xi\in\Hh:
\|\xi\|_{\Hh}^2\leq R_*(\omega)
\right\}$.
Since \(R_*\) is tempered, it holds that \(D_*\in\mathcal D(\Hh)\) and
\(\eta_M\in D_*(\theta_{-T_M}\omega)\). Therefore, it follows from
\eqref{eq:limit-preimage} and the pullback-attracting
property of \(\mathcal A\) that
\begin{align*}
\dist_{\Hh}\bigl(\{\eta\},\mathcal A(\omega)\bigr)
&\leq
\dist_{\Hh}\Bigl(
\Phi(T_M,\theta_{-T_M}\omega)
D_*(\theta_{-T_M}\omega),
\mathcal A(\omega)
\Bigr)
\longrightarrow0
\end{align*}
as \(M\to\infty\). Since \(\mathcal A(\omega)\) is closed,
we conclude that \(\eta\in\mathcal A(\omega)\).
Taking account of
\eqref{eq:an-limit}, this contradicts
\eqref{eq:euler-usc-contradiction} and completes the proof.
\end{proof}

Let \(P_N\) and \(Q_N\) be the projections introduced in
\eqref{eq4.32-0}, and let
$\mathbf x^N
:=
\bigl(P_Nx,(P_N\otimes P_N)\mathbb X\bigr)$
be the projected rough path. Consider the spatially
truncated equation
\begin{align}\label{eq-truncated}
du_N(t)
={}&
\Bigl[
-P_N\mathcal LP_Nu_N(t)
-\lambda u_N(t)
+P_NF\bigl(u_N(t),u_N(t-\tau)\bigr)
\Bigr]dt
+
P_Nh\bigl(u_N(t)\bigr)P_N\,d\mathbf x_t^N,
\end{align}
with initial history
$u_N(s)=P_N\xi(s)$ for
$s\in[-\tau,0]$ and \(\xi\in\Hh\).
Since the projected coefficients satisfy the original
Lipschitz and regularity bounds uniformly in \(N\), the
arguments used in Theorems~\ref{thm3.8}--\ref{thm3.9} show that
\eqref{eq-truncated} admits a unique global controlled
solution \(u_N\) and generates a continuous cocycle \(\Phi_N\) on
\(\Hh_N:=C([-\tau,0];P_NU)\).

\begin{proposition}\label{prop5.16}
Let \(T>0\), \(\omega\in\Omega_0\), and let
\(S\subset\Hh\) be compact. Then
\begin{align*}
 \sup_{\xi\in S}\sup_{0\leq t\leq T}
 \norm{
 \Phi_N(t,\omega)P_N\xi-\Phi(t,\omega)\xi
 }_{\Hh}
 \longrightarrow0
 \qquad\text{as }N\to\infty.
\end{align*}
\end{proposition}

\begin{proof}
Let \(u_N\) be the solution of \eqref{eq-truncated}
with initial history \(P_N\xi\), and let \(u\) be the
solution of \eqref{eq:vector} with initial history \(\xi\).
Set
$e_N:=u_N-P_Nu$.
In view of \(u_N=P_Nu_N\), we have
\begin{align*}
 de_N
 &=\bigl[-P_N\mathcal Le_N-\lambda e_N
 +P_N(F(u_N,u_N(\cdot-\tau))-F(u,u(\cdot-\tau)))
 +P_N\mathcal LQ_Nu\bigr]dt\notag\\
 &\quad+P_N(h(u_N)-h(u))\,d\mathbf x,
\end{align*}
where we have used the diagonal structure of \(h\) and the
definition of \(\mathbf x^N\) to rewrite the projected
rough integral with respect to \(\mathbf x\).
Note that
$\norm{P_N\mathcal LQ_Nu}
\leq
\norm{\mathcal L}\norm{Q_Nu}$.
By the coordinatewise structure of \(F\) and
the diagonal structure of \(h\), it holds that
$P_NF(u,v)=P_NF(P_Nu,P_Nv)$ and
$P_Nh(u)=P_Nh(P_Nu)$. Therefore, arguing as in the proof of
Lemma~\ref{lem4.2} on a finite greedy partition of
\([0,T]\), and then iterating the resulting local
estimates, we obtain that
\begin{equation}\label{eq:trunc-error-estimate}
 \norm{e_N}_{\infty,[0,T]}
 \leq
 C_{S,T}(\omega)
 \left(
 \norm{Q_N\xi}_{\Hh}
 +\int_0^T\norm{Q_Nu(s)}\,ds
 \right),
\end{equation}
where \(C_{S,T}(\omega)\) is independent of \(N\).
The compactness of the solution image
$\{u_t(\omega,\xi):0\leq t\leq T,\ \xi\in S\}\subset\Hh$ implies
\[
\sup_{\xi\in S}\sup_{0\leq t\leq T}\norm{Q_Nu_t(\omega,\xi)}_{\Hh}
 \rightarrow0.
\]
Combining this with \eqref{eq:trunc-error-estimate} and
$\norm{u_N-u}\leq\norm{e_N}+\norm{Q_Nu}$, we obtain the assertion.
\end{proof}

For each \(N\geq1\), the cocycle generated by the
projected equation \eqref{eq-truncated} possesses a
pullback attractor \(\mathcal A_N\) on \(\Hh_N\), as
follows by repeating the proof of
Theorem~\ref{thm:attractor}. Similarly, for every
\(N\geq1\) and every sufficiently small admissible step
size \(\Delta\), the compensated Euler scheme associated
with \eqref{eq-truncated} possesses a pullback attractor
\(\mathcal A_{N,\Delta}\), by the same argument as in the
proof of Theorem~\ref{thm5.14}.

\begin{lemma}\label{lem5.17}
For each \(N\geq1\), let \(\mathcal A_N\) denote the
pullback attractor of the projected equation \eqref{eq-truncated}.
Let \(\Delta_0>0\) be sufficiently small so that, for
every \(N\geq1\) and every admissible step size
\(\Delta=\tau/m\in(0,\Delta_0]\), the corresponding
compensated Euler scheme of \eqref{eq-truncated} possesses a pullback attractor
\(\mathcal A_{N,\Delta}\). Then, for every
\(\omega\in\Omega_0\), both sets
\[
 \bigcup_{N\geq1}\mathcal A_N(\omega)
 \quad\text{and}\quad
 \bigcup_{N\geq1,\;
 0<\Delta\leq\Delta_0}
 J_\Delta\mathcal A_{N,\Delta}(\omega)
\]
are relatively compact in \(\Hh\).
\end{lemma}

\begin{proof}
Since the projections \(P_N\) are contractions and commute with the spatial cutoff operators, the projected coefficients satisfy the original
bounds uniformly in \(N\), and the projected
rough-path tails are bounded by the original ones,
we find that
all the absorbing, lattice-tail, and
history-equicontinuity estimates used in the proof of
Lemma~\ref{lem5.15} remain valid
for the projected equations \eqref{eq-truncated} and their
compensated Euler schemes, with constants independent
of \(N\) and,  in the
discrete case, of \(\Delta\), over a common step-size
range independent of \(N\).

Moreover,
the relative compactness of
\(\bigcup_{N\geq1}\mathcal A_N(\omega)\) in \(\Hh\)
follows from the uniform lattice-tail estimate
and the Arzel\`a--Ascoli theorem applied to
finite-dimensional projections.
The relative compactness of
\(\bigcup_{N\geq1,\;0<\Delta\leq\Delta_0}
J_\Delta\mathcal A_{N,\Delta}(\omega)\)
in \(\Hh\) follows by repeating the sequence argument
of Lemma~\ref{lem5.15}, using the estimates uniform in \(N\).
We omit the details.
\end{proof}

\begin{theorem}\label{thm5.18}
Let \(\mathcal A\) be the pullback attractor given by
Theorem~\ref{thm:attractor}, and, for each \(N\geq1\),
let \(\mathcal A_N\) denote the pullback attractor of the projected equation \eqref{eq-truncated} on \(\Hh_N\). Then for every \(\omega\in\Omega_0\), we have
\begin{equation}\label{eq:trunc-usc}
 \dist_{\Hh}\bigl(
 \mathcal A_N(\omega),\mathcal A(\omega)
 \bigr)
 \longrightarrow0
 \qquad\text{as }N\to\infty.
\end{equation}
\end{theorem}

\begin{proof}
Fix \(\omega\in\Omega_0\) and suppose, to the contrary,
that \eqref{eq:trunc-usc} fails. Then there exist
\(\varepsilon_0>0\), \(N_j\to\infty\), and
\(\eta_j\in\mathcal A_{N_j}(\omega)\) such that
$\inf_{\chi\in\mathcal A(\omega)}
 \norm{\eta_j-\chi}_{\Hh}
 \geq\varepsilon_0$.
By Lemma~\ref{lem5.17}, after passing to a subsequence,
there exists \(\eta\in\Hh\) such that
$\eta_j\rightarrow\eta$
in $\Hh$.
Repeating the absorbing-set argument in the proof of
Lemma~\ref{lem4.7} for the projected equations \eqref{eq-truncated}, with
constants independent of \(N\), we obtain a tempered
random variable \(\tilde R\) such that
\[
\sup_{N\geq1}\sup_{\xi\in\mathcal A_N(\omega)}
\norm{\xi}_{\Hh}^{2}\leq\tilde R(\omega).
\]
Set
\(D_*(\omega):=
\{\xi\in\Hh:\norm{\xi}_{\Hh}^{2}\leq\tilde R(\omega)\}\).
By the strict invariance of each \(\mathcal A_{N_j}\)
under \(\Phi_{N_j}\),
Lemma~\ref{lem5.17}, Proposition~\ref{prop5.16}, and a
diagonal argument as in the proof of
Theorem~\ref{thm5.16}, for each \(M\geq1\) there exists
\(\xi_M\in D_*(\theta_{-M\tau}\omega)\) such that
$\eta=\Phi(M\tau,\theta_{-M
\tau}\omega)\xi_M$.
The pullback-attracting property and closedness of
\(\mathcal A(\omega)\) imply that
\(\eta\in\mathcal A(\omega)\).
Consequently,
\[
\varepsilon_0
\leq\inf_{\chi\in\mathcal A(\omega)}
\norm{\eta_j-\chi}_{\Hh}
\leq\norm{\eta_j-\eta}_{\Hh}
\rightarrow0.
\]
This contradicts \(\varepsilon_0>0\) and proves
\eqref{eq:trunc-usc}.
\end{proof}

\begin{corollary}\label{cor:joint-usc}
Let \(\mathcal A\) be the pullback attractor given by
Theorem~\ref{thm:attractor}. For each \(N\geq1\) and
each sufficiently small admissible step size \(\Delta\),
let \(\mathcal A_{N,\Delta}\) denote the pullback
attractor of the compensated Euler scheme associated
with the projected equation \eqref{eq-truncated}. Then, for every \(\omega\in\Omega_0\),
\[\dist_{\Hh}\bigl(
 J_\Delta\mathcal A_{N,\Delta}(\omega),
 \mathcal A(\omega)
 \bigr)
 \longrightarrow0
 \qquad\text{as }
 N\to\infty,\quad
 \Delta=\frac{\tau}{m}\to0.\]
\end{corollary}

\begin{proof}

Repeating the proof of Proposition~\ref{prop5.7},
with constants and step-size thresholds independent
of \(N\), we obtain that
for every compact \(S\subset\Hh\)
and every \(T>0\),
\[
 \sup_{N\geq1}\sup_{\xi\in S}
 \max_{0\leq k\Delta\leq T}
 \left\|
 J_\Delta\Phi_{N,\Delta}(k,\omega)P_\Delta P_N\xi
 -
 \Phi_N(k\Delta,\omega)P_N\xi
 \right\|_{\Hh}
 \longrightarrow0
 \qquad\text{as }\Delta\to0.
\]
Combining this convergence with
Proposition~\ref{prop5.16}, Lemma~\ref{lem5.17},
and the common tempered absorbing bounds,
we obtain the assertion by arguing as in the proofs
of Theorems~\ref{thm5.16} and \ref{thm5.18}. Here we omit the details.
\end{proof}

\medskip
\noindent\textbf{Data availability.} No data were generated or analysed in this study.

\medskip
\noindent\textbf{Conflict of interest.} The authors declare that they have no conflict of interest.

\medskip
\noindent\textbf{Declaration of AI assistance.}
During the preparation of this manuscript, the authors
used AI to improve the language and clarity
of expression. The authors reviewed and edited the
AI-assisted content and take full responsibility for
the final manuscript.

\end{document}